\documentclass[12pt,reqno]{amsart}

\usepackage{amsmath,amsthm,amssymb}
\usepackage{epsf}
\usepackage{amssymb}
\usepackage{graphics}
\usepackage{stackrel}
\usepackage[usenames,dvipsnames]{xcolor}
\usepackage{hyperref}

\usepackage{pdfsync}
\newcommand{\sh}{h} 
\newcommand{\wt}{k} 
\newcommand{\cusp}{\mathfrak{a}}

\newcommand{\beq}{\begin{equation}}
\newcommand{\eeq}{\end{equation}}
\newcommand{\mc}{\mathcal}

\newcommand{\hf}{\frac 12}
\newcommand{\thf}{\tfrac 12}
\newcommand{\qtr}{\frac 14}
\newcommand{\kt}{\frac k2}
\newcommand{\tkt}{\tfrac k2}
\newcommand{\lt}{\left}

\newcommand{\rt}{\right}
\newcommand{\vho}{\varrho}

\newtheorem{theorem}{Theorem}[section]

\newtheorem{lemma}[theorem]{Lemma}

\newtheorem{proposition}[theorem]{Proposition}

\numberwithin{equation}{section}

\theoremstyle{definition}
\newtheorem{remark}[theorem]{Remark}

\newcommand\ben{\begin{enumerate}}
\newcommand\een{\end{enumerate}}

\newcommand{\nc}{\newcommand}
\nc{\nen}{\newenvironment}

\nen{thma}[1]{\label{#1}{\bf Theorem A.} \em}{}
\nen{thmb}[1]{\label{#1}{\bf Theorem B.\ } \em}{}

\newcommand\be{\begin{equation}}
\newcommand\ee{\end{equation}}
\newcommand\benn{\begin{equation*}}
\newcommand\eenn{\end{equation*}}
\newcommand\bea{\begin{eqnarray}}
\newcommand\eea{\end{eqnarray}}
\newcommand\beann{\begin{eqnarray*}}
\newcommand\eeann{\end{eqnarray*}}
\nc{\br}{\mathbb R}

\newcommand{\boldH}{\mathbb{H}}
\newcommand{\R}{\mathbb{R}}
\newcommand{\C}{\mathbb{C}}
\newcommand{\Z}{\mathbb{Z}}
\newcommand{\Q}{\mathbb{Q}}
\newcommand{\N}{\mathbb{N}}

 \nc{\lm}{\lambda}

\newcommand{\ds}{\mathrm{d}s}
\newcommand{\dw}{\mathrm{d}w}

\newcommand{\ol}{\overline}

\nc{\ba}{\mathbb A}
\newcommand{\A}{\mathbb{A}}

\newcommand{\cI}{\mathcal{I}}

\newcommand{\Or}{\ensuremath{{\mathcal O}}}

\newcommand{\gd}{\delta}     
\newcommand{\G}{\Gamma}      
\newcommand{\Gi}{\Gamma_{\infty}}      
\newcommand{\g}{\gamma}      
\newcommand{\gw}{\omega}      
\newcommand{\inv}{^{-1}}
\nc{\SLR}{SL_2(\br)} 
\nc{\GLR}{GL_2(\br)} 
\nc{\PGLR}{PGL_2(\br)}
\nc{\PSLR}{PSL_2(\br)}
\nc{\PSLZ}{PSL_2(\bz)}
\nc{\SLC}{SL(2,\bc)}
\nc{\PGL}{PGL_2}
\nc{\bt}{\beta} 
\nc{\om}{\omega} 
\nc{\dl}{\delta}
 \nc{\Dl}{\Delta}

\def\Res{\qopname\relax m{Res}}

\newcommand{\vol}{\mathcal{V}}

\renewcommand{\mod}{\operatorname{mod}}

\renewcommand{\Re}{{\mathfrak{Re}}}
\renewcommand{\Im}{{\mathfrak{Im}}}

\newcommand{\<}{\left\langle}
\renewcommand{\>}{\right\rangle}

\newcommand{\bk}{\backslash}

\nc{\bz}{\mathbb Z} \nc{\bc}{\mathbb C}
\nc{\Hom}{\mathrm{Hom}}

\nc{\8}{\infty}
\nc{\uH}{\mathbb H}

\nc{\cal}{\mathcal}  
   \nc{\CH}{\cal H}   
 \nc{\al}{\alpha}

 \nc{\haf}{\frac{1}{2}}
\nc{\eps}{\epsilon}
\AtBeginDocument{%
   \def\MR#1{}
}

\begin{document}

\title{Multiple Dirichlet series and Shifted Convolutions}

\author{Jeff Hoffstein and Thomas A. Hulse} \email{jhoff@math.brown.edu,tahulse@colby.edu,reznikov@math.biu.ac.il}
\address{Department of Mathematics,
Brown University, Providence, RI, $02912$ \\ 
Department of Mathematics and Statistics, Colby College, Waterville, ME, $04901$ \\
 Bar Ilan University, Ramat-Gan, Israel}
 \date{\today}
 \dedicatory{Dedicated to Winnie Li with admiration.}
\maketitle
\begin{abstract}
We define, and obtain the meromorphic continuation of, shifted Rankin-Selberg convolutions in one and two variables. 
As sample applications, this continuation is used to obtain estimates for single and double shifted sums and
a Burgess-type bound for $L$-series associated to modular forms of arbitrary central character. Further applications are furnished by subsequent works by the authors and their colleagues. 
\end{abstract}

\section{Introduction}
In 1939 and 1940 in \cite{Ra} and \cite{S1}, Rankin and Selberg independently introduced what is now called a Rankin-Selberg convolution.    Very roughly, if 
$$
L(s,f) = \sum_{m \ge 1}\frac{a(m)}{m^{s+\frac{k-1}{2}}} \,\, \, \text{and} \,\, \, L(s,g) = \sum_{m \ge 1}\frac{b(m)}{m^{s+\frac{k-1}{2}}} 
$$
are two $L$-series corresponding to modular forms, each convergent when $\Re \, s >1$, with a functional equation in $s$, and a meromorphic continuation to $\C$, the Rankin-Selberg convolution provides a functional equation and a meromorphic continuation to $\C$ of the new $L$-series
$$
L(s, f \otimes g)= \sum_{m \ge 1}\frac{a(m) \bar b(m)}{m^{s+k-1}}.
$$
This idea had a profound impact on the field and was responsible for what was, at that time, the best progress toward the Ramanujan conjecture.   

In 1965, in \cite{Se}, a paper summing up progress in modular forms, Selberg introduced the notion of a shifted convolution.  If $h$ is a positive integer, his shifted convolution was the series
\be\label{S}
\sum_{m \ge 1} 
{a(m) \bar b(m+h)
\over
(2m+h)^{s+k-1}}.
\ee
This series converges locally normally for $\Re \, s >1$ and Selberg mentioned that the meromorphic continuation of this could be obtained. He went on to remark: ``We cannot  make much use of this function at the moment...". 

Since then it has been generally recognized that the analytic properties of series such as \eqref{S} play a very important role in estimating the size of certain shifted sums, such as
$$
\sum_{m <x}a(m) \bar b(m+h) \lt(m(m+h)\rt)^{\frac{1-k}{2}},
$$
and that these in turn play an important role in progress toward the Lindel\"of Hypothesis, that is, the bounding from above of automorphic $L$-series at points inside the critical strip. Good, in \cite{Gd1} and \cite{Gd2}, was the first to use the spectral expansion of \eqref{S} to demonstrate a polynomial bound in $s$ on vertical strips and provide an estimate of the second moment of the corresponding $L$-series. See \cite{LLY2} for a wider survey of this subject.


One purpose of this paper is to suggest that the series $\eqref{S}$ might not tell the whole story.   In particular, 
we will show that the meromorphic continuation to $\C$ of the  series
\be\label{oursum}
D(s;h)=\sum_{m \ge 1} 
{a(m+h) \bar b(m)
\over
 m^{s+k-1}},
\ee
and generalizations of it, can be a valuable source of additional information.
  Establishing this meromorphic continuation proves to be a surprisingly difficult task and is the subject of Sections 2--5.  Once done, we extend this notion to multiple shifted Dirichlet series, in which a sum is taken over the shifts $h$ as well, and show that these can also be meromorphically continued.  
  
  One interesting aspect of these shifted multiple Dirichlet series is that, just like their $L$-series and multiple Dirichlet series counterparts, they have an analog of a critical strip  where all the deep number theoretic information is buried.  Once the continuations in both complex variables have been established, some strong results on shifted sums, in which there is significant cancellation uniformly in the length and the shift parameter, are obtained (Theorem~\ref{thm:sum}).  From these estimates subconvexity bounds (Theorem~\ref{thm:sub}) 
 follow with little difficulty.      The main purpose of this work is the introduction of 
the series \eqref{oursum} and its generalizations, and of the techniques leading to Theorem~\ref{thm:sum}.    The application to subconvexity is included as a ``proof of concept" to demonstrate the advantages of this approach.
\subsection{A subconvexity bound} 
The subconvexity problem for $GL(2)$ can be posed in several different aspects and here we will focus on the conductor of a twisting character in the context of holomorphic cusp forms.    Specifically, let 
 $f$ be a primitive holomorphic cusp form of even weight $\wt$ and character $\chi_1$  for 
$\G_0(N_0)$, with $N_0$ square free, having Fourier expansion
$$
f(z)=\sum_{m\ge1} A(m)m^{(k-1)/2} e^{2\pi i m z}.
$$

For $\chi$ a Dirichlet character mod $Q$, we write the twisted $L$-series as
$$
L(s,f,\chi) = \sum_{m\ge1}\frac{A(m)\chi(m)}{m^s}.
$$
This possesses a functional equation as $s \rightarrow 1-s$, and the convexity bound for the central value which follows is 
$$
L(\thf,f,\chi) \ll_f Q^{\hf + \epsilon}.
$$
The Lindel\"of Hypothesis in the $Q$ aspect states that the exponent can be reduced from 
$\hf + \epsilon$ to $\epsilon$, and any bound that reduces $\hf$ to $\hf - \delta$, for some 
$\delta >0$ is called a convexity breaking bound.  The corresponding convexity bound for $GL(1)$ is $1/4$.

In a groundbreaking series of papers, Duke, Friedlander, and Iwaniec developed techniques that broke the convexity bound for various families of $L$-functions \cite{DFI1,DFI2,DFI3}.    Another celebrated result, in the context  of $GL(1)$ and a Dirichlet $L$-series, was obtained by Conrey and Iwaniec, \cite{CI}, who showed that
$$
L(\thf, \chi) \ll Q^{1/6 + \epsilon},
$$
where the character $\chi$ is real and primitive with odd conductor $Q$.   This is analogous to the Weyl bound for the Riemann zeta function, and it remains the best known result for a $GL(1)$ $L$-series where $Q$ is varied over odd integers. 

Progress has been made in $GL(2)$ by a number of different authors, with the most recent and strongest results in the twisting parameter being obtained by Blomer and Harcos.   In \cite{BH}, they demonstrate that for $f$ a primitive holomorphic or Maass cusp form of archimedean parameter $\mu$, level $N$ and trivial nebentypus,
\begin{equation*}
L(\thf, f,\chi)  \ll_{\epsilon}\left( |\mu|^{\hf}N^{1/4}q^{3/8}+|\mu|N^{\hf}(N,q)^{1/4}q^{1/4}\right)(|\mu|Nq)^\epsilon
\end{equation*}
when $f$ is holomorphic, and 
\begin{equation*}
L(\thf, f,\chi)  \ll_{\epsilon}\left( (1+|\mu|)^{3}N^{1/4}q^{3/8}+(1+|\mu|)^{7/2}N^{\hf}(N,q)^{1/4}q^{1/4}\right)\\ (1+|\mu|)Nq)^\epsilon
\end{equation*}
in the Maass form case.

In \cite{BH10}, Blomer and Harcos work over an arbitrary totally real field $K$.  In this context they let $\pi$  be an irreducible cuspidal representation of $GL(2,K)\bk GL(2,\A_K)$,
with unitary central character and $\chi$  a Hecke character with conductor of norm $Q$.  They also let $\theta$, with $0 \le \theta \le \hf$, be the best progress toward the Ramanujan-Petersson conjecture.  (Here $\theta = 7/64$ is the best known bound.  See  \cite{KS} for $\Q$ and  \cite{BB11} for an arbitrary number field.  The Ramanujan-Petersson conjecture predicts $\theta = 0$.) They obtain the bound
$$
L(\thf, \pi \otimes\chi) \ll_{\pi, \chi_\infty,K,\epsilon} Q^{\frac12 -\frac18(1-2\theta) + \epsilon}.
$$

Here, as our point is to demonstrate a proof of concept for a new method, we work in a more restricted setting than in the work of Blomer and Harcos, restricting ourselves to $\Q$, and the case of $f$ holomorphic of square free level with even weight.   
 In Section~\ref{sect:subcon} we prove the following
\begin{theorem}\label{thm:sub}
Let $f$ be a primitive  holomorphic cusp form of even weight $k$, character $\chi_1$ and fixed square free level $N_0$, and let $\chi$ be a Dirichlet character modulo $Q$.  Then
$$
L(\thf,f,\chi) \ll_{f,\epsilon} Q^{3/8+\theta/4+\epsilon}.
$$
\end{theorem}

\begin{remark}
Since an earlier version of this paper was circulated, Blomer and Harcos have succeeded in removing the dependence on nebentypus and the Ramanujan conjectures in an addendum to their paper \cite{BH}.  We, in the opposite direction, have discovered that we are unable to remove it. Thus we note that earlier versions of this paper contain the mistaken claim that the bound we are able to obtain is independent of progress toward the Ramanujan Conjecture.
\end{remark}

The techniques of this paper have been further generalized to produce other subconvexity results. In \cite{arxiv5}, Kuan obtained a hybrid bound for $L(\thf+it,f,\chi)$ in the $t$ and $Q$-aspects. In \cite{arxiv4}, K\i{}ral was able to produce a completely analogous bound to our Theorem \ref{thm:sub} in the case of twisted $L$-functions for half-integral weight holomorphic cusp forms. In \cite{arxiv2}, the first author and Min Lee were able to obtain a Weyl-type estimate for the Rankin-Selberg convolution of a holomorphic cusp form and a Maass form in the eigenvalue aspect. We anticipate that more varied subconvexity results can be achieved via these shifted sums.

\subsection{A bit more on shifted convolutions}
As mentioned above, the main purpose of this paper
is to present a new method for approaching problems related to shifted sums, subconvexity and moments.  In particular, we will introduce a new class of shifted single and multiple Dirichlet series and show how their full meromorphic continuations can be obtained and how they can be applied to this class of problems.

Let $\ell_1,\ell_2$ be two positive integers, and fix $h \ge 1$.   Suppose that $f,g$ are modular forms, of even weight $k$ and square-free level $N_0$, which have Fourier expansions
\begin{align}
f(z) &= \sum_{m \ge 1}a(m)e^{2 \pi i m z} = \sum_{m \ge 1}A(m)m^{(k-1)/2}e^{2 \pi i m z}, \notag \\
g(z) &= \sum_{m \ge 1}b(m)e^{2 \pi i m z} = \sum_{m \ge 1}B(m)m^{(k-1)/2}e^{2 \pi i m z}. \label{obvious}
\end{align}

The constraint that $k$ is even arises from our use of Andre Reznikov's result in the appendix of this work.

The fundamental single Dirichlet series that we will investigate is 
$$
D(s;h, \ell_1,\ell_2)=\sum_{m_2 \ge 1, \atop m_1\ell_1= m_2\ell_2 + {h}} 
{a(m_1) \bar b(m_2)
\over
 (m_2\ell_2)^{s+\wt-1}}.
$$
The $\ell_1,\ell_2$ parameters are included for use in amplification arguments targeted at upper bounds for individual $L$-series.   See, for example, 
\cite{B}, \cite{BH} and \cite{Sa2}.   When $\ell_1 = \ell_2 =1$ this specializes to 
\be\label{Dfirst}
D(s;h)=\sum_{m \ge 1} 
{a(m+h) \bar b(m)
\over
 m^{s+\wt-1}}.
\ee 
Selberg mentions that the meromorphic continuation of his series can be obtained by combining ``the technique of Rankin" with a certain Poincare series:
$$
 \sum_{\g\in\Gi\bk\G}(\Im (\g z))^s e^{2\pi i  h \g z }.
$$
Goldfeld, \cite{Go1}, wrote down the series \eqref{Dfirst}, rather than \eqref{S} and remarked that  Selberg had briefly indicated how to obtain its meromorphic continuation.   Goldfeld approaches that continuation in a different way, using instead the Poincare series
$$
\sum_{\g\in\Gi\bk\G}(\Im \g z)^{\hf}I_{s-\hf}(2 \pi |h|\Im \, \g z)e^{2 \pi i h \Re \, \g z}.
$$
The function $I_v(y)$ is the modified Bessel function of the first kind, and grows exponentially.
Goldfeld achieves the meromorphic continuation of \eqref{Dfirst} to $\Re \, s >\hf$ via the analytic properties of an inner product of his Poincare series with $f$ and $g$.   However the Dirichlet series that comes out of this method is more closely related to \eqref{S} than to \eqref{Dfirst}, and a hypergeometric series needs to be introduced to carry over the information.  The presence of this 
hypergeometric series seriously impacts efforts to continue past $\Re \, s >\hf$.

A fundamental difficulty with the series \eqref{S} seems to be the tying together of the $h$ and the $m$ variables in the denominator.   One can see, for example, that if one is interested in estimating sums of coefficients of a given length $x$, that this sum becomes seriously unbalanced when $h$ greatly exceeds $x$.   It seems to be the case that the best way to ``untie" the $m$ and $h$, so as to avoid unwanted alternative coupling from integral expansions and smoothing functions, is to reverse the sign of the $h$ in the exponent of the Poincare series obtaining, for example:
$$
 P_h(z;s)=\sum_{\g\in\Gi\bk\G}(\Im (\g z))^s e^{-2\pi i \, h \,\g z }.
$$
When one does this, the series $D(s;h)$ comes naturally out of a formal computation of a triple product.  Unfortunately, the series $P_h(z;s)$ grows exponentially in $y$, making it difficult to turn the formal computation into a genuine one.   Goldfeld faced a similar difficulty in \cite{Go1} and dealt with it by  truncating the fundamental domain.    In the present case the difficulties seem to be more fundamental and it appears to be necessary to work with an approximation to $ P_h(z;s)$ by two extra parameters.   This is done in Sections 2-5. 
 It is shown there that the general series $D(s;h, \ell_1,\ell_2)$, which converges for $\Re \, s >1$, has a meromorphic continuation to all of $\C$, with rightmost poles at the points $\hf \pm it_j$, where the eigenvalues of the Laplacian are $\frac{1}{4}+t_j^2$. To accomplish this, a spectral expansion for 
$D(s;h, \ell_1,\ell_2)$ is found which is locally normally convergent for $\Re \, s< \hf -\kt$.   The gap between $1$ and $\hf-\kt$ is bridged by an approximation to $D(s;h)$ which is locally normally convergent everywhere, and whose limit approaches $D(s;h)$ as the two parameters tend to infinity and zero respectively. 
The consequent meromorphic continuation of $D(s;h)$ is the content of Proposition~\ref{prop:D0}.

Most applications of shifted sums call for a sum over the shift $h$.   Because of this it is natural to consider the collective behavior of $D(s;h, \ell_1,\ell_2)$ as $h$ varies.   This is accomplished first in the meromorphic continuation to $\{(s,w) \in \C^2| \Re \, w >1\}$ of the shifted double Dirichlet series:
$$
Z(s,w) =  \sum_{h \ge 1}\frac{ D(s;h)}{h^{w+(k-1)/2}} = \sum_{h \ge 1,  m_2 \ge 1 \atop m_1\ell_1= m_2\ell_2 + {h}} 
{a(m_1) \bar b(m_2)
\over
 (m_2)^{s+\wt-1}h^{w + (k-1)/2}},
 $$
 and more generally of the series
$$
Z_Q(s,w)= \sum_{ h, m_2 \ge 1 \atop m_1\ell_1= m_2\ell_2 + {hQ}} 
{a(m_1) \bar b(m_2)
\over
 (\ell_2m_2)^{s+\wt-1}(hQ)^{w + (k-1)/2}}.
$$

This series has poles in $s$ at the same locations as $D(s;h)$,
 but it also has additional emergent polar lines at $w+2s+\kt-\frac{5}{2} \in \Z_{\leq 0}.$     In Section~\ref{sect:Zsums} we prove the following theorem, which is the underpinning of the application.
\begin{theorem}\label{thm:sum}
Fix $N_0,\ell_1,\ell_2,h\in\N$,  with $N_0,\ell_1,\ell_2$ square free and $(N_0,\ell_1\ell_2)=1$.
Let 
$$
S_Q(x,y) = \sum_{h, m_2 \ge 1 \atop \ell_1m_1 = \ell_2m_2 + hQ  }
A(m_1)\bar B(m_2) G_1(m_1/y)G_2(m_2/x),
$$
where $G_1(x),G_2(x)$ are smooth, with compact support in the interval $[1,2]$.
Suppose $y \gg x \gg 1$, $\ell_1, \ell_2 \sim L \ll Q^{\hf}$.   Then 
$$
S_Q(x,y) \ll Q^{\theta-\hf + \epsilon}Ly^{1+\epsilon}(y/x)^{ \epsilon}.
$$
\end{theorem}
The coefficients $A(m), B(m)$ are constant on average and the length of this sum is 
$xy/Q$.   Thus the Theorem gives a square root of the length estimate when $y$ and $x$ are of comparable sizes which, if we let $Q =1$, matches the trivial cancellation we expect to get if we treat $A(n)$ and $B(n)$ like random variables and heuristically factor the double sum as the product of two partial sums. The result becomes non-trivial when $Q^{1+\epsilon}\gg y = x$, which is the case which is relevant to this paper. This result could be considerably improved if one had further information about upper bounds for $Z_Q(s,w)$ in the region $\Re \, w \le 1$.  (See Remark~\ref{remark:bottle}.)  Such information is available from \cite{HLee}, but considering the length of the present work, such improvements are deferred to another paper.

Analogous results can be obtained when $f$ and $g$ are replaced with other automorphic forms. Indeed, by considering the shifted convolution sums of Fourier coefficients of a theta function and a half-integral weight Eisenstein series, the second author, K\i{}ral, Kuan, and Li-Mei Lim were able to produce asymptotic counts of binary quadratic forms with square discriminants under certain constraints \cite{arxiv6}.

One interesting aspect of the one and two-variable shifted Dirichlet series discussed is that each has two incarnations: a Dirichlet series that converges locally normally in one range, and a spectral expansion that converges locally normally in a disjoint region.   As one might expect, all the interesting information seems to be hidden in between in a ``critical strip" $\hf-\kt<\Re \, s <1$.    In particular, the point $(1-\kt,\hf)$ lies inside this critical strip, and  the ``correct" bound for $S_Q(x,y) $ would follow from an accurate estimate for  $Z_Q(1-\kt,\hf)$. 

Indeed, by obtaining mean value estimates for $Z_Q(1-\kt,\hf)$ as $Q$ varies over a short interval, the first author and Min Lee were able to obtain an asymptotic formula for the second moment
\[
 \frac{1}{\varphi{(Q)}} \sum_{\chi \mbox{ \tiny mod} Q} L(1/2,f,\chi)\overline{L(1/2,g,\chi)}.
\]
By demonstrating that this result is non-zero, they were able to produce a simultaneous non-vanishing result \cite{HLee}. 

The spectral expansions yield another curious development.   The terms decay polynomially in the spectral parameter $t_j$, rather than exponentially, and the ratio of gamma factors governing this decay seems to act in a manner analogous to the usual $m^{-\sigma}$ indexing of a Dirichlet series (polynomial decay corresponding to the size of $\sigma$.)

This work is an expanded version of part of a previous work \cite{arxiv} by the first author. A systematic error in the meromorphic continuation in the continuous part of the spectrum was observed and corrected by the second author along with strengthening other results in sections \ref{section:M}-\ref{sect:app}, and \ref{sect:Zsums}.

The authors would like to thank Valentin Blomer, David Hansen, Roman Holowinsky, Mehmet K\i{}ral, Alex Kontorovich, Min Lee, Andre Reznikov,  and Matt Young for many very helpful and enlightening conversations. We would also like to thank this paper's referee, who provided many suggestions that helped us to clarify and streamline many of our arguments.

\subsection{An outline of the paper}
In Section~\ref{sect:spectral}, the Poincare series $P_{\sh,Y}(z,s;\delta)$ is introduced and its spectral expansion is computed.   A function $M(s,t,\delta)$ appears in this expansion and,  in Section~\ref{section:M}, the meromorphic continuation of $M(s,t,\delta)$ is obtained.   In Section 4, an approximation to $D(s;h)$, namely $D(s;h,\delta)$, is introduced, and its meromorphic continuation is obtained.   The Appendix is an important ingredient in this continuation.   In Section 5, the limit as $\delta \to 0$ is taken, yielding the meromorphic continuation of the single shifted Dirichlet series $D(s;h)$.  In Section 6, an application to single shifted sums is given, and in Section 7, the double Dirichlet series $Z_Q(s,w)$ is introduced and its meromorphic continuation is obtained.     In Section 8, an application to double shifted sums is given, and in Section 9, this double shifted sum is in turn applied to obtain our subconvexity bound.

\begin{remark}
In this paper any use of the term  ``convergence" shall  refer to local normal convergence.   Also, unless otherwise indicated, the implied constant in a $\ll$, $\gg$ and big-$\Or$ expression will depend upon a particular $\epsilon$ appearing in the expression, a constant $A$, which is defined in Section~\ref{section:M}, and  cusp forms $f,g$, including their level $N_0$, but will be independent of all other variables. Furthermore the notation $x \sim T$ specifically denotes all allowable $x$ found in the interval $[\frac{T}{2},2T]$ and is somewhat different from the usual Landau Notation for asymptotic. 
\end{remark}

\section{Spectral expansions}
\label{sect:spectral}

Fix $N_0,\ell_1,\ell_2,h\in\N$,  with $N_0,\ell_1,\ell_2$ square free, $(N_0,\ell_1\ell_2)=1$ and set $N=N_0\ell_1\ell_2/(\ell_1,\ell_2)$ and  $\G=\G_0(N)$. The requirement that $N_0$ be square-free is imposed to simplify the Fourier expansions of Eisenstein series that appear in spectral expansions.

For $z\in\boldH$ and $s\in\C$, $\Re(s)>1$, consider the Poincare series:
\be\label{Pdef}
P_{\sh}(z;s)
:=
\sum_{\g\in\Gi\bk\G}
(\Im (\g z))^s e^{-2\pi i \, {\sh} \,\g z }
.
\ee
This is the usual Poincare series except that the sign in the exponent is negative rather than positive.
For any fixed $z$ the series converges locally normally where  $\Re(s)>1$ and is invariant under action by $\G$.  It does, however, grow exponentially in $y$, which complicates efforts to determine its spectral properties.   For this reason, for any fixed $Y \gg 1$ and $\delta>0$ we define the modified series $P_{\sh,Y}(z,s;\delta)$ by
\be\label{PYdeltadef}
P_{\sh,Y}(z,s;\delta)
:=
\sum_{\g\in\Gi\bk\G} \psi_Y( \Im \g z)
(\Im (\g z))^s e^{-2\pi i \, {\sh} \,\Re \g z + (2\pi  \, {\sh} \,\Im \g z)(1-\delta) }
,
\ee
where $\psi_Y$ is the characteristic function of the interval $[Y^{-1},Y]$.

The series $P_{\sh,Y}(z,s;\delta)$  consists of finitely many terms and is in $L^2(\G\bk\boldH)$.  It also possesses a Fourier expansion, which can be easily worked out:
\be\label{FourierExp}
P_{\sh,Y}(z,s;\delta) =
\sum_{m}a_{m,Y}(s,y;\delta)e^{2\pi i m x},
\ee
where
\bea\label{FourierCoef}
a_{m,Y}(s,y;\delta)
&=&
\delta_{m,-h} y^s e^{2 \pi h y(1-\delta)}\psi_Y(y)
\\
\nonumber
& +&
\sum_{c \ne 0}
{{S(m,-h;cN)}\over{N^{2s}|c|^{2s}}} \cdot y^{1-s}
\\
\nonumber
& \times&
\int_{-\infty}^\infty{e^{(2\pi i h u + 2 \pi h(1-\delta))/(N^2 c^2 y(u^2 +1))-2\pi i m u y}\over(u^2 + 1)^s}
 \psi_Y\left({1\over N^2 c^2 y (u^2 +1)}\right)du.
\eea

The Petersson inner product of two forms $F,G$ in $L^2(\G\bk\boldH)$ is defined by
\be\label{Innerdef}
\<F,G\> = {1 \over \vol} \iint_{\G\bk\boldH} F(z) \overline{G(z)}{dx dy\over y^2},
\ee
where
$$\mathcal{V} =\text{volume}(\G\bk\boldH)
		= \frac{\pi}{3}N\prod_{p\mid N} \left(1+p^{-1}\right).
	$$
Set $V$ to be the function
\be\label{Vdef}
V_{f,g}(z;\ell_1,\ell_2):=y^\wt  \overline{f(\ell_1 z)}g(\ell_2z),
\ee
where $f$ and $g$ are as defined in \eqref{obvious}, which is rapidly decreasing in the cusps of $\G$. We see that $V_{f,g}$ is $\G$-invariant. Indeed, for $\gamma =\left( \begin{smallmatrix} a & b \\ Nc & d \end{smallmatrix}\right) \in \Gamma$, we have that, 
\begin{align}
f(\ell_1 (\gamma z)) & = f\left(\begin{pmatrix}
\ell_1 & 0 \\
0 & 1
\end{pmatrix}
\begin{pmatrix}
a & b \\
\frac{N_0\ell_1\ell_2}{(\ell_1,\ell_2)}c & d
\end{pmatrix} z \right) \\
& =f\left(\begin{pmatrix}
a & b\ell_1 \\
\frac{N_0\ell_2}{(\ell_1,\ell_2)}c & d
\end{pmatrix} \ell_1 z \right) = (Ncz+d)^k f(\ell_1 z), \notag
\end{align}
as $\frac{N_0\ell_2}{(\ell_1,\ell_2)} \in N_0 \Z$. Similarly, we have $g(\ell_2 (\gamma z)) = (Ncz+d)^kg(\ell_2 z)$ and so 
\beq 
\Im(\gamma z)^k \ol{f(\ell_1(\gamma z))}  g(\ell_2(\gamma z))=y^k\ol{f(\ell_1 z)}g(\ell_2 z),
\eeq giving the $\G$-invariance of $V$. 

Our approach will be to compute the inner product $\< P_{\sh,Y}(*,s;\delta),V\>$ in several different ways, and to then let $Y \rightarrow \infty$.  To begin, we compute
\bea\label{Idef}
\cI_{\ell_1,\ell_2,Y,\gd}(s;{\sh})
&:=&
\<P_{\sh,Y}(*,s;\delta),V\>
\\
\nonumber
&=&
\frac{1}{\vol}\iint_{\G\bk\boldH}P_{\sh,Y}(z,s;\delta)f(\ell_1 z) \overline{g(\ell_2z)}y^\wt{dx\, dy\over y^2}
\\
\nonumber
&=&
\frac{1}{\vol}\sum_{m_1\ge1} 
\sum_{m_2\ge1}
{a(m_1) \bar b(m_2)
}
\int_0^1 
 e^{2\pi i x (m_1\ell_1 -m_2\ell_2{-\sh})}
 dx
 \\
 \nonumber
 & & 
 \times
\int_0^\infty 
 e^{-2\pi  y(m_1\ell_1 +m_2\ell_2{-\sh}(1-\delta))} 
y^{s+\wt-1} \psi_Y(y)
{ dy\over y}
\\
\nonumber
&=&
\frac{1}{\vol} (4\pi)^{-(s+\wt-1)}
\sum_{m_1\ell_1= m_2\ell_2 + {h}} 
{a(m_1) \bar b(m_2)
\over
 (m_2\ell_2)^{s+\wt-1}
}\\
 \nonumber
 & & 
 \times
\int_{Y^{-1}4 \pi \ell_2 m_2}^{Y4 \pi \ell_2 m_2}
 e^{- y(1 + h \delta / (2 \ell_2 m_2)  )  } 
y^{s+k-1} 
{ dy\over y}.
\eea
Note that for $\Re s >1$, and all $Y >1, \delta  >0$ the integral and sum converge locally normally and satisfy the upper bound
$$
\cI_{\ell_1,\ell_2,Y,\gd}(s;{\sh}) \ll {1 \over N} \G(\Re \, s + k-1) \sum_{m_1\ell_1= m_2\ell_2 + {h}} 
{
|a(m_1) \bar b(m_2)|
\over
 (m_2\ell_2)^{\Re s+\wt-1}
},
$$
where the implied constant is absolute.

Let $\{u_j\}_{j\ge0}$ be an orthonormal basis (with respect to the inner product defined in \eqref{Innerdef}) for the discrete part of the spectrum of the Laplace operator on $L^2(\G\bk\boldH)$.  Suppose the $u_j$ have  eigenvalue $1/4+t_j^2$, and Fourier expansion
$$
u_j(z)=\sum_n\rho_j(n) y^{\hf}K_{i t_j}(2\pi |n| y)e^{2\pi i n x}.
$$

As $P_{\sh,Y}(z,s;\delta)\in L^2(\G\bk\boldH)$
we can obtain the spectral expansion of $\cI$:
\bea\label{otherhand}
\cI_{\ell_1,\ell_2,Y,\gd}(s;{\sh})&=&\sum_{j\ge1}\overline{\<V,u_j\>}  \<P_{\sh,Y}(    *    ,s;\delta),u_j\>\\&+&\frac{1}{4 \pi} \sum_\cusp \vol \int_{-\infty}^\infty {\overline{\<V,E_\cusp(*,\thf + it)\>}}\<P_{\sh,Y}(    *    ,s;\delta),E_\cusp(*,\thf + it)\>dt\nonumber.
\eea

We easily compute
\be\label{ujP}
\<P_{\sh,Y}(*,s;\delta),u_j\> = {\overline{\rho_j(-h)}\over \vol (2\pi \sh)^{s - \hf}}\int_{Y^{-1}2 \pi \sh}^{Y2 \pi \sh} y^{s-\hf}e^{ y(1-\delta)}K_{it_j}(y){dy \over y}
\ee
and
\begin{align}
& \<P_{\sh,Y}(*,s;\delta),  E_\cusp(*,\thf + it)\>\notag \\ 
& \ \ \ \ \ =
\frac{1}{\vol (2 \pi h)^{s-\hf}}\overline{\left(\frac{2\pi^{\frac{1}{2}+it} |h|^{it} \rho_\cusp\left(\frac{1}{2}+it, -h\right)}{\Gamma\left(\frac{1}{2}+it\right)}\right)}
\int_{Y^{-1}2 \pi \sh}^{Y2 \pi \sh} y^{s-\hf}e^{ y(1-\delta)}K_{it}(y){dy \over y}. \label{EP}
\end{align}
Note that $u_0$ is constant and $\<u_0,P_{\sh,Y}\>=0$.  We have taken the Eisenstein series notation from \cite{DI}, that is: $E_\cusp(z,s)$ denotes the Eisenstein series expanded at the cusp $\cusp$.  Here for each cusp $\cusp = u/w$,
\begin{multline}\label{Eadef}
E_\cusp(z,s)= \delta_{\mathfrak{a},\infty} y^s + \frac{\sqrt{\pi}\G(s-\hf)\rho_\cusp(s,0)y^{1-s}}{\G(s)}\\
+ \frac{2 \pi^s\sqrt{y}}{\G(s)}\sum_{m \ne 0}|m|^{s-\hf}\rho_\cusp(s,m)K_{s-\hf}(2\pi |m|y)e^{2\pi i mx},
\end{multline}
and
\begin{equation}\label{rhoadef}
\rho_\cusp(s,m) = \left(\frac{\left(w,\frac{N}{w}\right)}{wN}\right)^s
\sum_{(\gamma,N/w) =1}\frac{1}{\gamma^{2s}}\sum_{{{\delta \mod \gamma w}
\atop {(\delta, \gamma w) =1}}\atop \delta \gamma \equiv u \mod (w,N/w)}
e^{-2\pi im 
\delta/(\gamma w)}.
\end{equation}
 More precisely, since $N$ is square free we have, by \cite{HLee},
	for each cusp $\cusp=\frac{1}{w}$ where $w\mid N$ that
	\begin{multline}\label{rhoexact}
	\rho_\cusp(s, n)
	=
	\left(\frac{w}{N}\right)^s
	\frac{\sum_{d\mid n, \atop (d, N)=1} d^{1-2s}}
	{\zeta(2s)}
	\\
	\times
	\prod_{p\mid N} (1-p^{-2s})^{-1}
	\prod_{p\mid w,\atop p^k\|n, k\geq 0}
	\frac{p^{-2s}} {1-p^{-2s+1}}
	\left(p-p^{k(-2s+1)+1} -1 + p^{(k+1)(-2s+1)}\right)
	\end{multline}
	for $n\neq 0$.

By Lemma 3.4 in \cite{B} we have that for $t \in \R$ the $\rho_\cusp(s,m)$ satisfy 
\be\label{Blomer}
\sum_\cusp | \rho_\cusp(\thf + it,m)|^2 \ll_\epsilon ((1+ |t|)m N)^\epsilon.
\ee

We will ultimately need to understand the behavior of the integral appearing in \eqref{ujP} and \eqref{EP} for fixed small $\delta$
as $Y \rightarrow \infty$.  With this in mind we define
\be\label{Mdef}
M(s,t,\delta) = \int_0^\infty y^{s-\hf}e^{ y(1-\delta)}K_{it}(y){dy \over y}. 
\ee
Consider the well-known bounds
\be\label{Kbound1}
K_{\gamma+ir}(y) \ll \frac{e^{-y}}{\sqrt{y}}
\ee
for $y \gg 1$ and
\be\label{Kbound2}
K_{\gamma +ir}(y) \ll \lt\{
\begin{array}{ll}
y^{-|\gamma|} & \mbox{ if } \gamma+ir \neq 0 \\
|\log(y)| & \mbox{ if } \gamma+ir = 0
\end{array}
\rt.
\ee
for $y \ll 1$ where implicit constants are dependent on $\gamma,r\in\R$. From the integral formula
\be \label{wellknown}
K_{\gamma+ir}(y) = \frac12\int_0^\infty e^{-\frac y2(u+u^{-1})}u^{\gamma+ir}\frac{du}{u},
\ee
we see that $|K_{\gamma+ir}(y)| \leq K_\gamma(y)$. Thus we can take the bound $\eqref{Kbound1}$ to be uniform in $r$ for fixed $\gamma$, and similarly if $\gamma \neq 0$ we can take the bound $\eqref{Kbound2}$ to be uniform in $r$ as well. Combining this with the recursive formula \be
K_{ir}(y) = \frac{y}{2ir}\left(K_{1+ir}(y)+K_{-1+ir}(y) \right)
\ee
we see that the bound $\eqref{Kbound2}$ can be taken to be uniform for $\gamma=0$ and $|r| \geq 1$.
Given that, it is also clear that the integral expression for $M(s,t,\delta)$, given in \eqref{Mdef}, converges for any $\delta > 0$ when 
$\Re \, s > \hf+|\Im \, t|$.  (It also converges for $\delta = 0$, when 
 ${1\over 2}+ |\Im \, t| < \Re s < 1$, a fact that will be made use of later.)   The following proposition makes explicit the extent to which the integral on the right hand side of \eqref{ujP} approximates $M(s,t,\delta)$ when $\Re \, s > \hf+|\Im \, t|$. 

\begin{proposition}
\label{prop:Ylimit} 
Let $M(s,t,\delta) $ be as defined in \eqref{Mdef}.
Fix $\epsilon > 0,$ $1 > \delta > 0$, $Y \gg \max(\delta^{-1},h),$  $A \in \Z_{\geq 0},$ $t \in \C$  and $s \in \C$ with $\Re \, s >  \hf +|\Im \, t|+ \epsilon$ .   
Then
\beann
\left| M(s,t,\delta) -  \int_{Y^{-1}2 \pi \sh}^{Y 2 \pi \sh}  y^{s-\hf}e^{ y(1-\delta)}K_{it}(y){dy \over y}\right|   \ll
\frac{e^{-Y2\pi h \delta}(Y h)^{\Re \,s+A-2} }
{ \delta(1 + |t|)^A}  + \frac{(Y^{-1}  \sh)^{\Re \, s - \hf-|\Im \, t| }}{(1+|t|)^A},
\eeann
with the implied constant depending only on $A$, $s$, $|\Im \, t |$ and $\epsilon$.
\end{proposition}
\begin{proof}

We first need to bound $K_{it}(y)$ uniformly in $t$, for $y$ both large and small.   For this we have
\begin{lemma}\label{lemma:Ktbound}
 For  $A \in \Z_{\geq 0}$ and $|r|,y \gg1$,
$$
K_{\gamma+ir}(y) \ll_{A,\gamma} \frac{e^{-y}y^A}{\sqrt{y} (1 + |r|)^A}.
$$
For $A\in \Z_{\geq 0},|r| \gg1$, and $0<y \ll 1$, 
$$
K_{\gamma+ir}(y) \ll_{A,\gamma}
 \frac{y^{-|\gamma|}   }{(1 + |r|)^A}.
$$
\end{lemma}

To prove Lemma~\ref{lemma:Ktbound} we recall the definition \eqref{wellknown}.
The result follows after integrating by parts $A$ times, applying the bound \eqref{Kbound1} when $y$ is large and \eqref{Kbound2} when $y$ is small to  $K_{m+\gamma+ir}(y)$, with $|m| \leq A$. 

Applying Lemma~\ref{lemma:Ktbound} we find that 
\beann
 \int_{Y2 \pi \sh}^{\infty}  
 &
 &
 y^{s-\hf}e^{ y(1-\delta)}K_{it}(y){dy \over y}
 \ll_{A,|\Im \, t|} 
 \int_{Y2 \pi \sh}^{\infty} y^{\Re \,s-\hf}\frac{e^{ y(1-\delta)}e^{-y} y^{A-\hf}}{(1 + |t|)^A}\frac{dy}{y}
 \\
&\ll& 
{1\over(1 + |t|)^A} \int_{Y2 \pi \sh}^{\infty} y^{\Re \,s +A -1}e^{ -y\delta} \frac{dy}{y}
\\
&=& 
{1\over(1 + |t|)^A \gd^{\Re \, s+A-1}} 
\G\left[\Re \,s+A-1,Y2\pi h\delta\right]
 ,
 \eeann
where
$$
\G[s,x]:=\int_x^\infty e^{-y} y^s {dy\over y} \ll_s e^{-x}x^{\Re \, s-1},
$$
as $x\to\infty$. Thus when $Y \gg \delta^{-1}$, 
$$
\int_{Y2 \pi \sh}^{\infty}  
 y^{s-\hf}e^{ y(1-\delta)}K_{it}(y){dy \over y}
\ll_{A,s,|\Im \, t|} 
\frac{e^{-Y2\pi h \delta}(Y h)^{\Re\, s+A-2}} 
{ \delta(1 + |t|)^A }.
 $$

Similarly, applying the second part of the lemma,
$$
 \int_{0}^{Y^{-1}2 \pi \sh}  y^{s-\hf}e^{ y(1-\delta)}K_{it}(y){dy \over y} \ll_{A,|\Im \, t|} \frac{(Y^{-1} 2 \pi h)^{\Re \,s - \hf -|\Im \, t| }}{(1 + |t|)^A},
 $$
completing the proof of Proposition~\ref{prop:Ylimit}.

\end{proof}
\section{The analytic continuation of $M(s,t,\delta)$}\label{section:M}

In this section we will provide a meromorphic continuation of  $M(s,t,\delta)$  to $s\in \C$ with $\Re \, s \le \hf+|\Im\, t|$.   The following proposition locates the poles of $M(s,t,\delta)$ and gives three different estimates at points, $s$, where $M(s,t,\delta)$ is analytic. 

\begin{remark}
In the proposition below, we find that it is easier to work with $M(s,t,\delta)$ if we work with the variable $z=it$ rather than $t$ alone. So we consider instead $M(s,z/i,\delta)$. Thus the region $\Re \, s \le \hf+|\Im\, t|$ becomes  $\Re \, s \leq \frac{1}{2}+|\Re\,  z|$. Furthermore, to simplify error growth estimates, we let $t = \Im \, z$ in the statement of the proposition. This differs somewhat from the definition of $M(s,t,\delta)$, where $t$ can be complex, but since we keep $|\Im \ t|$ small we will find that the error estimates are accurate for both definitions of $t$ and so this notation is justified. 
\end{remark}

\begin{proposition}
\label{prop:Mbounds} 
Fix $\epsilon$ and $\delta$ such that  $1 > \epsilon, \delta >0$ and fix  $A \gg 1$ to be a non-integer. 
The function $M(s,z/i,\delta)$ has a meromorphic continuation to all $s$ in the half-plane $\Re \, s> \hf - A-|\Re \, z|$ and for $z \in \C$ where $|\Re \, z| < \tfrac 14$.  In this region $M(s,z/i,\delta)$ is analytic except, in the case when $z \neq 0$, for simple poles at the points $s = \hf \pm z - \ell$ for $\ell \in \Z_{\geq 0}$ with residues 
\bea\label{Mres}
\Res_{s =\hf \pm z - \ell} M(s,z/i,\delta)& =& \frac{(-1)^\ell \sqrt{\pi}2^{\ell \mp z} \Gamma(\hf \mp z + \ell) \Gamma (\pm 2 z - \ell) }{\ell ! \Gamma (\hf + z)\Gamma(\hf - z)}\nonumber \\
 &&+ \mathcal{O}_\ell((1+|t|)^{\pm \Re \, z -\hf}e^{-\frac{\pi}{2}|t|}\delta),
\eea
which are meromorphic in $z$ and here $t=\Im \, z $. There is no error term when $\ell=0$. When $z =0$ there are double poles at $s=\hf-\ell$ for $\ell \geq 0$. 


Considering the singularities as poles in $z$, we have that
\begin{align}
\Res_{z=\pm(s+\ell-\hf)} M(s,z/i,\delta)=& \notag \mp \frac{(-1)^\ell \sqrt{\pi}2^{\hf-s}\Gamma(1-s)\Gamma(2s+\ell-1)}{\ell!\Gamma(s+\ell)\Gamma(1-s-\ell)} \\
& +\Or_{\ell,\sigma}\lt(\left|\frac{\Gamma(2s+\ell)}{\Gamma(s+1)} \right| \delta\rt) \label{ires2}
\end{align}
These residues above have meromorphic continuations to all $s$ in the half-plane specified above. The error terms have poles at $s = -\hf-r$ for $2r \geq \ell-1$ as well, but retain the dependence on $\delta$ in the residue. There is no error term if $\ell=0$. 

When $s$ and $z$ are at least $\epsilon$ away from the poles of $M(s,z/i,\delta)$, we have the upper bound
\be\label{upper1}
M(s,z/i,\delta) \ll_{A,\epsilon} (1 + |t|)^{2 \Re\, s -2 - 2A}(1+|s|)^{4A- 3 \Re\, s +3}\delta^{-A }e^{-\frac\pi 2 |\Im \, s|}.
\ee

Restricting the half-plane to a vertical strip by imposing the additional condition $\Re \, s \leq \hf+|\Re \, z|$, and further requiring $s$ and $z$ be at least a distance $\epsilon$ away from poles, we have that when
$\delta (1 +|t|)^{2} \le 1$, 
\bea\label{deltasmall}
M(s,z/i,\delta)& =&    \frac{\sqrt{\pi} 2^{\hf-s}\G(s-\hf -z) \G(s-\hf + z) \G(1-s)}{\G(\hf - z)\G(\hf + z)}\nonumber \\  &+ &
\Or_{A,\epsilon} \left( (1 + |t|)^{2\Re \, s-2 +  2\epsilon}(1 + |s|)^{ \hf - 2\Re s+|\Re \, z|}e^{-\frac{\pi}{2}|\Im \, s|} \delta^\epsilon \right)
\eea
while for $\delta (1 +|t|)^{2} > 1$
\be\label{deltabig}
M(s,z/i,\delta) \ll_{A,\epsilon}  (1 + |t|)^{2\sigma -2}(1 + |s|)^{ 4A-3\Re \, s +3 }e^{-\frac{\pi}{2}|\Im \, s|}.
\ee

When $\Re \, z =0$ and $|z|,|\Im \, s| \gg 1$, $|s\pm z -\hf -m| =\epsilon>0$,  for $\epsilon$ small, we have
\beq \label{nearp1}
M(s,z/i,\delta) \ll_m \epsilon^{-1}(1+|s|)^{1-\Re \, s}e^{-\frac{\pi}{2}|\Im \, s|}.
\eeq
\end{proposition}

\begin{remark}
We note that $M(s,z/i,\delta)$ does not have a pole at $s=1$, in spite of the appearance of $\G(1-s)$ in \eqref{deltasmall}. It is cancelled by the error term in that formula, which does not vanish as $\delta \rightarrow 0$ when $\Re \, (s+z) > \hf$.
\end{remark}

\begin{proof}

As in the statement of the proposition, take $|\Re\, z| < \frac{1}{4}$, we will assume this throughout the proof. From \cite{GR} formula 6.621(3) we obtain, when $\Re \, s >\hf+|\Re\, z|$, 
\bea\label{Mbound1}
&&M(s,z/i,\gd) =  \int_0^\infty y^{s-\hf}e^{ y(1-\delta)}K_{z}(y){dy \over y}
\\ 
\nonumber
 && = \frac{\sqrt{\pi} 2^{z}\G(s-\hf +z)\G(s-\hf-z)}{\delta^{s - \hf + z}\G(s)}  F(s-\thf +z,\thf+z;s;1-\tfrac 2\delta),  
\eea
where $F(\alpha,\beta;\gamma ; x)$ is the usual hypergeometric function. By using the transformation formula 9.132(1), also in \cite{GR}, this expansion becomes
\begin{align}
M(s,z/i,\delta) =&  \frac{\sqrt{\pi} 2^{\hf-s}\G(s-\hf +z)\G(s-\hf-z)}{\G(1-s)\Gamma(\hf+z)\G(\hf-z)}F(s-\thf+z,s-\thf-z;s;\tfrac{\delta}{2}) \\
& + \left(\tfrac{\pi}{2}\right)^{\hf}\delta^{1-s} \G(s-1)F(\thf+z,\thf-z;2-s;\tfrac{\delta}{2}). \notag
\end{align}

Since the Taylor expansion for $F(\alpha,\beta;\gamma;x)$ converges for $|x|<1$ when $\Re(\alpha+\beta-\gamma)<1$, we see that for $ \Re\, s<2$, away from the obvious potential poles, we have the following expansion,
\begin{align} \label{newexpansion}
M(s,z/i,\delta) =&  \frac{\pi^{\hf} 2^{\hf-s}\Gamma(s)\G(1-s)}{\Gamma(\hf+z)\G(\hf-z)}\sum_{n=0}^\infty \frac{\Gamma(s-\hf+z+n)\G(s-\hf-z+n)}{n!\G(s+n)} \left(\frac{\delta}{2} \right)^n \\
& + \frac{\left(\frac{\pi}{2}\right)^{\hf}\delta^{1-s} \G(s-1)\G(2-s)}{\Gamma(\hf+z)\G(\hf-z)} \sum_{n=0}^\infty \frac{\Gamma(\hf+z+n)\G(\hf-z+n)}{n! \Gamma(2-s+n)}\left(\frac{\delta}{2} \right)^n. \notag
\end{align}
From this it is easy to compute residues. First we note that,
\beq
\Res_{s=1} M(s,z/i,\delta) = (1-1)\left(\frac{\pi}{2}\right)^{\hf} \sum_{n=0}^\infty \frac{\Gamma(\hf+z+n)\G(\hf-z+n)}{(n!)^2\Gamma(\hf+z)\G(\hf-z)}\left(\frac{\delta}{2} \right)^n =0
\eeq
and similarly for $m \in \Z_{>0}$, 
\begin{align}
\Res_{s=-m} M(s,z/i,\delta) =& \frac{(-1)^m \pi^{\hf}2^{\hf+m}\G(s)}{\G(\hf+z)\G(\hf-z)} \sum_{n=m+1}^{\infty} \frac{\G(n-m-\hf+z)\G(n-m-\hf-z)}{(n-m-1)!n!} \left(\frac{\delta}{2} \right)^n \notag \\
& - \frac{(-1)^m \pi^\hf 2^{-\hf} \delta^{m+1}}{\G(\hf+z)\G(\hf+z)} \sum_{n=0}^\infty \frac{\G(\hf+z+n)\G(\hf-z+n)}{n!(m+n+1)!} \left(\frac{\delta}{2} \right)^n,
\end{align}
where we see the above residue is also zero when we change the index of the first sum above from $n \to n+m+1$. So there are no poles for $s \in \Z_{\leq 0}$. Furthermore we can compute that
\begin{align} \label{rescomp}
\Res_{s=\hf\pm z -\ell} M(s,z/i,\delta) = &  \frac{\pi^{\hf}2^{\ell \mp z}\G(\hf\pm z -\ell)\G(\hf \mp z+\ell)}{\Gamma(\hf+z)\Gamma(\hf-z)} \\
& \times  \sum_{n=0}^\ell \frac{\Gamma(\pm 2z+n-\ell)(-1)^{\ell-n}}{n!(\ell-n)!\G(\hf \pm z +n-\ell)}\left(\frac{\delta}{2}\right)^n. \notag
\end{align}

We recall Stirling's formula for the gamma function. For bounded $x \in \mathbb{R}$, as $y \rightarrow \infty$,
\bea\label{Stirling}
|\G(x+iy)| \asymp (1+|y|)^{x - \hf}e^{-\frac\pi 2 |y|},
\eea
where $f(y) \asymp g(y)$ means that $f(y) \ll g(y)$ and $g(y) \ll f(y)$. Using this with \eqref{rescomp}, we see we can take $n=0$ as the main term and bound the remainder to get \eqref{Mres}. We see from \eqref{newexpansion} that when $z=0$ that are double poles at $s=\hf-\ell$ for $\ell \leq 0$ with Laurent series that do not vanish as $\delta \to 0$. 

 If we consider the poles as poles in $z$ rather than poles in $s$, we get that:
\begin{align}
\Res_{z=\pm(s+\ell-\hf)} M(s,z/i,\delta) =& \mp \frac{\pi^{\hf}2^{\hf-s} \Gamma(s)\Gamma(1-s)}{\Gamma(s+\ell)\Gamma(1-s-\ell)} \\
& \times  \sum_{n=0}^\ell \frac{\G(2s+\ell+n-1)(-1)^{\ell-n}}{n!(\ell-n)!\Gamma(s+n)}\left(\frac{\delta}{2}\right)^n \notag
\end{align}
and again taking $n=0$ and using Stirling's formula gives us \eqref{ires2}.


By  formula 9.113 in \cite{GR},  the hypergeometric function in \eqref{Mbound1} can be written as a contour integral
\begin{align}
\label{Mbound5}
F(s-&\thf+z, \thf+z;s;  1-\tfrac 2\delta)
 = \\ & \frac{\G(s)}{\G(s-\hf+z)\G(\hf + z)}
\lt( \frac{1}{2 \pi i} \int_C
  \frac{\G(s-\hf+u+z)\G(\hf +u+ z)\G(-u)\left(\frac{2}{\delta}-1\right)^u}{\G(s+u)} \ du \rt) \nonumber 
\end{align}
where $C$ goes from $-i\infty$ to $i\infty$ and curves so that the poles of $\G(-u)$ lie to the right of $C$ and the poles of $\G(s-\hf+u+z),\G(\hf +u+ z)$ lie to the left. For  $\Re (s+z) > \hf$, $C$ can simply be the vertical line $\Re \, u = a$, where  $a$ is chosen so that 
$\max(\hf -\Re(s+z),-\hf-\Re \, z) < a < 0$.  

The expression \eqref{Mbound5} has a meromorphic continuation to all $s \in \C$ and $z \in \C$ with $|\Re \, z| < \qtr$. Indeed, we only need to show this for the integral, and we see that it is possible to move the $(s,z)$ pair to nearly any other position in $\C^2$, adjusting $C$ along the way, so long as the points at $s+z= \hf -\ell$ and $z=-\hf-\ell$  for $\ell \geq 0$ are not crossed.  This combined with \eqref{Mbound1} provides a continuation of 
$M(s,z/i,\gd)$.    For general $(s,z) \in \C^2$ where $s+z-\hf \notin \mathbb{Z}_{\leq 0}$ with $\Re(s+z) \leq \hf$ and $|\Re \, z|<\qtr$, the curve $C$ can be described as follows.    Write
$\Re(s+z) = \hf - R$, with $R \ge 0$.   Choose $b$ such that $ 0 > b > \max (R - [R] -1, -\tfrac 14) $. The poles of $\G(\hf +u+ z)$ all lie to the left of $\Re \, u = b$ and
the poles of $\G(-u)$ all lie to the right of $\Re \, u = b$. The poles of $\G(s-\hf+u+z)$  to the right of 
$\Re \, u = b$ are at  $\Re\, u =  \hf - \Re(s+z) -\ell = R - \ell $ for $0 \le  \ell \le [R ]$.  We can make the curve $C$ follow the line $\Re \, u = b$ until $\Im \, u = -\Im(s+z)$ and then move and weave around the finitely many poles between $\Re \, u = R-[R] $ and $\Re \, u = R$, continuing to $i\infty$ along the line $\Re\, u = c$, where $R < c < [R] +1$. (See Figure~\ref{fig:curve}.) We observe that this is still permissible when $R \in \mathbb{Z}$, as long as $s+z-\hf \notin \mathbb{Z}_{\leq 0}$, as the poles of $\G(-u)$ and $\G(s-\hf+u+z)$ do not overlap. 
  \begin{figure}
\begin{center}
\includegraphics{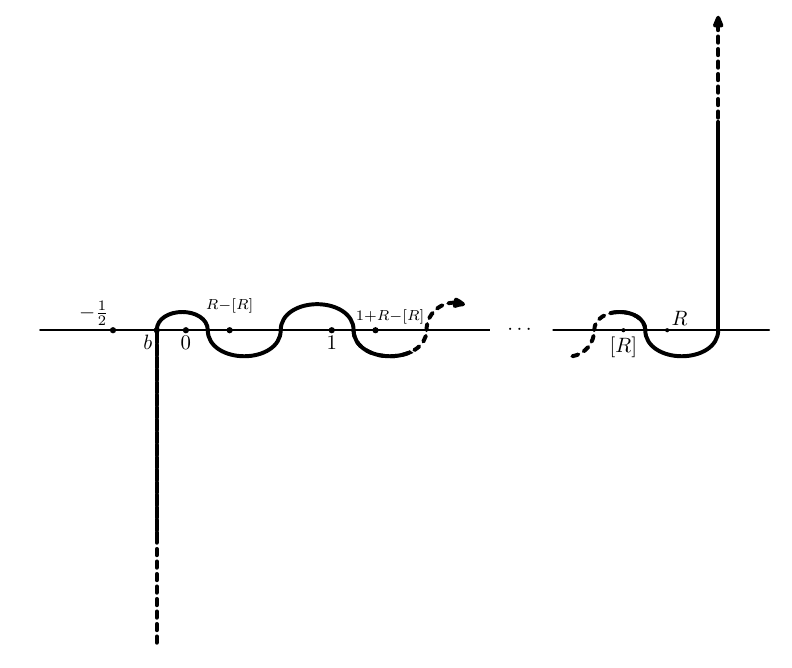}
\end{center}
\caption{The contour when $\Im \, s =0$ and $z =0$ }
\label{fig:curve}
\end{figure}

We continue to require that $s+z-\hf \notin \mathbb{Z}_{\leq 0}$. For $A \gg 1$ and $A \notin \Z$, fix any $R =\hf-\Re(s+z)$ such that $A \notin \mathbb{Z}$ and  $(A+R) \notin \Z$,  and $A>\hf+|\Re \, s|+|\Re \, z|+\epsilon$, and straighten $C$ to the vertical line $\Re \, u = A$.  In the process, poles of the integrand are passed over at $u = \ell$, for $0 \le \ell < A $.   The residue after passing the pole at $u = \ell$ is given by 
\be\label{line0}
 \frac{\sqrt{\pi}2^{z}\G(s - \hf - z)}{ \delta^{s - \hf + z}\G(\hf + z )} R(s,z,\ell),
\ee
where
\be
\label{ResDef}
R(s,z,\ell) = \frac{\G(s-\hf + z +\ell)\G(\hf + z+ \ell)\left(\frac2\delta -1 \right)^{ \ell}(-1)^{\ell+1}}{\ell!\G(s+\ell)}.
\ee
Thus
$$
 M(s,z/i,\delta) =  \frac{\sqrt{\pi}2^{z}\G(s - \hf - z)}{ \delta^{s - \hf + z}\G(\hf + z )} N_1(s,z,\delta),
$$
where
\bea\label{line2.5}
&&N_1(s,z,\delta)= -\sum_{0 \le \ell <A }R(s,z,\ell)\\ &+ & \frac{1}{2 \pi i} \int_{(A)}\frac{\G(s-\hf+u+z)\G(\hf +u+ z)\G(-u)\left(\frac{2}{\delta}-1\right)^u du}{\G(s+u)}.\nonumber
\eea 

The integral in \eqref{line2.5} contributes no poles in $s$ or $z$, as $\Re(s+z+u) -\hf=A-R >\epsilon$ and $\hf+\Re(z+u)>\qtr+A>0$. 
Inspecting lines \eqref{line0} to \eqref{line2.5} it is clear that the only possible locations of poles of 
$M(s,z/i,\delta)$ for $\hf+A>\Re(s+z) > \hf - A$ are of the form $s = \hf \pm z -m$ for  $ m\in \Z_{\geq 0}$. 

Repeat the argument above but now shift to the vertical line $\Re \, u = A +R$.  This time poles are passed over at $u = \ell$, for $ 0  \le \ell <  A + R$.  
Changing variables, 
$ u \rightarrow u - (s - \hf +z)$, and substituting back into \eqref{Mbound1}, we now have, after moving the contour and subtracting the residue contribution,
$$
 M(s,z/i,\delta) =  \frac{\sqrt{\pi}2^{z}\G(s - \hf - z)}{ \delta^{s - \hf + z}\G(\hf + z )} N_2(s,z,\delta),
$$
where
\bea\label{line2}
&&N_2(s,z,\delta) = -\sum_{0 \le \ell <A + R}R(s,z,\ell)\\ &+ & \frac{1}{2 \pi i} \int_{(A)}\frac{\G(u)\G(1-s+u)
\G(s-\hf +z-u)
\left(\frac{2}{\delta} - 1\right)^{u-s+\hf-z}du}{\G(\hf - z +u)}.\nonumber
\eea

Suppose $s$ and $z$ are at least some fixed $\epsilon>0$ away from the poles of $M(s,z/i,\delta)$. Writing $s = \sigma + ir$ and $u = A + iv$ and recalling that $\Im \, z =t$, the following upper bounds are easily confirmed, using \eqref{Stirling}: 
$$
R(s,z,\ell) \ll_{A,\epsilon} \delta^{ -\ell}e^{-\frac{\pi}{2}(|r+t|+|t|-|r|)}\frac{(1+|r+t|)^{\sigma+\Re \, z +\ell -1}(1+|t|)^{\Re \, z+\ell}}{(1+|r|)^{\sigma + \ell -\hf}},
$$

\be\label{expo}
 \frac{\sqrt{\pi}2^{z}\G(s - \hf - z)}{ \delta^{s - \hf + z}\G(\hf + z )}\ll_{A,\epsilon} \delta^R (1 +|r-t|)^{\sigma-\Re \, z - 1}
(1+|t|)^{-\Re \, z}e^{-\frac{\pi}{2}(|r-t|-|t|)},
\ee
and
\bea\label{thethird}
&&\frac{\G(u)\G(1-s+u)\G(s-\hf +it-u)
\left(\frac{2}{\delta} - 1\right)^{u-s+\hf-it}}{\G(\hf - it +u)}\\ &&\ll_{A,\epsilon} \frac{(1+|v|)^{A-\hf}(1+|v-r|)^{A+\hf-\sigma}
(1+|t-v+r|)^{\sigma +\Re \, z- 1 -A}}{\delta^{A+R}(1 +|t-v|)^{A-\Re \, z}e^{\frac{\pi}{2}(|v| +|v-r| + |r-v+t| -|t-v|)}}
\nonumber.
\eea

Combining the first two, when $|t| \geq 2|r|$ one has
\bea\label{part1}
&&\sum_{0 \le \ell <  A +R }R(s,z,\ell) \frac{\sqrt{\pi}2^{z}\G(s - \hf - z)}{ \delta^{s - \hf + z}\G(\hf + z )}\\ 
&&\ll_{A,\epsilon} (1+|t|)^{2A-\hf}(1+|r|)^{\hf-\sigma}\delta^{ -A}e^{-\frac{\pi}{2}|t|}.\nonumber
\eea
When $ |t| \leq  2|r|$, 
\bea\label{part101}
&&\sum_{0 \le \ell <  A +R }R(s,z,\ell) \frac{\sqrt{\pi}2^{z}\G(s - \hf - z)}{ \delta^{s - \hf + z}\G(\hf + z )}\\ 
&&\ll_{A,\epsilon} (1+|r|)^{A-\sigma}\delta^{ -A}e^{-\frac{\pi}{2}|r|}.\nonumber
\eea

By a case analysis of the possible respective sizes and signs of $v,r$ and $t$, we find that the exponent $|v| +|v-r| + |r-v+t| -|t-v|$
in \eqref{thethird} added to the exponent $|r-t|-|t|$ of \eqref{expo} is always $\ge |v| - |r|$ and $ |r|$. So we have that when $|t| \geq 2|r|$, which forces $s$ and $z$ to not be near the poles of $M(s,z/i,\delta)$,
\begin{align}
& \frac{\G(s - \hf - z)}{ \delta^{s - \hf + z}\G(\hf + z )}  \int_{A - i \infty}^{A + i \infty}\frac{\G(u)\G(1-s+u)
\G(s-\hf +z-u)  \notag
\left(\frac{2}{\delta} - 1\right)^{u-s+\hf-z}du}{\G(\hf - z +u)} \\
& \ \  \ll_A \delta^{-A}(1+|t|)^{ 2\sigma-2A-2}e^{-\frac{\pi}{2}|r|} \int_{A-2ir}^{A+2ir} (1+|v|)^{A-\hf}(1+|v-r|)^{A+\hf-\sigma} \ du  \notag \\
& \ \ \  \ + \delta^{-A}(1+|t|)^{\sigma-2\Re \, z-1}e^{-\frac{\pi}{2}|r|}\int\limits_{\Re \, u=A, |v|>2r} \frac{(1+|v|)^{2A-\sigma}}{(1+|t-v|)^{2A+1-\sigma -2 \Re \, z}} e^{-\frac{\pi}{2}(|v|-2|r|)} \ du \notag \\
& \ \  \ll_A \delta^{-A}(1+|t|)^{2\sigma-2A-2}(1+|r|)^{2A-\sigma+1}e^{-\frac{\pi}{2}|r|}.
\end{align}

It is easier to show by a similar argument that when $|r| \geq 2|t|$  and $s$ and $z$ are at least $\epsilon$ away from the poles of $M(s,z/i,\delta)$, then
\begin{align}
 \frac{\G(s - \hf - z)}{ \delta^{s - \hf + z}\G(\hf + z )} & \int_{(A)}\frac{\G(u)\G(1-s+u)
\G(s-\hf +z-u)  \notag
\left(\frac{2}{\delta} - 1\right)^{u-s+\hf-z}du}{\G(\hf - z +u)}  \\
& \ll_{A,\epsilon} \delta^{-A} (1+|r|)^{2A-\sigma+1}e^{-\frac{\pi}{2}|r|}.
\end{align}
Combining these two bounds with \eqref{part1}, and noting that  when $|t| \leq |r|$
$$
(1+|t|)^{2\sigma-2A-2}(1+|r|)^{4A-3\sigma+3} \gg (1+|r|)^{2A-\sigma+1},
$$
we get the upper bound \eqref{upper1}. 

To prove  \eqref{deltasmall} and \eqref{deltabig}, we return to the original contour integral in line \eqref{Mbound5}.  Recall that for the contour to exist we require $\hf-\Re(s+z) \notin \Z_{\geq 0}$. Also recall that $|\Re \, z| < \frac{1}{4}$. Again we write $\Re(s+z) = \hf - R$, but now begin with $R > \min(0,-2\Re \, z)$.  Instead of deforming the contour by moving it to the right, we instead straighten the contour until it becomes the line 
$\Re\, u = b$, where $b$ is as defined above, and simple poles are passed over at the points $u = \hf -s -\ell -z $, for $0 \le \ell \le [R]$.
 The residue of the integrand at $u = \hf -s -\ell -z $ is
$$
\frac{(-1)^\ell \G(1-s-\ell)\G(s-\hf + z + \ell)(\tfrac{2}{\delta} -1)^{\hf -s -z -\ell}}{\ell !\G(\hf -z -\ell)}.
$$
Thus
\be
\label{pre}
M(s,z/i,\delta) = C+D,
\ee 
where
\bea\label{pre1}
C& = &\sum_{ \ell =0}^{[R]} \frac{(-1)^\ell \sqrt{\pi}2^{z}\G(s-\hf -z)}{\ell !\G(\hf+z)\G(\hf -z -\ell)}\\ &\times & \G(1-s-\ell)\G(s-\thf + z + \ell)(2-\delta )^{\hf -s -z }(\tfrac 2\delta -1)^{ -\ell}\nonumber
\eea
and

\bea\label{pre2}
D& =& \frac{\sqrt{\pi}2^{z}\G(s - \hf - z)}{ \delta^{s - \hf + z}\G(\hf + z )}\\
& \times &
 \frac{1}{2 \pi i} \int_{b - i \infty}^{b + i \infty}\frac{\G(s-\hf +z +u)\G(\hf+z+u)
\G(-u)
\left(\frac{2}{\delta} - 1\right)^{u}}{\G(s +u)}du\nonumber.
\eea
Fix some small $\epsilon$ where $0< \epsilon < \qtr$. In the case where $\Re \, z >0$ and $0>R>-2\Re \, z$, this means $C=0$, but we can shift $b$ to $-2\Re \, z -\qtr+\epsilon$ and then we will retrieve the $\ell=0$ term.

 Again letting $s$ and $z$ be $\epsilon>0$ away from any poles and $R \geq 0$, an application of Stirling's formula and the binomial theorem to \eqref{pre}, \eqref{pre1}, \eqref{pre2} yields 
\bea\label{Mmain}
&&M(s,z/i,\delta) = \frac{ \sqrt{\pi}2^{\hf-s}\G(s-\hf +z)\Gamma(s-\hf-z)\Gamma(1-s)}{\G(\hf+z)
\G(\hf -z )\nonumber}\\ &+ &\Or_{A,\epsilon} \left( (1 + |t|)^{2 \sigma -2}(1 + |s|)^{\hf - 2\sigma+|\Re \, z|}e^{-\frac{\pi}{2}|\Im \, s|} \max_{\ell = 1,[R\rfloor} \left(\delta^\ell (1 + |t|)^{2 \ell}\right)\right)\nonumber\\
&+ &\Or_{A,\epsilon} \left( (1 + |t|)^{-1-2\Re \, z-2b}(1 + |s|)^{\hf - 2\sigma+|\Re \, z| }e^{-\frac{\pi}{2}|\Im \, s|} \delta^{\hf - \sigma-z -b}
\right),
\eea
where the main term comes from the $\ell = 0$ part of $C$. By shifting the line of integration of $D$, we can also get  roughly the expression when $\Re \, z> 0$ and $0 \geq R \geq -2\Re\, z$, replacing $b$ with $-2\Re\,z-\qtr+\epsilon$ and the first error term given above is eliminated. 

Now consider the two cases: $\delta (1 + |t|)^2 \le 1$ and  $\delta (1 + |t|)^2 > 1$.  Choose any 
$1>\epsilon > 0$.  In the first case
$$
\max_{\ell = 1,[R\rfloor} \left(\delta^\ell (1 + |t|)^{2 \ell}\right)
= \delta (1 + |t|)^{2 } \le \delta^\epsilon (1 + |t|)^{2 \epsilon}.
$$
The first error term of \eqref{Mmain} is then
$$
\ll (1 + |t|)^{2 \sigma -2+2\epsilon}(1 + |s|)^{\hf - 2\sigma+|\Re \, z|}e^{-\frac{\pi}{2}|\Im \, s|} \delta^\epsilon
$$
and the second is
$$
\ll (1 + |t|)^{-1}(1 + |s|)^{\hf - 2\sigma+|\Re \, z|}e^{-\frac{\pi}{2}|\Im \, s|} \delta^{\hf - \sigma +\epsilon -\epsilon }
$$
which is
$$
\ll (1 + |t|)^{2 \sigma-2 + 2 \epsilon}(1 + |s|)^{\hf - 2\sigma+|\Re \, z|a}e^{-\frac{\pi}{2}|\Im \, s|} \delta^\epsilon.
$$
Thus both pieces of the error term are 
$$
\Or \left( (1 + |t|)^{2 \sigma-2 + 2 \epsilon}(1 + |s|)^{\hf - 2\sigma+|\Re \, z| }e^{-\frac{\pi}{2}|\Im \, s|} \delta^\epsilon \right),
$$
giving us \eqref{deltasmall}.

In the second case,  $\delta (1 + |t|)^2 > 1$, we return to the first estimate, \eqref{upper1}.  Substituting
$\delta^{-A} < (1 + |t|)^{2A}$ we obtain 
$$
M(s,z/i,\delta) \ll  (1 + |t|)^{2\sigma -2}(1 + |s|)^{ 4A-3\sigma +3 }e^{-\frac{\pi}{2}|\Im \, s|},
$$
giving us \eqref{deltabig}.

Let $r =\Im \, s$. Let $\Re \, z =0$, and $|t|,|r| \gg 1$ and suppose that $|s+z-\hf-m|=\epsilon$ for small $\epsilon>0$, which forces $t \sim -r$. Then using \eqref{pre}, Stirling's Approximation and 
$$
\Gamma(s-\thf+z+m) \ll_m \epsilon^{-1},
$$
we have that 
\beq \label{nearp}
M(s,z/i,\delta) \ll_m  \epsilon^{-1} e^{-\frac{\pi}{2}|r|}.
\eeq
Similarly suppose that $|s-z-\hf-m|=\epsilon>0$, which forces $t \sim r$. Again using \eqref{pre}, Stirling's Approximation, and 
$$
\Gamma(s-\thf-z) \ll_m \epsilon^{-1},
$$
we have that
\beq
M(s,z/i,\delta) \ll_m \epsilon^{-1} (1+|r|)^{1-\sigma}e^{-\frac{\pi}{2}|r|}
\eeq
which together with \eqref{nearp} gives \eqref{nearp1}, completing the proof of Proposition~\ref{prop:Mbounds}.
\end{proof}

\section{The limit as $Y \rightarrow \infty$} \label{sect:why}
We are now in a position to analyze the uniformity of convergence of the spectral decomposition of 
$\cI_{\ell_1,\ell_2,Y,\gd}(s;{\sh})$ given in \eqref{otherhand}.
We assume at first that $\Re \, s > \hf$.
Recall that by \eqref{otherhand} we may write
$$
\cI_{\ell_1,\ell_2,Y,\gd}(s;{\sh}) = \cI_{\ell_1,\ell_2,Y,\gd}^{\,\text{cusp}}(s;{\sh})+ \cI_{\ell_1,\ell_2,Y,\gd}^{\,\text{cts}}(s;{\sh}),
$$
where
$$
\cI_{\ell_1,\ell_2,Y,\gd}^{\,\text{cusp}}(s;{\sh}) 
	= 
\sum_j\frac{\overline{\rho_j(-h)}}{\vol(2\pi\sh )^{s-\hf}} 
\int_{Y^{-1}2 \pi \sh}^{Y 2 \pi \sh}  y^{s-\hf}e^{ y(1-\delta)}K_{it_j}(y){dy \over y}\overline{\<V,u_j\>}
$$
and
\begin{multline*}
\cI_{\ell_1,\ell_2,Y,\gd}^{\,\text{cts}}(s;{\sh}) = \frac{1}{4 \pi \vol(2\pi\sh )^{s-\hf}} \sum_\cusp \int_{-\infty}^\infty \overline{\left( \vol \frac{2\pi^{\frac{1}{2}+it} |h|^{it} \rho_\cusp\left(\frac{1}{2}+it, -h\right)}{\Gamma\left(\frac{1}{2}+it\right)}\right)}
	\\  
	\times\int_{Y^{-1}2 \pi \sh}^{Y2 \pi \sh} y^{s-\hf}e^{ y(1-\delta)}K_{it}(y){dy \over y}\overline{\<V,E_\cusp(*,\thf + it)\>} \  dt.
\end{multline*}

Of key importance is the average size of $\overline{\rho_j(-h)}$ and of 
$\overline{\<V,u_j\>}$.  The second is addressed by the following proposition, a generalization of \cite{BR4},  which is proved in the Appendix by Andre Reznikov  (Theorem  B):
\begin{proposition}\label{prop:Rezresult}
With the notation as given above
$$
\sum_{\mathfrak{a}}\int_{-T}^{T} \vol \left| \<V,E_\cusp(*,\thf + it)\>\right|^2e^{\pi |t|}dt+\sum_{|t_j| \sim T}|\<V,u_j\>|^2e^{\pi |t_j|} \ll (\ell_1 \ell_2)^{-k}T^{2k}\log (T) .
$$
\end{proposition}

\begin{remark}
An explicit analysis by Hansen, \cite{Ha}, of some formulas of Ichino, \cite{Ich}, reveals that Proposition~\ref{prop:Rezresult} is essentially stating that 
\be\label{Lindresult}
\frac{L(\thf, f \otimes \bar g\otimes \bar u_j)}{L(1,\mathrm{Ad}\,u_j)}
\ee
is close to constant on average.
\end{remark}

An application of the Kuznetsov trace formula (see, for example, \cite{IK}, Theorem 16.8, line (16.56)) shows that 
\be\label{rhosum0}
 \sum_{|t_j| \sim T} |\overline{\rho_j(-h)}|^2e^{-\pi |t_j|} \ll \ell_1\ell_2T^2 +
 T^{7/4}\sigma_0(h) \sigma_0(\ell_1\ell_2) (h,\ell_1\ell_2)^{\hf}h^{\hf }\log^2(3h).
\ee
Alternatively, for $(h,\ell_1\ell_2N_0) =1$, one can factor $\rho_j(-h) = \rho_j(-1)\lambda_j(h)$, where $\lambda_j(h)$ is the Hecke eigenvalue at $h$ of the new form associated to $u_j$, of level dividing $N$.   Note that if $u_j$ is an old form it is possible that  $\rho_j(-1) = 0$.  One then has
\be\label{rhosum1}
 \sum_{|t_j| \sim T} |\overline{\rho_j(-h)}|^2e^{-\pi |t_j|} \ll h^{2\theta} \sum_{|t_j| \sim T} |\overline{\rho_j(-1)}|^2e^{-\pi |t_j|} \ll
 h^{2\theta} \ell_1\ell_2T^2,
\ee
where $\theta$ represents the best progress toward the (non-archimedean) Ramanujan conjecture.

Proposition~\ref{prop:Rezresult},  Cauchy-Schwarz,  and \eqref{rhosum0} give us
\begin{align}\label{partbound}
&\sum_{|t_j| \sim T}\overline{\rho_j(-h)}\overline{\<V,u_j\>} 
\ll \left( \sum_{|t_j| \sim T} |\rho_j(-h)|^2e^{-\pi |t_j|}\right)^{\hf}
\left( \sum_{|t_j| \sim T}|\<V,u_j\>|^2e^{\pi |t_j|} \right)^{\hf}
\\ \notag & \ll \left(  \ell_1\ell_2T^2 +
 T^{7/4} \sigma_0(\ell_1\ell_2) (h,\ell_1\ell_2)^{\hf}h^{\hf }\log^2(3h) \right)^{\hf}
\left( (\ell_1 \ell_2)^{-k}T^{2k}\log(T)\right)^{\hf}
   \\ \notag & \ll  (\ell_1 \ell_2)^{\hf-\kt}T^{1+k}\log(T)^{\frac12}+(\ell_1 \ell_2)^{-\kt}T^{7/8+k}\left( \sigma_0(\ell_1\ell_2) (h,\ell_1\ell_2)^{\hf}h^{\hf }\log^2(3h)\log(T) \right)^{\hf}.
\end{align}

Note that for $\ell_1,\ell_2$ prime, $\ell_1,\ell_2 \gg h^{1/4}\log(3h)$, with $(h,\ell_1\ell_2)=1$, the first of the two terms on the right hand side of \eqref{partbound} dominates.

We summarize the above discussion in the following proposition. 
\begin{proposition}\label{prop:innerup}
For $(h,N)=1$,
$$
\sum_{|t_j| \sim T}\overline{\rho_j(-h)}\overline{\<V,u_j\>} \\
   \ll  h^\theta(\ell_1 \ell_2)^{\hf-\kt}T^{1+k}\log(T)^\frac12.
$$
In addition,
\begin{multline*}
\sum_{|t_j| \sim T}\overline{\rho_j(-h)}\overline{\<V,u_j\>} \\
   \ll  (\ell_1 \ell_2)^{\hf-\kt}T^{1+k}\log(T)^\frac12+(\ell_1 \ell_2)^{-\kt}T^{7/8+k}\left( \sigma_0(\ell_1\ell_2) (h,\ell_1\ell_2)^{\hf}h^{\hf }\log^2(3h)\log(T) \right)^{\hf}.
\end{multline*}
\end{proposition}  
Similarly, Proposition~\ref{prop:Rezresult},  Cauchy-Schwarz, and \eqref{Blomer} imply that 
\begin{multline}\label{conttriplebound}
\int_{-T}^T\left|\sum_\cusp \frac{ \vol 2\pi^{\frac{1}{2}-it} |h|^{-it} \rho_\cusp\left(\frac{1}{2}-it, -h\right)
\overline{\<V,E_\cusp(z,\thf+it)\>}}{\Gamma(\hf - it)\zeta(1+2it)}\right| dt  \ll  (\ell_1 \ell_2)^{-\kt+\epsilon}T^{k+\hf+\epsilon}.
\end{multline}

Recalling \eqref{otherhand}, \eqref{ujP} and \eqref{EP},
we define $\cI_{\ell_1,\ell_2,\delta}(s;{\sh})$ for $\Re \, s > \hf$ by
\be\label{Isum}
\cI_{\ell_1,\ell_2,\gd}(s;{\sh}) = \cI_{\ell_1,\ell_2,\gd}^{\,\text{cusp}}(s;{\sh})+ \cI_{\ell_1,\ell_2,\gd}^{\,\text{cts}}(s;{\sh}),
\ee
where
\be\label{YinftyDefcusp}
\cI_{\ell_1,\ell_2,\gd}^{\,\text{cusp}}(s;{\sh})  = 
\sum_j\frac{\overline{\rho_j(-h)}}{\vol(2\pi\sh )^{s-\hf}} 
  M(s,t_j,\delta)\overline{\<V,u_j\>}
\ee
and
\begin{multline}\label{YinftyDefcont}
\cI_{\ell_1,\ell_2,\gd}^{\,\text{cts}}(s;{\sh})
= \frac{1}{4 \pi}\sum_\cusp\int_{-\infty}^\infty  \frac{1}{\vol(2\pi\sh )^{s-\hf} } 
\left(\vol \frac{2\pi^{\frac{1}{2}-it} |h|^{-it} \rho_\cusp\left(\frac{1}{2}-it, -h\right)}{\Gamma\left(\frac{1}{2}-it\right)}\right)\\ \times  M(s,t,\delta)\overline{\<V,E_\cusp(*,\thf + it)\>}dt.
\end{multline}
By the first upper bound of Proposition~\ref{prop:Mbounds}, \eqref{upper1}, for any fixed $s$ and $\delta$, with $s$ at least a distance $\epsilon$ from the poles, $M(s,t,\delta)$ decays faster than any power of $|t|$ as 
$|t| \rightarrow \infty$. Using this, another application of Proposition~\ref{prop:Rezresult} and Cauchy-Schwarz, we see that the integral over $t$ in \eqref{YinftyDefcont} converges for any fixed $s$ such that $\Re \, s \neq \hf-n$ for $n \geq 0$.
However, the integral representation of $\cI_{\ell_1,\ell_2,\gd}^{\,\text{cts}}(s;{\sh})$ only holds for when $\Re \, s > \hf$, and so we investigate its meromorphic continuation to all $\Re \, s > \hf-A$. Changing variables from $t \to z/i$, we rewrite $\cI_{\ell_1,\ell_2,\gd}^{\,\text{cts}}(s;{\sh})$ as
\begin{align}
&\cI_{\ell_1,\ell_2,\gd}^{\,\text{cts}}(s;{\sh})  \label{eq3}\\
&=  \sum_\cusp \frac{1}{2\pi i} \int_{(0)} \frac{1}{\vol(2\pi\sh )^{s-\hf} } \left(\vol \frac{\pi^{\frac{1}{2}-z} |h|^{-z} \rho_\cusp\left(\frac{1}{2}-z, -h\right)}{\Gamma\left(\frac{1}{2}-z\right)\zeta^*(1+2z)}\right)M(s,z/i,\delta)\ol{\langle V,E^*_\cusp(*,\thf+z) \rangle}  \ dz. \notag
\end{align}
where $\zeta^*(z)=\pi^{-\frac{z}{2}}\Gamma(\tfrac{z}{2})\zeta(z)$ is the completed Riemann zeta function and  $E^*_\cusp(s,z)=\zeta^*(2z)E_\cusp(s,z)$. Consider $s$ in the region of $\C$ to the right of $\Re \,s =\hf$ and left of $\hf +C$, where $C$ is a curve of the form $C(x)=\frac{c}{\log(2+|x|)}+ix$ for some constant $c>0$ such that $\zeta^*(1+2s)$ has no zeros on, or to the right of, $-\overline{C}$.  We can then shift the line of integration in \eqref{eq3} to the left, moving over a pole at $z=\hf-s$ to get
\begin{multline}\label{shift1}
	\cI_{\ell_1,\ell_2,\gd}^{\,\text{cts}}(s;{\sh})
	=   
	\\
	\sum_\cusp  
	\lt[  \frac{1}{2\pi i} 
	\int_{-\ol{C}} 
	 \frac{1}{\vol(2\pi\sh )^{s-\hf} } 
	\left(\vol \frac{\pi^{\frac{1}{2}-z} |h|^{-z} \rho_\cusp\left(\frac{1}{2}-z, -h\right)}
	{\Gamma\left(\frac{1}{2}-z\right)\zeta^*(1+2z)}\right)M(s,z/iŒ,\delta)
	\ol{\langle V,E^*_\cusp(*,\thf+z) \rangle}  \ dz \rt.  \\
	+\lt. 
	 \frac{1}{\vol(2\pi\sh )^{s-\hf} } 
	\left(\vol \frac{\pi^{s} |h|^{-\hf+s} \rho_\cusp\left(s, -h\right)}
	{\Gamma\left(s\right)\zeta^*(2-2s)}\right) 
	\frac{\sqrt\pi 2^{\hf-s} \Gamma(2s-1)}{\Gamma(s)}
	\ol{\langle V,E^*_\cusp(*,1-s) \rangle}\rt]. 
\end{multline}
The residue is obtained from \eqref{ires2} in Proposition \ref{prop:Mbounds}. We see that \eqref{shift1} defines a meromorphic function for $s$ between $\hf-C$ and $\hf +C$. So shifting $s$ left such that it is between $\hf-C$ and $(0)$, we can shift the line of integration for $z$ back to $0$ past the pole at $z=s-\hf$. This gives us
\begin{multline*}
	\cI_{\ell_1,\ell_2,\gd}^{\,\text{cts}}(s;{\sh})
	= \\
	\frac{1}{\vol(2\pi\sh )^{s-\hf} } 
 	\sum_\cusp  
	\lt[  \frac{1}{2\pi i} \int_{(0)} 
	\left(\vol \frac{\pi^{\frac{1}{2}-z} |h|^{-z} \rho_\cusp\left(\frac{1}{2}-z, -h\right)}
	{\Gamma\left(\frac{1}{2}-z\right)\zeta^*(1+2z)}\right) 
	M(s,z/iŒ,\delta)\ol{\langle V,E^*_\cusp(*,\thf+z) \rangle}  \ dz \rt.  
	\\
	+ \lt(\vol
	\frac{ \pi^{1-s} |h|^{-s+\hf} \rho_\cusp \left(1-s, -h\right) } 
	{ \Gamma(1-s) \zeta^*(2s) }  		
	\overline{\langle V, E_\cusp^*\left(*, s\right)\rangle }
	+
	\vol
	\frac{\pi^{s} |h|^{-\hf+s} \rho_\cusp\left(s, -h\right)}
	{\Gamma\left(s\right)\zeta^*(2-2s)} 
	\ol{\langle V,E^*_\cusp(*,1-s) \rangle}
	\rt)
	\\
	\lt. \times 
	\frac{\sqrt\pi 2^{\frac{1}{2}-s} \Gamma(2s-1)}{\Gamma(s)}
	\rt]   
\end{multline*}
for $-\hf<\Re\, s<\hf$. This argument can be reproduced as we pass over every line at $s=\hf-n$, and so proceeding recursively, we get that for $-\hf-n < \Re \, s \leq \hf -n$, 
\begin{align} \label{icts}
&	\cI_{\ell_1,\ell_2,\gd}^{\,\text{cts}}(s;{\sh}) &
	\\
&	=\notag
	 \frac{1}{\vol(2\pi\sh )^{s-\hf} } 
   \sum_\cusp  & \mkern -10  mu
   \lt[  
   \frac{1}{2\pi i} 
   \int\limits_{-\overline{C_\sigma}} 
   \left(\vol \frac{\pi^{\frac{1}{2}-z} |h|^{-z} \rho_\cusp \left(\frac{1}{2}-z, -h\right)}{\Gamma\left(\frac{1}{2}-z\right)\zeta^*(1+2z)}\right) 
   M(s,z/iŒ,\delta)
   \ol{\langle V,E^*_\cusp (*,\thf+z) \rangle}  
   \ dz \rt. &
   \\ && \notag
	+ 
	\sum_{\ell=0}^{\lfloor\hf-\sigma\rfloor} 
	\lt(
	\vol 
	\frac{ \pi^{1-s-\ell} |h|^{-s-\ell +\hf} \rho_\cusp \left(1-s-\ell, -h\right) } { \Gamma(1-s-\ell) \zeta^*(2s+2\ell) }   
	\overline{\langle V, E_\cusp^*\left(*, s+\ell \right)\rangle } 
	\rt. &
	\\ && \notag 
	+ 
	\lt. 
	(1-\pmb{\delta}_{\sigma, \ell} )
	\vol 
	\frac{\pi^{s+\ell } |h|^{-\hf+s+\ell } \rho_\cusp \left(s+\ell , -h\right)}
	{\Gamma\left(s+\ell \right)\zeta^*(2-2s-2\ell )} 
	\ol{\langle V,E^*_\cusp(*,1-s-\ell ) \rangle}
	\rt) &
	\\ && \notag
	\lt. 
	\times
	\frac{(-1)^{\ell} \sqrt \pi 2^{\hf-s} \Gamma(1-s) \Gamma(2s+\ell-1)}
	{\ell! \Gamma(s+\ell) \Gamma(1-s-\ell)}
	\lt(1+\Or_\sigma \lt(\delta\rt)\rt)\rt]   &
\end{align}

where
 \beq \label{csigma}
 C_\sigma = \lt\{
 \begin{array}{ll}
 (0) & \mbox{ when } \thf-\sigma \notin \Z_{\geq 0} \\
 C & \mbox{ when } \thf -\sigma \in \Z_{\geq 0}
 \end{array}
\rt.
\eeq
and
 \beq \label{deltasigma}
 \pmb{\delta}_{\sigma,\ell} = \lt\{
 \begin{array}{ll}
 0 & \mbox{ when } \ell \neq  \thf-\sigma \\
 1 & \mbox{ when } \ell = \thf -\sigma .
 \end{array}
\rt.
 \eeq
The dependence on $\sigma$ in the error term arises from our estimates of the residues of $M(s,z/i,\delta)$ in \eqref{ires2} and the dependence on the location of the poles. From what we know about \eqref{ires2} in Proposition \ref{prop:Mbounds}, the error term above also has a meromorphic continuation to all $s$ in the strip we are working in, and we see it has, at most, the same poles as the residual term, though the residues of these poles will vanish as $\delta \to 0$.  


Proposition \ref{prop:innerup} combined with \eqref{upper1}  shows that the right-hand side of \eqref{YinftyDefcusp}  converges.  By the argument given above, $\cI_{\ell_1,\ell_2,\gd}^{\,\text{cts}}(s;{\sh})$ also has a meromorphic continuation to all $\Re s > \hf-A$. 
The same argument, together with  Proposition~\ref{prop:Ylimit}, shows that choosing $A$ sufficiently large will ensure that the expression for $\cI_{\ell_1,\ell_2,Y,\gd}(s;{\sh})$ in \eqref{otherhand} converge for any fixed $Y$, as long as $\Re \, s > \hf+\theta$, and that 
 the difference, $| \cI_{\ell_1,\ell_2,\delta}(s;{\sh}) -  \cI_{\ell_1,\ell_2,Y,\gd}(s;{\sh})|$, converges  as well, allowing the interchange of the limit as $Y \rightarrow \infty$ with the summation.  As the limit of each term in the summation is 0 as $Y \rightarrow \infty$ this shows
that for fixed $\delta > 0$ and $s$ with $\Re \,s> \hf+\theta$,  $ \cI_{\ell_1,\ell_2,\delta}(s;{\sh})$ is an analytic function of $s$ and 
$$ 
  \cI_{\ell_1,\ell_2,\delta}(s;{\sh}) =  \lim_{Y \rightarrow \infty} \cI_{\ell_1,\ell_2,Y,\gd}(s;{\sh}).
$$

Recall that in \eqref{Idef} we had an expression for $ \cI_{\ell_1,\ell_2,Y,\gd}(s;{\sh})$ as a Dirichlet series.   For $\Re \, s >1$ it is clear that the sum converges and that a limit exists as 
$Y \rightarrow \infty$, namely the analytic function

\beann
 \lim_{Y \rightarrow \infty} \cI_{\ell_1,\ell_2,Y,\gd}(s;{\sh}) 
&=&
\frac{1}{\vol} (4\pi)^{-(s+\wt-1)} \mkern -20 mu
\sum_{m_1\ell_1= m_2\ell_2 + {h}} 
{a(m_1) \bar b(m_2)
\over
 (m_2\ell_2)^{s+\wt-1}
}
\int_{0}^{\infty}
 e^{- y\left(1 +  \frac{\delta h}{ 2 \ell_2 m_2}  \right)  } 
y^{s+k-1} 
{ dy\over y}
\\
&=& 
\frac{\G(s+k-1)}{(4 \pi)^{s + k -1}\vol }D(s;h,\delta),
\eeann
where
\be\label{Dlim}
 D(s;\sh,\delta):= \sum_{m_1\ell_1= m_2\ell_2 + \sh} 
{a(m_1) \bar b(m_2)
\over
 (m_2\ell_2)^{ s+\wt-1}}\left(1 + \frac{\delta \sh}{2 \ell_2 m_2} \right)^{-(s+k-1)}
 \ee
converges for $\Re \, s >1$.

We have now established that the function $\cI_{\ell_1,\ell_2,\delta}(s;{\sh}) $ defined by  \eqref{Isum} is an infinite sum that converges when $\Re \, s > \hf$ and is related to 
$D(s;\sh,\delta)$ by
\be\label{righthand}
D(s;\sh,\delta):=  \frac{(4 \pi)^{s+k-1} \vol}{\G(s+k-1)} \cI_{\ell_1,\ell_2,\delta}(s;{\sh}),
 \ee
when $\Re \, s > 1$.   

The meromorphic continuation of $D(s;\sh,\delta)$ is related to that of 
$ \cI_{\ell_1,\ell_2,\delta}(s;{\sh})$ via \eqref{righthand}.  Using \eqref{Isum},\eqref{YinftyDefcusp},
\eqref{YinftyDefcont} we obtain a spectral expansion when  $\Re \, s=\sigma >\hf -A$.  
 For $\Re(s)>\frac{1}{2}$, we have
\beq \label{Dspec1}
D(s; h, \delta) = D_{\rm cusp}(s; h, \delta) +D_{\rm cts}(s; h, \delta),
\eeq
where 
	\beq\label{cuspdelta}
	D_{\rm cusp}(s; h, \delta)
	:= 
	\frac{(4\pi)^k 2^{s-\frac{1}{2}}}
	{2\sqrt{\pi}\Gamma(s+k-1) h^{s-\frac{1}{2}}}
	\sum_j \overline{\rho_j(-h)} M(s, t_j, \delta) \overline{\left<V, u_j\right>}
	\eeq
	and


	\be\label{Dcontextend}
	D_{\rm cts}(s; h, \delta): 
	= D_{\rm int}(s; h, \delta)+\Omega(s; h, \delta),
	\eeq
where, as with \eqref{icts}, we have that  $D_{\rm int}(s; h, \delta)$ is the integral component,
\begin{multline} \label{Dspec} 
	D_{\rm int}(s; h, \delta)
	:=   
	\frac{(4 \pi)^{k}2^{s-\hf}}
	{2\sqrt{\pi}\Gamma(s+k-1) h^{s-\hf}}
	\\ 
	\times
	 \sum_\cusp   
	 \frac{1}{2\pi i} 
	 \int_{-\ol{C_\sigma}} 
	 \frac{ \vol 
	 \pi^{\frac{1}{2}-z} \rho_\cusp\left(\frac{1}{2}-z, -h\right)
	 |h|^{-z}}
	 {\Gamma\left(\frac{1}{2}-z\right) \zeta^*(1+2z)}
	 M(s,z/i,\delta)
	 \ol{\langle V,E^*_\cusp(*,\thf+z) \rangle}  \ dz
 	\end{multline}
	 and $\Omega(s; h, \delta)$ arises from the residues of poles,
	 \begin{multline}\label{Omega}
	\Omega(s; h, \delta):=   
	\frac{(4 \pi)^{k}}
	{2\Gamma(s+k-1) }
	\\
	\times
	\sum_\cusp 
	\lt[
	\sum_{\ell=0}^{\lfloor\hf-\sigma\rfloor} 
	\lt(
	\frac{ \pi^{1-s-\ell} h^{1-2s-\ell} \vol \rho_\cusp \left(1-s-\ell, -h\right) } { \Gamma(1-s-\ell) \zeta^*(2s+2\ell) }   
	\overline{\langle V, E_\cusp^*\left(*, s+\ell \right)\rangle } 
	\rt.
	\rt.
	\\
	+ 
	\lt. 
	(1-\pmb{\delta}_{\sigma, \ell} )
	\frac{\pi^{s+\ell } h^{\ell } \vol \rho_\cusp \left(s+\ell , -h\right)}
	{\Gamma\left(s+\ell \right)\zeta^*(2-2s-2\ell )} 
	\ol{\langle V,E^*_\cusp(*,1-s-\ell ) \rangle}
	\rt)
	\\
	\lt. 
	\times
	\frac{(-1)^{\ell} \Gamma(1-s) \Gamma(2s+\ell-1)}
	{\ell! \Gamma(s+\ell) \Gamma(1-s-\ell)}
	\lt(1+\Or_\sigma\lt(\delta\rt)\rt)
	\rt],
	\,
	\end{multline}
	and is zero if $\sigma>\hf$. 
	Here $C_\sigma$ and $\pmb{\delta}_{\sigma,\ell}$ are as in \eqref{csigma} and \eqref{deltasigma}. 

By \eqref{upper1},  \eqref{partbound}, \eqref {conttriplebound}, \eqref{Isum},
for any sufficiently large fixed $A \gg 1$, the series expressions for  
$\cI_{\ell_1,\ell_2,\gd}^{\,\text{cusp}}(s;{\sh})$, $\cI_{\ell_1,\ell_2,\gd}^{\,\text{cts}}(s;{\sh})$, and $D(s;\sh,\delta)$ given in 
\eqref{YinftyDefcusp},\eqref{YinftyDefcont}, and  \eqref{Dspec1}  converge for  $s$ with $\Re \, s > \hf -A$, but the upper bound can have a factor of $\delta^{-A}$ in it.   As we intend to let $\delta \rightarrow 0$ it is useful to identify an interval for  $\Re \, s$ in which there is absolute convergence with an upper bound that is independent of $\delta$.

Take $\Im \, t_j =0$. Combining \eqref{deltasmall} and \eqref{deltabig} we have, for $\hf-A<\Re \, s = \sigma  \le \hf$, regardless of the relation between $\delta$ and $(1 + |t|)^{-2}$,

\be\label{Mdelta}
M(s,t_j,\delta) \ll_{A,\epsilon}  (1 + |t_j|)^{ 2\sigma-2 + 2 \epsilon}(1 + |s|)^{ 4A-3\sigma+3 }e^{-\frac{\pi}{2}|\Im \, s|},
\ee
when $s$ is at least $\epsilon$ away from the poles at $s=\hf\pm it_j -r$ for $r \geq 0$. If $\Im t_j \neq 0$ for some $t_j$ then by the same reasoning a comparable bound holds in the $s$ aspect and we are unconcerned with the growth in the $t_j$ aspect from a finite number of terms as they will not affect convergence of the sum over $t_j$. 
Combining this with Proposition~\ref{prop:innerup} and noting that $t_j \in \R$ for sufficiently large $j$, we have, for any 
$T, |s| \gg 1$, with  $\min_{m\in \Z}|s-\hf\pm t_j-m|>\epsilon$,
\begin{align}\label{cuspidalsum}
&\left|\sum_{|t_j|\sim T} \overline{\rho_j(-h)}  M(s,t_j,\delta)\overline{\<V,u_j\>}\right|  \notag \\
& \ \ \ \ll_{A,\epsilon} (1 + |s|)^{ 4A-3\sigma+3 }e^{-\frac{\pi}{2}|\Im \, s|} \left[
    (\ell_1 \ell_2)^{\hf-\kt}T^{k + 2\sigma -1 +2\epsilon}\log(T)^\frac12\right. \\ &\left. +(\ell_1 \ell_2)^{-\kt}T^{-9/8+k+2\sigma +2\epsilon}\left( \sigma_0(\ell_1\ell_2) (h,\ell_1\ell_2)^{\hf}h^{\hf }\log^2(3h)\log(T) \right)^{\hf} \right].\notag 
\end{align}
When $T$ is small, a similar bound holds in all aspects except for $T$. Indeed \eqref{Mdelta} is generally unchanged in the $s$-aspect, with some polynomial growth dependent on $A$ and the same exponential decay, Furthermore, Proposition~\ref{prop:innerup} is uniform and thus unchanged in the $\ell_1,\ell_2$ and $h$ aspects even when $T$ is small.

Applying \eqref{conttriplebound} in place of the proposition, we obtain   
\begin{multline}\label{contsum}
\int_{-T}^T
	\sum_\cusp \frac{\vol \rho_\cusp\left(\hf-it, -h\right)} {\zeta(1+2it)}  M(s,t,\delta) \overline{\<V,E_\cusp(*,\thf + it)\>}dt  \\ \ll
(\ell_1 \ell_2)^{-\kt+\epsilon}T^{2\sigma-3/2+k+2\epsilon}e^{-\frac{\pi}{2}|\Im \, s|}
(1+|s|)^{4A -3\sigma+3}
\end{multline}
when once more $s$ is at least $\epsilon$ away from any poles of $M(s,t,\delta)$. 

Breaking up the interval from 1 to $T$ dyadically and letting $T \rightarrow \infty$ it is clear from
 \eqref{cuspidalsum} that 
the cuspidal part of the spectral expansion of $D(s;\sh,\delta)$ converges when 
$\hf -A<\Re\, s < \hf -\kt -\epsilon$.  Referring to \eqref {contsum}, we see that the continuous part also converges in this region and satisfies a $\delta$-independent upper bound.   Thus  \eqref{cuspidalsum}, together with \eqref{contsum}, \eqref{Dspec} and Stirling's formula provide the upper bound for $D(s;\sh,\delta)$ of Proposition~\ref{prop:Imero} stated below in \eqref{semifinalupper2}.

  The details of the above discussion are summarized, and some additional information is added, in the following
\begin{proposition}
\label{prop:Imero}
The function  $D(s;\sh,\delta)$, defined in \eqref{Dlim}, has a meromorphic continuation to the half plane $\Re \, s >\hf -A$.
For $\Re(s)>\frac{1}{2}$, we have the decomposition
$$
D(s; h, \delta) = D_{\rm cusp}(s; h, \delta) +D_{\rm cts}(s; h, \delta),
$$
where $D_{\rm cusp}(s; h, \delta) $ and $D_{\rm cts}(s; h, \delta)$ are as in \eqref{cuspdelta} and \eqref{Dcontextend} respectively, are locally normally convergent as sums and integrals away from the poles
	of $M(s,t,\delta)$ and have meromorphic 
continuations  to  $\Re(s) >\frac{1}{2}-A$.

 If $t_j \neq 0$ for all $j$,   $D_{\rm cusp}(s; h, \delta)$ has simple poles at  $s = \hf \pm it_j - r$ for each $t_j$ and each $0 \le r < A \mp \Im \, t_j$. Furthermore $D_{\rm cts}(s; h, \delta)$ has poles when $s =\hf -r $ for $A>r \geq 1$, and at $s=\frac \vho 2-r$ for $A+\frac{\Re(\vho)-1}{2}>r\geq 0$ where $\vho$ is a non-trivial zero of $\zeta(s)$. Indeed, let
$$
R(s_0;h,\delta)= \Res_{s = s_0}D(s;\sh,\delta)
$$
then we have
\beann
&&R(\thf \pm it_j - r;\sh,\delta) \\
&=& \frac{(-1)^r(4\pi)^k h^{r \mp it_j}\overline{\rho_j(-h)}\overline{\<V,u_j\>}\Gamma(\hf \mp it_j +r)\Gamma(\pm 2it_j -r)}{2r!\Gamma(\hf + it_j)\Gamma(\hf-i t_j) \Gamma(k-\hf \pm it_j -r)} \\
&+& \mathcal{O}\left((1 + |t_j|)^{r}h^{r}\overline{\rho_j(-h)}\overline{\<V,u_j\>}\delta^{\hf}\right).
\eeann
If there does exist some $t_j=0$ then we also have double poles at $s=\hf-r$ for $r \geq 0$. 


In the region  $ \hf -A + \epsilon< \Re \, s =\sigma < \hf -\kt $, with $s$ at least a distance $\epsilon>0$ from poles $\hf \pm it_j-r$,
 and $|\Im\, s|\gg 1$,  $D(s;h,\delta)$ satisfies the upper bound 
\begin{multline} 
D(s;\sh,\delta)-\Omega(s; h, \delta) \ll (1 + |s|)^{A'}  (\ell_1 \ell_2)^{(1-k)/2} h^{\hf-\sigma}\\ \times
\max\left(1, \frac{(h, \ell_1\ell_2)^{1/4}\sigma_0(\ell_1\ell_2)^\hf h^{1/4}\log(3h)}{(\ell_1\ell_2)^\hf}\right). \label{semifinalupper2} 
\end{multline}
Here $A'$  depends only on $A$.   Also, when $(h,N)=1$, the term $\max\left(*\right)$ can be replaced by $h^\theta$, where $\theta$ represents the best progress toward the Ramanujan conjecture.

For $2 \ge \sigma \ge \hf-A$ the upper bound in the above expression is multiplied by $\delta^{-A}$.

For $ \sigma >1 + \epsilon,$ $D(s;\sh,\delta)$ satisfies the bound
\be
\label{1bound}
D(s;h,\delta) \ll \frac{\left(1 + \frac{h}{\ell_2}\right)^{(k-1)/2}}{\left(\ell_1 \ell_2\right)^{ (k-1)/2 }\ell_2^{1 + \epsilon}}(1+|s|)^{A'}.
\ee
\end{proposition}

\begin{proof}
Proposition~\ref{prop:Mbounds}, \eqref{upper1},  states that for sufficiently large $A$, and fixed 
$\delta>0$ and $s$ satisfying $\Re \, s >\hf -A$ at least a distance $\epsilon$ from the poles, the meromorphic continuation of  $M(s,t,\delta)$ decays faster than any power of $|t|$. 
Thus for any fixed $\delta > 0$ and such an $s$ with $\Re \, s >\hf -A$, the series expression for 
$ \cI_{\ell_1,\ell_2,\delta}(s;{\sh})$ given in \eqref{Isum} converges, giving a meromorphic continuation of  $\cI_{\ell_1,\ell_2,\delta}(s;{\sh})$ back to $\Re \, s > \hf-A$, with possible poles at the points specified.   The corresponding meromorphic continuation of $D(s;h,\delta)$ follows from \eqref{righthand}.  The residues at these points of $ \cI_{\ell_1,\ell_2,\delta}(s;{\sh})$, and consequently of $D(s;\sh,\delta)$,  follow from the corresponding residue evaluations of $M(s,t,\delta)$ given in Proposition~\ref{prop:Mbounds}, as well as by examining properties of \eqref{Dspec}.

We observe that the piece corresponding to the continuous part of the spectrum,
$\cI_{\ell_1,\ell_2,\gd}^{\,\text{cts}}(s;{\sh})$, has poles due to the terms
$$
\frac{\Gamma(2s+\ell-1)}{\zeta^*(2s+2\ell)\Gamma(s+\ell)}
$$
and  it is a straightforward exercise to compute these residues.
The bounds on the discrete continuous parts of $D(s;h,\delta)$ for $\hf -A <\sigma <\hf- \kt $ have been explained above.  The bound for $\sigma > 1+ \epsilon$ is immediate.
\end{proof}

\section{The limit as $\delta \rightarrow 0$} \label{sec:5}
The object of the section is to prove the following proposition.
\begin{proposition}
\label{prop:D0}
For $\Re s >1$, let 
\be\label{Dsumdeltato0}
D(s;h) = \lim_{\delta \rightarrow 0}D(s;h,\delta)=  \sum_{m_1\ell_1= m_2\ell_2 + \sh} 
{a(m_1) \bar b(m_2)
\over (m_2\ell_2)^{ s+\wt-1}}.
\ee
The function $D(s;h)$ converges for $\Re \, s >1$ and has a meromorphic continuation to all $s$ with $\Re s > \hf -A$.  If $t_j \neq 0$ for all $j$, this function has simple poles at  $s = \hf \pm it_j - r$ for each $t_j$ and each $0 \le r < A \mp \Im \, t_j$. Furthermore it has possible poles, due to $\Omega(s;h)$ as in \eqref{omega2}, when $s =\hf -r $ for $A>r \geq 1$, and at $s=\frac \vho 2-r$ for $A+\frac{\Re(\vho)-1}{2}>r\geq 0$ where $\vho$ is a non-trivial zero of $\zeta(s)$. In particular
\be\label{Rdef}
R(\thf + it_j - r;\sh) = \Res_{s = \hf + it_j - r}D(s;h)= c_{r,j} h^{r-it_j}\overline{\rho_j(-h)},
\ee
where now the index $j$ identifies our choice of $+it_j$ or $-it_j$ and
\be\label{crjdef}
c_{r,j} =
 \frac{(-1)^r (4\pi)^k\overline{\<V,u_j\>}\Gamma(\hf - it_j +r)\Gamma(2it_j -r)}{2r!\Gamma(\hf + it_j)\Gamma(\hf-i t_j) \Gamma(k-\hf + it_j -r)}.
\ee
For $T \gg1$, if  $\ell_1,\ell_2$ are prime,  and $\ell_1,\ell_2 \gg h^{1/4}\log(3h)$, with $(h,\ell_1\ell_2)=1$, then the $c_{r,j} $ satisfy the average upper bounds
\be\label{crj2}
 \sum_{|t_j|\sim T}|c_{r,j}|^2e^{\pi |t_j|} \ll\log( T) (\ell_1 \ell_2)^{ -k} T^{2r + 1}.
 \ee
If there exists $t_j=0$, then all of the above is true but $D(s;h)$ instead has double poles at $s=\hf-r$ for $r \geq 0$, $r \in \Z$. 

For $s$ in the vertical strip 
$\hf -A <\sigma < \hf  -\kt$,  $D(s;h)$ can be expressed by the convergent spectral expression given in \eqref{Dspec3} with upper bounds given by \eqref{reallyfinalupper}. For $s$ in the vertical strip 
$c <\sigma < 2$, at least $\epsilon$ away from the poles noted above, for any $c$ with $\hf - A + \epsilon <c <\hf -\kt$, we have that $D(s;h)$ is given by the contour integral in \eqref{Cauchy0} and satisfies the upper bound \eqref{critbound}.

\end{proposition}

\begin{proof}
Suppose $t_j\neq 0$ for all $j$. Using the information in Proposition~\ref{prop:Imero},  the Cauchy residue theorem allows us to express each $D(s;h,\delta)$, with 
$\hf - A + \epsilon < \Re \, s < 2 $,  as follows.  Set $T \gg 1$.   Then when $|\Im \, s| <\frac{T}{2}$, 
\bea\label{Cauchy1}
D_{\rm cusp}(s; h, \delta) &+& \sum_{\substack{|t_j| <T \\ 0 \le r <  \hf\mp \Im(t_j)-c}} \notag
\frac{R(\hf + it_j -r;h,\delta)e^{(\hf + it_j -r-s)^2}}{\hf + it_j -r-s}\\
&=& I_1(T)-I_2(T)-I_3(T)+I_4(T),\
\eea
where $D_{\rm cusp}(s; h, \delta)$ is as in \eqref{cuspdelta} and $$
I_1(T) = \frac{1}{2 \pi i}\int_{2-iT}^{2+iT}\frac{D_{\rm cusp}(u;h,\delta)e^{(u-s)^2}du}{u-s}, 
\ \ 
I_2(T) =  \frac{1}{2 \pi i}\int_{c+iT}^{2+iT}\frac{D_{\rm cusp}(u;h,\delta)e^{(u-s)^2}du}{u-s}, 
$$
$$
I_3(T) =  \frac{1}{2 \pi i}\int_{c-iT}^{c+iT}\frac{D_{\rm cusp}(u;h,\delta)e^{(u-s)^2}du}{u-s}, \ \ 
I_4(T)=\frac{1}{2 \pi i}\int_{c-iT}^{2-iT}\frac{D_{\rm cusp}(u;h,\delta)e^{(u-s)^2}du}{u-s}.
$$ 
We choose $T$ is such that the contour lines for $I_2$ and $I_4$ never get closer to the poles of $D_{\rm cusp}(u;h,\delta)$ than  $\frac{T^{-2}}{c}$ for some $c>0$, which is permitted by Weyl's Law. 
Thus by the upper bounds given in \eqref{nearp1} and Proposition~\ref{prop:Imero}, we have that, for $T$ large
\beq \label{polygrow}
I_2 \ll_{c,N,h,s}T^{A'}e^{-{T}^2}(\delta^{-A}+ cT^2)
\eeq
From this it is obvious that $I_2(T) \rightarrow 0$ as $T \rightarrow \infty$.
Similarly $\lim_{T \rightarrow \infty} I_4(T)= 0$. This gives us the integral expression for $D_{\rm cusp}(s; h, \delta)$,
\begin{align}
\label{Cauchy11}
D_{\rm cusp}(s; h, \delta) =& \frac{1}{2 \pi i}\int_{(2)}\frac{D_{\rm cusp}(u;h,\delta)e^{(u-s)^2}du}{u-s} -\frac{1}{2 \pi i}\int_{(c)}\frac{D_{\rm cusp}(u;h,\delta)e^{(u-s)^2}du}{u-s} \\
& - \sum_{t_j, 0 \le r <  \hf\mp \Im(t_j)-c} \notag
\frac{R(\hf + it_j -r;h,\delta)e^{(\hf + it_j -r-s)^2}}{\hf + it_j -r-s} .
\end{align}

Now we observe that for $\hf< \sigma<2$,
\begin{align}
&  \frac{1}{2 \pi i}\int_{(2)}\frac{D_{\rm cts}(u;h,\delta)e^{(u-s)^2}du}{u-s} \\
& =\sum_\cusp \lt(\frac{1}{2\pi i} \rt)^2  \iint\limits_{(2)(0)} \frac{\mc{V}\pi^{k-z}\rho_\cusp(\hf-z,-h)h^{\hf-u-z}M(u,z/i,\delta)
	 \ol{\langle V,E^*_\cusp(*,\thf+z) \rangle}  }{2^{\frac{3}{2}-2k-u}\Gamma(u+k-1)\Gamma(\hf-z)\zeta^*(1+2z)(u-s)}e^{(u-s)^2}  \ dzdu \notag
\end{align}
where $D_{\rm cts}(u;h,\delta)$ is as in \eqref{Dcontextend}. Since the double integral is convergent in $u$ and $z$, we can change the order of integration and shift the line of integration of $u$ left from $\Re \, u =2$ to $\Re \, u =c$. This gives us that
\begin{align}
&  \frac{1}{2 \pi i}\int_{(2)}\frac{D_{\rm cts}(u;h,\delta)e^{(u-s)^2}du}{u-s} =D_{\rm cts}(s;h,\delta)\\
& \ +\sum_\cusp \lt(\frac{1}{2\pi i} \rt)^2  \iint\limits_{(c)(0)} \frac{\mc{V}\pi^{k-z}\rho_\cusp(\hf-z,-h)h^{\hf-u-z}M(u,z/i,\delta) \ol{\langle V,E^*_\cusp(*,\thf+z) \rangle}  }{2^{\frac{3}{2}-2k-u}\Gamma(u+k-1)\Gamma(\hf-z)\zeta^*(1+2z)(u-s)}e^{(u-s)^2}  \ dzdu \notag \\
& \ +\sum_{r=0}^{\lfloor \hf-c \rfloor}\sum_\cusp \frac{1}{2\pi i}  \int\limits_{(0)} \frac{\mc{V}\pi^{\hf-z}\rho_\cusp(\hf-z,-h)h^{r} \ol{\langle V,E^*_\cusp(*,\thf+z) \rangle}  }{2(-1)^r r! (4\pi)^{-k}\Gamma(\hf+z)\Gamma(\hf-z)^2\zeta^*(1+2z)}\notag \\
&  \ \ \ \ \times  \lt[ \frac{h^{-2z} \Gamma(\hf-z+r)\Gamma(2z-r)e^{(\hf+z-r-s)^2}}{\Gamma(k-\hf+z-r)(\hf+z-r-s)} + \frac{\Gamma(\hf+z+r)\Gamma(-2z-r)e^{(\hf-z-r-s)^2}}{\Gamma(k-\hf-z-r)(\hf-z-r-s)} \rt] \ dz \notag\\
& \ + \mc{O}_{s,c}(\delta). \notag
\end{align}

We note that we can disregard the issue of double poles at $z=0$. Indeed, when we move the line of integration for $u$ near $\Re\, u = \hf-r$, we can shift the line of integration for $z$ slightly to the left of $\Re \, z =0$ into the zero-free region of $\zeta^*(1+2z)$ and then move the line for $u$ past simple poles at $u = \hf\pm z -r$. Once we have moved past these poles we can shift the line of integration for $z$ back to $\Re \, z=0$ without consequence. Furthermore, we observe the integrands of residual terms in the above expression are well-defined when $z=0$ and that these residual terms are also well-defined. 

Rewriting the above equality we get
\begin{align}
&D_{\rm cts}(s;h,\delta) =  \frac{1}{2 \pi i}\int_{(2)}\frac{D_{\rm cts}(u;h,\delta)e^{(u-s)^2}du}{u-s}\\
& \ - \frac{1}{2 \pi i}\int_{(c)}\frac{(D_{\rm cts}(u;h,\delta)-\Omega(u;h,\delta))e^{(u-s)^2}du}{u-s} \notag \\
& \ -\sum_{r=0}^{\lfloor \hf-c \rfloor}\sum_\cusp \frac{1}{2\pi i}  \int\limits_{(0)} \frac{\mc{V}\pi^{\hf-z}\rho_\cusp(\hf-z,-h) \ol{\langle V,E^*_\cusp(*,\thf+z) \rangle}  }{2(-1)^r r! (4\pi)^{-k}\Gamma(\hf+z)\Gamma(\hf-z)^2\zeta^*(1+2z)}\notag \\
& \   \times  \lt[ \frac{h^{r-2z} \Gamma(\hf-z+r)\Gamma(2z-r)e^{(\hf+z-r-s)^2}}{\Gamma(k-\hf+z-r)(\hf+z-r-s)} + \frac{h^r\Gamma(\hf+z+r)\Gamma(-2z-r)e^{(\hf-z-r-s)^2}}{\Gamma(k-\hf-z-r)(\hf-z-r-s)} \rt] \ dz \notag\\
& \ - \mc{O}_{s,c}(\delta), \notag
\end{align}
where $\Omega(s;h,\delta)$ is as in \eqref{Omega}, when $\hf<\sigma<2$. We see that the integrands of the first two integrals have no poles in $s$ when $c<\sigma<2$ so we must only consider the sum of integrals to meromorphically continue the above formulation of $D_{\rm cts}(s;h,\delta)$. Indeed, we see that we only need be concerned with the case when $s-\hf +r$ is near negative integral values. By an argument analogous to the one that gave us the continuation of $D_{\rm cts}(s;h,\delta)$ in Proposition \ref{prop:Imero}, we are able to shift the lines of integration for $z$ left and then back again as the line of integration for $s$ moves past to get
 \begin{align} \label{thing11}
&D_{\rm cts}(s;h,\delta) =  \frac{1}{2 \pi i}\int_{(2)}\frac{D_{\rm cts}(u;h,\delta)e^{(u-s)^2}du}{u-s}\\
& \ - \frac{1}{2 \pi i}\int_{(c)}\frac{(D_{\rm cts}(u;h,\delta)-\Omega(u;h,\delta))e^{(u-s)^2}du}{u-s} \notag \\
& \ -\sum_{r=0}^{\lfloor \hf-c \rfloor}\sum_\cusp \frac{1}{2\pi i}  \int\limits_{C_\sigma} \frac{\mc{V}\pi^{\hf-z}\rho_\cusp(\hf-z,-h) \ol{\langle V,E^*_\cusp(*,\thf+z) \rangle}  }{2(-1)^r r! (4\pi)^{-k}\Gamma(\hf+z)\Gamma(\hf-z)^2\zeta^*(1+2z)}\notag \\
& \   \times  \lt[ \frac{h^{r-2z} \Gamma(\hf-z+r)\Gamma(2z-r)e^{(\hf+z-r-s)^2}}{\Gamma(k-\hf+z-r)(\hf+z-r-s)} + \frac{h^r\Gamma(\hf+z+r)\Gamma(-2z-r)e^{(\hf-z-r-s)^2}}{\Gamma(k-\hf-z-r)(\hf-z-r-s)} \rt] \ dz \notag\\
& \ +\Omega(s;h,\delta) - \mc{O}_{s,c}(\delta) \notag
\end{align}
where $C_\sigma$ is as in \eqref{csigma}. Combining \eqref{Cauchy11} and \eqref{thing11} we get that
\begin{subequations} \label{thing13}
 \begin{align}  \notag
&D(s;h,\delta)  \\
& =  \frac{1}{2 \pi i}\int_{(2)}\frac{D(u;h,\delta)e^{(u-s)^2}du}{u-s} \label{thing111}\\
& \ - \frac{1}{2 \pi i}\int_{(c)}\frac{(D(u;h,\delta)-\Omega(u;h,\delta))e^{(u-s)^2}du}{u-s} \label{thing12} \\
& \ -\sum_{r=0}^{\lfloor \hf-c \rfloor}\sum_\cusp \frac{1}{2\pi i}  \int\limits_{C_\sigma} \frac{\mc{V}\pi^{\hf-z}\rho_\cusp(\hf-z,-h) \ol{\langle V,E^*_\cusp(*,\thf+z) \rangle}  }{2(-1)^r r! (4\pi)^{-k}\Gamma(\hf+z)\Gamma(\hf-z)^2\zeta^*(1+2z)}\notag \\
& \   \times  \lt[ \frac{h^{r-2z} \Gamma(\hf-z+r)\Gamma(2z-r)e^{(\hf+z-r-s)^2}}{\Gamma(k-\hf+z-r)(\hf+z-r-s)} + \frac{h^r\Gamma(\hf+z+r)\Gamma(-2z-r)e^{(\hf-z-r-s)^2}}{\Gamma(k-\hf-z-r)(\hf-z-r-s)} \rt] \ dz \notag\\
& \ +\Omega(s;h,\delta) - \sum_{t_j, 0 \le r <  \hf\mp \Im(t_j)-c} \notag
\frac{R(\hf + it_j -r;h,\delta)e^{(\hf + it_j -r-s)^2}}{\hf + it_j -r-s} 
\end{align}
\end{subequations}

By \eqref{Dlim}, we see that when $\Re ( u) >1$, the limit as $\delta \rightarrow 0$ of the integrand of \eqref{thing111} exists.  By the absolute convergence of $D(\sigma + it;h)$ for $\sigma >1$ and the exponential decay of $e^{u^2}$, the limit as $\delta \rightarrow 0$ of \eqref{thing111} exists and the limit can be interchanged with the integral.

Turning to \eqref{thing12} and referring to \eqref{deltasmall} we see that for any fixed $t$, and $\Re \, s \le \hf+|\Im\, t|$, $\lim_{\delta \rightarrow 0} M(s,t,\delta)$ exists and 
\begin{align}\label{M0def}
M(s,t) &= \lim_{\delta \rightarrow 0} M(s,t,\delta)=\frac{\sqrt{\pi} 2^{\hf-s}\G(s-\hf -it) \G(s-\hf + it) \G(1-s)}{\G(\hf - it)\G(\hf + it)}. 
\end{align}
Referring to Proposition~\ref{prop:Imero}, line \eqref{semifinalupper2}, it is clear that in strip  $\hf -A + \epsilon<\sigma < \hf -\kt$  the limit as  $\delta \rightarrow 0$ of $D(s;h,\delta)$ exists,  equals 
$D(s;h)$ as defined by

\begin{align}\label{Dspec3}
 D(s;\sh)
	=&  \frac{(4 \pi)^{k}}{2\Gamma(s+k-1)h^{s-\hf}} \left(\sum_j
	\overline{\rho_j(-h)}   \frac{\G(s-\hf -it_j) \G(s-\hf + it_j) \G(1-s)}{\G(\hf - it_j)\G(\hf + it_j)}	\overline{\<V,u_j\>} \right. \notag \\
&   +
   \sum_\cusp  
   \lt[  \frac{1}{2\pi i} 
   \int_{-\ol{C_\sigma}} 
   \frac{\vol 2\pi^{\hf-z} |h|^{-z} \rho_\cusp\left(\hf-z, -h\right) }   
   {\G(\hf-z) \zeta^*(1+2z) } 
   \rt. \notag
   \\
 &  \lt.\lt.
   \times
   \frac{ \G(s-\hf -z) \G(s-\hf + z) \G(1-s)}
   {\G(\hf - z)\G(\hf + z)}
   \ol{\langle V,E^*_\cusp(*,\thf+z) \rangle}  \ dz 
   \rt]\rt) 
   	+ 
	\Omega(s;\sh),
\end{align}
where
\begin{align}
& \Omega(s; \sh) \label{omega2}
=\frac{(4 \pi)^{k}}{2\Gamma(s+k-1)}  \sum_\cusp \sum_{\ell=0}^{\lfloor\hf-\sigma\rfloor}  \frac{(-1)^{\ell}\Gamma(1-s) \Gamma(2s+\ell-1)}
	{\ell! \Gamma(s+\ell) \Gamma(1-s-\ell)}
	\\ \notag
	& \mkern150mu \times 
	\lt[
		\frac{ \pi^{1-s-\ell} h^{1-2s-\ell} \vol \rho_\cusp \left(1-s-\ell, -h\right) } { \Gamma(1-s-\ell) \zeta^*(2s+2\ell) }   
	\overline{\langle V, E_\cusp^*\left(*, s+\ell \right)\rangle } 
	\rt.
	\\
	& \lt. \mkern200mu + \notag
	(1-\pmb{\delta}_{\sigma, \ell} )
	\frac{\pi^{s+\ell } h^{\ell} \vol \rho_\cusp \left(s+\ell , -h\right)}
	{\Gamma\left(s+\ell \right)\zeta^*(2-2s-2\ell )} 
	\ol{\langle V,E^*_\cusp(*,1-s-\ell ) \rangle} \rt],	
	\end{align}
	and where  $\pmb{\delta}_{\sigma,\ell}$ is given in  \eqref{deltasigma}.

Furthermore, in this range $D(s;\sh)$ satisfies the upper bound 
\begin{multline}\label{reallyfinalupper}
D(s;\sh)-\Omega(s;\sh)\ll_{A,\epsilon} (1 + |s|)^{A'}  (\ell_1 \ell_2)^{(1-k)/2} h^{\hf-\sigma}
\\ \times
\max\left(1, \frac{(h, \ell_1\ell_2)^{1/4}\sigma_0(\ell_1\ell_2)^\hf h^{1/4}\log(3h)}{(\ell_1\ell_2)^\hf}\right),
\end{multline}
where $|\Im \, s|$ is sufficiently large, $A'$ is determined by $A$ and $s$ is at least $\epsilon$ away from the poles at $\hf \pm it_j-r$. Also, when $(h,N)=1$, the term $\max\left(*\right)$ can be replaced by $h^\theta$, where $\theta$ represents the best progress toward the Ramanujan conjecture.
   The exponential decay of $D(s;h)$ in the imaginary part of $s$ shown in \eqref{reallyfinalupper}, together with the square-exponential decay of $e^{u^2}$ shows that for 
$\hf - A + \epsilon < \Re \, s < \hf - \kt $,   the limit as $\delta\rightarrow 0$ of $\eqref{thing12}$ exists and equals the integral of the limit. 

By the description in Proposition~\ref{prop:Imero}, each residue term 
$R(\hf + it_j -r;h,\delta)$ grows at most polynomially in $|t_j|$ for any fixed $s$ with 
$\hf - A + \epsilon < \Re \, s < 2 $.  
As $e^{(\hf + it_j -r-s)^2}$ decays square-exponentially in $|t_j|$, the sum over the  residues converges and approaches the limits given above in \eqref{Rdef} and \eqref{Cauchy0}.   
We have thus established that for any $s$ in the range $\hf - A + \epsilon <\Re \, s <2$, the limit as $\delta \rightarrow 0$ of $D(s;h,\delta)$ exists. Indeed, for $s$ in the vertical strip 
$c <\sigma < 2$, at least $\epsilon$ away from the poles noted above for any $c$ with $\hf - A + \epsilon <c <\hf -\kt$ then the limit, $D(s;h)$, is given by
\begin{align}\label{Cauchy0}
&D(s;h) =\\
&\frac{1}{2 \pi i}\int_{2-i\infty}^{2+i\infty}\frac{D(u;h)e^{(u-s)^2}du}{u-s} - \frac{1}{2 \pi i}\int_{c-i\infty}^{c+i\infty}\frac{\lt( D(u;h) - \Omega(u;h)\rt)e^{(u-s)^2}du}{u-s}\nonumber\\
 &-\sum_{r=0}^{\lfloor \hf-c \rfloor}\sum_\cusp \frac{1}{2\pi i}  \int\limits_{C_\sigma} \frac{\mc{V}\pi^{\hf-z}\rho_\cusp(\hf-z,-h) \ol{\langle V,E^*_\cusp(*,\thf+z) \rangle}  }{2(-1)^r r! (4\pi)^{-k}\Gamma(\hf+z)\Gamma(\hf-z)^2\zeta^*(1+2z)}\notag \\
&    \times  \lt[ \frac{h^{r-2z} \Gamma(\hf-z+r)\Gamma(2z-r)e^{(\hf+z-r-s)^2}}{\Gamma(k-\hf+z-r)(\hf+z-r-s)} + \frac{h^r\Gamma(\hf+z+r)\Gamma(-2z-r)e^{(\hf-z-r-s)^2}}{\Gamma(k-\hf-z-r)(\hf-z-r-s)} \rt] \ dz \notag\\
&+\Omega(s;h)-\sum_{\pm t_j, 0 \le r <  \hf \mp \Im(t_j) -c}\frac{R(\hf + it_j -r;h)e^{(\hf + it_j -r-s)^2}}{\hf + it_j -r-s} 
\end{align}
 Here in the integral with $\Re \, u = 2$, $D(u;h)$ is given by \eqref{Dsumdeltato0}, and in the integral with $\Re \, u = c$, $D(u;h)$ is given by \eqref{Dspec3}.
In this strip, $D(s;\sh,\delta)$ satisfies the same upper bound as given in \eqref{reallyfinalupper}.

The average upper bounds for the $c_{r,j}$ are obtained by an application of Stirling's formula, and Propositions~\ref{prop:Rezresult} and \ref{prop:innerup}.  It follows then that
\begin{align} \label{resisum}
\sum_{t_j, 0 \le r <  \hf -c}\frac{R(\hf + it_j -r;h)e^{(\hf + it_j -r-s)^2}}{\hf + it_j -r-s} \ll_c ( \log (\ell_1 \ell_2))^{\hf} (\ell_1 \ell_2)^{\hf - \kt} h^{\hf - \sigma +\theta}. 
\end{align}
By \eqref{reallyfinalupper}, the integral along
$\Re \, s = c$  is 
\begin{multline}\label{critbound}
 \ll (1 + |s|)^{A'} (\ell_1 \ell_2)^{\hf - \kt} h^{\hf-c}
\max\left(h^\epsilon, \frac{(h, \ell_1\ell_2)^{1/4}\sigma_0(\ell_1\ell_2)^\hf h^{1/4}\log(3h)}{(\ell_1\ell_2)^\hf}\right),
\end{multline} 
and by \eqref{1bound}, the integral along
$\Re \, s = 2$  is 
\begin{align*}
\ll \frac{\left( 1 + \frac{h}{\ell_2} \right)^{(k-1)/2}}{\ell_2 (\ell_1\ell_2)^{(k-1)/2}}(1+|s|)^{A'} & \ll \ell_2^{-1}h^{(k-1)/2} 
(\ell_1\ell_2)^{-(k-1)/2} (1+|s|)^{A'}.
\end{align*}
Noting that $\frac{1}{2}-c > \frac{k-1}{2}$,
 it follows from the previous three lines that for any $s$ with $c < \Re \, s < 2$,
$D(s;\sh)-\Omega(s;\sh)$ satisfies the upper bound
 \eqref{critbound}. 
 
 When $t_j=0$ for some $j$, the above argument is nearly identical except for additional, slightly different residual terms from \eqref{Cauchy1} arising from the double poles of $D(s;h,\delta)$, indicating that $D(s;h)$ has double poles at $s=\hf-r$.

 \end{proof}

\section{An application to single shifted sums }
\label{sect:app}
We begin by defining the smoothed shifted sum,
$$
S(x;h) := \sum_{\ell_1m_1 = \ell_2m_2+h}A(m_1)B(m_2)G(m_2/x).
$$
Here
\be\label{gGdef}
G(x) =  \frac{1}{2 \pi i}\int_{2 - i\infty}^{2 +i\infty} g(s)x^s ds
\ee
is smooth, with compact support in the interval $[1,2]$ and $g(s)$ is holomorphic, with arbitrary polynomial decay in vertical strips.
The objective of this section is to prove the following theorem. 
\begin{theorem}\label{theorem:singlesum}
Fix $N_0,\ell_1,\ell_2,h\in\N$,  with $N_0,\ell_1,\ell_2$ square free and $(N_0,\ell_1\ell_2)=1$. Suppose $h \ll \ell_2x$.   Then if $\theta>0$,
$$
S(x;h)\ll (\ell_2x)^\hf (\ell_2x/h)^\theta \max\left(1, \frac{(h, \ell_1\ell_2)^{1/4}\sigma_0(\ell_1\ell_2)^\hf h^{1/4}\log(3h)}{(\ell_1\ell_2)^\hf}\right).
$$
Also  if $(N_0\ell_1\ell_2,h)=1$, then
$$
S(x;h)\ll (\ell_2x)^\hf (\ell_2x/h)^\theta h^\theta =  (\ell_2x)^{\hf+\theta}.
$$
The above bounds still hold if $\theta=0$, unless $t_j=0$ for some $j$, in which case we replace $(\ell_2 x/h)^\theta$ with $\log(\ell_2 x /h)$.

Suppose $h \gg \ell_2x$,  then
\be\label{b1}
S(x;h)\ll h^\hf (h/\ell_2x)^\epsilon \max\left(1, \frac{(h, \ell_1\ell_2)^{1/4}\sigma_0(\ell_1\ell_2)^\hf h^{1/4}\log(3h)}{(\ell_1\ell_2)^\hf}\right).
\ee
Also,  if $(N_0\ell_1\ell_2,h)=1$ then
\be\label{b2}
S(x;h)\ll h^{\hf + \theta}(h/\ell_2x)^\epsilon.
\ee

\end{theorem}
The remainder of this section is devoted to the proof of this result.   To the best of the  knowledge of the authors, this theorem roughly equals the current state of the art in estimates for single shifted sums, although the proof is accomplished by new methods.  We include this mainly as a reference point from which we will build a new theory of double shifted sums.
\subsection{The expression of the shifted sum as an inverse Mellin transform}
The smoothed sum $S(x;h)$ can be expressed as an inverse Mellin transform of $D(s;h)$ as follows.
\begin{proposition}\label{prop:DtoS} For $0 < \gamma < (k-1)/2$,
$$
S(x;h)=\frac{(\ell_1\ell_2)^{(k-1)/2}}{(2 \pi i)^2}\int_{(\gamma)}\int_{(2)}
\frac{\G(u)\G((k-1)/2 -u)D(s;h)(\ell_2x)^{s+u}g(s+u)dsdu}{\G((k-1)/2)h^u}.
$$
\end{proposition}

\begin{proof} We observe that,
\begin{multline*}
S(x;h) =\sum_{\ell_1m_1 = \ell_2m_2+h} \frac{1}{2 \pi i}\int_{(2)} \frac{a(m_1)\bar b(m_2)g(s)x^sds}{m_1^{(k-1)/2}m_2^{(k-1)/2 +s}}\\
=\sum_{\ell_1m_1 = \ell_2m_2+h} \frac{(\ell_1\ell_2)^{(k-1)/2}}{2 \pi i}\int_{(2)} \frac{a(m_1)\bar b(m_2)g(s)}{(\ell_1m_1)^{(k-1)/2}(\ell_2m_2)^{(k-1)/2 }}\left(\frac{\ell_2x}{\ell_2m_2}\right)^sds\\
= \sum_{\ell_1m_1 = \ell_2m_2+h} \frac{(\ell_1\ell_2)^{(k-1)/2}}{2 \pi i}\int_{(2)} \frac{a(m_1)\bar b(m_2)(\ell_2 x)^sg(s)ds}{(\ell_2m_2 +h)^{(k-1)/2}(\ell_2m_2)^{(k-1)/2 +s}}\\
= \sum_{\ell_1m_1 = \ell_2m_2+h} \frac{(\ell_1\ell_2)^{(k-1)/2}}{2 \pi i}\int_{(2)} \frac{a(m_1)\bar b(m_2)}{(\ell_2m_2)^{k-1 +s}}
 \left(1 + \frac{h}{\ell_2m_2} \right)^{-(k-1)/2}(\ell_2 x)^sg(s)ds.
\end{multline*}
A form of the following identity can be found in \cite{GR}, 6.422(3).  For $\Re \, \beta > \gamma>0$  and $\text{arg}(t) < \pi$,
\be\label{GRid}
\frac{1}{2 \pi i} \int_{(\gamma)}\G(u)\G(\beta -u) t^{-u}du = \G(\beta)(1+t)^{-\beta}.
\ee
Move the line of integration to $\Re \, s = 2+\gamma$.  Then setting $\beta= (k-1)/2$ and $t = h/(\ell_2m_2)$  and substituting into the above, we have
\begin{align*}
S(x;h) &=
 \sum_{\ell_1m_1 = \ell_2m_2+h} \frac{(\ell_1\ell_2)^{(k-1)/2}}{(2 \pi i)^2}\int_{(2+\gamma)}  \int_{(\gamma)}
  \frac{a(m_1)\bar b(m_2)\G(u)\G((k-1)/2-u)}{(\ell_2m_2)^{k-1 +s-u}\G((k-1)/2)h^u}
(\ell_2 x)^sg(s)duds\\
&=\frac{(\ell_1\ell_2)^{(k-1)/2}}{(2 \pi i)^2}\int_{(2+\gamma)} \int_{(\gamma)}
  \frac{\G(u)\G((k-1)/2-u)D(s-u;h)(\ell_2 x)^sg(s)duds}{\G((k-1)/2)h^u}\\
&  =\frac{(\ell_1\ell_2)^{(k-1)/2}}{(2 \pi i)^2}\int_{(2)}  \int_{(\gamma)}
  \frac{\G(u)\G((k-1)/2-u)D(s;h)(\ell_2 x)^{s+u}g(s+u)duds}{\G((k-1)/2)h^u}\\
\end{align*}
The proposition follows after interchanging the order of the integrals.
\end{proof}
\subsection{Moving the line of integration} 
Taking the expression from Proposition \ref{prop:DtoS}, we move the $s$ line of integration to the contour, $C$, just to the left of the line $\Re s =\frac{1}{2}-\frac{k}{2}$ but still to the right of the zeros of $\zeta(2s+k)$. We bend $C$ so that it passes over every pole of the form $\hf+it_j-r$, including those potential exceptional poles to the left of $\Re s =\frac{1}{2}-\frac{k}{2}$ that would lie on the real line.

Thus by Proposition \ref{Dsumdeltato0}, if we assume that $t_j \neq 0$ for all $j$ then we obtain
\begin{align}\label{temp1}
&S(x;h)\\ 
&\notag=\frac{(\ell_1\ell_2)^{(k-1)/2}}{(2 \pi i)^2}\int_{(\gamma)}\int_{C}
\frac{\G(u)\G((k-1)/2 -u)(D(s;h)-\Omega(s;h))(\ell_2x)^{s+u}g(s+u)dsdu}{\G((k-1)/2)h^u} \\
&\notag + \sum_{j,0\le r <\hf-c}\frac{(\ell_1\ell_2)^{(k-1)/2}}{2 \pi i}\int_{(\gamma)}
\frac{\G(u)\G((k-1)/2 -u)c_{r,j} \overline{\rho_j(-h)}(\ell_2x)^{\hf}g(\hf+u-r+it_j)du}{\G((k-1)/2)(h/\ell_2 x)^{u-r+it_j}} \\ 
& \notag + \frac{(\ell_1\ell_2)^{(k-1)/2}}{(2 \pi i)^2}\int_{(\gamma)}\int_{(\hf)}
\frac{\G(u)\G((k-1)/2 -u)\Omega(s;h)(\ell_2x)^{s+u}g(s+u)dsdu}{\G((k-1)/2)h^u},
\end{align}
where the last integral above is more of a shorthand for the residual terms due to the polar contributors from the continuous part of $D(s;h)$, since $\Omega(s;h)$ is defined piecewise along vertical strips,.  Indeed, if we expand the last integral we get that
\begin{multline}
\frac{(\ell_1\ell_2)^{(k-1)/2}}{(2 \pi i)^2}\int_{(\gamma)}\int_{(\hf)}
\frac{\G(u)\G((k-1)/2 -u)\Omega(s;h)(\ell_2x)^{s+u}g(s+u)dsdu}{\G((k-1)/2)h^u}  \\
=\sum_{\mathfrak{a}} \sum_{r=0}^{\kt} \frac{(4\pi)^k(-1)^r(\ell_1\ell_2)^{(k-1)/2}}{2r!(2 \pi i)^2}\int_{(\gamma)}\int_{(\hf-r)} \frac{\G(u)\G((k-1)/2 -u)(\ell_2x)^{s+u}g(s+u)}{\G((k-1)/2)\Gamma(s+k-1)h^u} \\
\times \frac{\Gamma(1-s)\Gamma(2s+r-1)}{\Gamma(s+r)\Gamma(1-s-r)}
\lt[
		\frac{ \pi^{1-s-r} h^{1-2s-r} \vol \rho_\cusp \left(1-s-r, -h\right) } { \Gamma(1-s-r) \zeta^*(2s+2r) }   
	\overline{\langle V, E_\cusp^*\left(*, s+r \right)\rangle } 
	\rt.
	\\
	 \lt. +
	\frac{\pi^{s+r } h^{\ell} \vol \rho_\cusp \left(s+r , -h\right)}
	{\Gamma\left(s+r\right)\zeta^*(2-2s-2r )} 
	\ol{\langle V,E^*_\cusp(*,1-s-r ) \rangle} \rt] \ ds du .
	\end{multline}
	Changing $s \to s+\frac{1}{2}-r$ and changing the order of integration we get that this becomes
	\begin{multline}\label{omegapoles}
	 =	 \sum_{\substack{\mathfrak{a}, \\0\leq r\leq \kt }} \frac{(4\pi)^k(-1)^r(\ell_1\ell_2)^{(k-1)/2}}{2r!(2 \pi i)^2}\iint\limits_{(0) \ (\gamma)} \frac{\G(u)\G((k-1)/2 -u)(\ell_2x)^{\frac{1}{2}-r+s+u}g(\frac{1}{2}-r+s+u)}{\G((k-1)/2)\Gamma(s+k-\frac{1}{2}-r)h^{s+u-r}} \\
\times \frac{\Gamma(\frac{1}{2}-r-s)\Gamma(2s-r)}{\Gamma(s+\frac{1}{2})\Gamma(\frac{1}{2}-s)}
\lt[
		\frac{ \pi^{\frac{1}{2}-s} h^{-s} \vol \rho_\cusp \left(\frac{1}{2}-s, -h\right) } { \Gamma(\frac{1}{2}-s) \zeta^*(1+2s) }   
	\overline{\langle V, E_\cusp^*\left(*, \tfrac{1}{2}+s \right)\rangle } 
	\rt.
	\\
	 \lt. +
	\frac{\pi^{\frac{1}{2}+s } h^{s} \vol \rho_\cusp \left(\frac{1}{2}+s , -h\right)}
	{\Gamma\left(\frac{1}{2}+s\right)\zeta^*(1-2s)} 
	\ol{\langle V,E^*_\cusp(*,\tfrac{1}{2}-s ) \rangle} \rt] \ duds,
\end{multline}
and so we see this is actually completely analogous to the contribution from the poles at $s=\frac{1}{2}+it_j-r$ from the discrete part of the spectrum in \eqref{temp1}.

Since $D(s;h)-\Omega(s;h)$ satisfies the upper bound given in \eqref{reallyfinalupper} when $\sigma < \frac{1}{2}-\frac{k}{2}$, we see the size of the first double integral in \eqref{temp1} is controlled by the factor 
$$
h^{\hf}\left(\frac{\ell_2x}{h}\right)^{s+u} 
\max\left(1, \frac{(h, \ell_1\ell_2)^{1/4}\sigma_0(\ell_1\ell_2)^\hf h^{1/4}\log(3h)}{(\ell_1\ell_2)^\hf}\right)
$$ 
or furthermore, when $(h,N)=1$, by 
$$
h^{\hf+\theta}\left(\frac{\ell_2x}{h}\right)^{s+u}. 
$$ 
Similarly, we use Proposition~\ref{prop:innerup}, the bound in \eqref{conttriplebound}, and the expansion in \eqref{omegapoles} to show that the residual terms and the integral of $\Omega(h,s)$ are bounded by
\[
(\ell_2x)^\hf\left(\frac{\ell_2x}{h}\right)^{\Re(it_j+u)}\max \left(1, \sigma_0(\ell_1\ell_2)^\frac{1}{2}
\left(\frac{h (h,\ell_1\ell_2)}{(\ell_1\ell_2)^2}\right)^{1/4}\log(3h)\right) \]
or alternately, when $(h,N)=1$,
\[
(\ell_2x)^\hf\left(\frac{\ell_2x}{h}\right)^{\Re(it_j+u)}h^\theta.
\]
Thus in both cases, it is optimal to make the exponent as small as possible if $\ell_2x > h$, and as large as possible if 
$\ell_2x < h$.  If $\ell_2 x =h$ then it doesn't matter.

In the case $\ell_2x > h$, we take the integral component and we lower $u$ past the pole at 0, to $-1+\epsilon$.   Thus for $\theta>0$ the total contribution is on the order of 
$$
(\ell_2x)^\hf (\ell_2x/h)^\theta\max\left(1, \frac{(h, \ell_1\ell_2)^{1/4}\sigma_0(\ell_1\ell_2)^\hf h^{1/4}\log(3h)}{(\ell_1\ell_2)^\hf}\right).
$$
In particular, if $\ell_1,\ell_2$ are prime and relatively prime to $h$ then
\beq \label{thetazero}
S(x;h)\ll (\ell_2x)^\hf (\ell_2x/h)^\theta \max\left(1,\frac{h^{1/4}\log(3h)}{(\ell_1\ell_2)^\hf}\right),
\eeq
and alternatively, if $(N,h)=1$ then
\beq \notag
S(x;h)\ll (\ell_2x)^\hf (\ell_2x/h)^\theta h^\theta =  (\ell_2x)^{\hf+\theta}.
\eeq
The case where $t_j=0$ for some $j$ only alters the above argument when $\theta=0$, as it will contribute the main term. In particular, it adds terms to the residual spectrum of the form $\log(\ell_2x/h)$ rather than the $(\ell_2 x /h)^\theta$ that appears in \eqref{thetazero}.

In the case $\ell_2x < h$, we get the greatest contribution from the smallest exponent. Considering at the factor
\be\label{factor}
 h^\hf\left(\frac{\ell_2x}{h}\right)^{s+u},
\ee 
we increase $u$ past the pole at $(k-1)/2$, to $(k+1)/2-\epsilon$.   In this case the dominant contribution comes from the first double integral term in \eqref{temp1}.  Since $\sigma$ is always just slightly to the left of $\frac{1}{2}-\frac{k}{2}$ we obtain the bound
\be\label{b3}
S(x;h)\ll h^\hf (h/\ell_2x)^{\epsilon} \max\left(1,\frac{h^{1/4}}{(\ell_1\ell_2)^\hf}\right).
\ee
Alternatively, if $(N,h)=1$ then
\be\label{b4}
S(x;h)\ll h^{\hf + \theta}(h/\ell_2x)^{\epsilon}.
\ee

 This completes the proof of the Theorem \ref{theorem:singlesum}. \qed

\section{A double Dirichlet series}
\label{sect:double}
Recall that
\be\label{Dagain}
D(s;h) =  \sum_{m_1\ell_1= m_2\ell_2 + \sh} 
{a(m_1) \bar b(m_2)
\over (m_2\ell_2)^{ s+\wt-1}}
\ee
is convergent when $\Re \, s >1$.

Recall we have fixed $N_0,\ell_1,\ell_2 \in\N$,  with $N_0,\ell_1,\ell_2$ square free, $(N_0,\ell_1\ell_2)=1$ and have set $N=N_0\ell_1\ell_2/(\ell_1,\ell_2)$.   Fix  a positive integer $Q$, with $(Q,N)=1$.
For $s,w$ with real parts greater than one, the sum
\be
\label{Z0series}
Z_Q(s,w) = (\ell_1\ell_2)^{(k-1)/2} \sum_{h_0 \ge 1}\frac{ D(s;h_0Q)}{(h_0Q)^{w+(k-1)/2}}
\ee
is convergent, and can also be written 
 in the form
\be\label{Z0Def}
Z_Q(s,w) =  \sum_{{h_0 \ge 1}\atop{\ell_1m_1= \ell_2m_2 + h_0Q}}
\frac{ A(m_1)\bar B(m_2)\left( 1+ \frac{h_0Q}{\ell_2 m_2 }\right)^{(k-1)/2}  }{(\ell_2 m_2)^s(h_0Q)^{w+(k-1)/2}}.
\ee

The object of this section is to obtain the meromorphic continuation of 
$Z_Q(s,w)$ to  the part of $\C^2$ with $\Re \, w >1$.   The function $Z_Q(s,w)$ can be meromorphically continued to all of $\C^2$, but as this lengthens the argument, and we have no immediate application, we will instead refer to \cite{HLee}, where this second continuation is obtained in a different way. 

Before stating our first proposition we need to make several definitions.
For $\Re \, s' >1$, we let
\be
\label{LQDef}
L_Q(s',\overline{u_j}) = 
 \sum_{h_0 \ge 1}\frac{\overline{\rho_j(-h_0Q)}}{(h_0Q)^{s'}}.
\ee
This is, up to the presence of the factor $Q$, the familiar  $L$-series associated to $\overline{u_j}$, but with coefficients normalized so that the inner product of $u_j$ with itself equals one. 

We also define the function $\zeta_{\mathfrak{a},Q}(s',z)$, for each cusp $\cusp$, by a definition analogous to $L_Q$, with
\be
\label{zetaQDef2}\zeta_{\mathfrak{a},Q}(s',z)= 
 Q^{-{z}}\zeta(1-2z) \sum_{h\geq 1} \frac{\rho_\cusp\left(\frac{1}{2}-z, -hQ\right)}{h^{s'+z}} \nonumber.
\ee
Substituting \eqref{rhoexact} and summing, we obtain the following, as in \cite{HLee} line (2.34).
	By our assumptions,  $N$ is square free. For each cusp $\cusp=\frac{1}{w}$ for $w\mid N$, 
	we have,
	for $Q\geq 1$,
		\begin{multline}\label{zetaQexact}
	\zeta_{\cusp, Q}(s', z)
	=
	\zeta(s'+z) \zeta(s'-z)
	\left(\frac{1}{\frac{N}{w}}\right)^{\frac{1}{2}-z}
	Q^{-z}
	\prod_{p\mid N} \left(1-p^{-1+2z}\right)^{-1}
	\prod_{p\mid\frac{N}{w}} (1-p^{-(s'-z)})
	\\
	\times
	\prod_{p\mid w, p^\alpha\|Q, \atop \alpha\geq 0}
	\frac{p^{-1+2z}}{1-p^{2z}}
	\left((p-1)(1-p^{-(s'-z)})
	+ p^{\alpha 2z}(p^{2z}-p)(1-p^{-(s'+z)}) 
	\right)
	\\
	\times
	\prod_{q\mid Q, q\nmid N, \atop q^\alpha\|Q}
	\left(1-q^{2z}\right)^{-1}
	\left((1-q^{-(s'-z)}) - q^{(\alpha+1)2z}(1-q^{-(s'+z)}) \right)
	\,.
	\end{multline}

Finally, we define a ratio of Dirichlet polynomials $K^\pm_{\cusp, Q}(s')$ by setting
\be \label{bigKdef}
\Res_{z=\pm(1-s')} \zeta_{\cusp, Q}(s', z) = K^\pm_{\cusp, Q}(s') \zeta(-1+2s').
\ee

With this notation in place, we state the following
\begin{proposition}
\label{prop:ZQright}
Choose $c$, with $\hf -A < c < \hf -\kt$.  Let $\ell_1,\ell_2 \sim L$ as in Theorem \ref{thm:sum}.  The function $Z_Q(s,w)$ continues to the region $\Re \, s > c$, $\Re \, w >1$ and in this region, it is meromorphic. If $t_j\neq 0$ for all Maass forms at level $N$, then $Z_Q(s,w)$ has simple poles at the points $s = \hf-r \pm it_j$, for $0\le r <\hf-c\mp \Im\,(t_j)$. Furthermore it has poles due to $\Psi(s,w)$ as defined in \eqref{Psidef}, which are simple poles when $s =\hf -r $ for $r \geq 1$ and poles which are in the same place as the poles of $\zeta(2s+2r)^{-1}$ for $r\geq 0$. There are also simple polar lines at $w+2s+\kt= \frac{5}{2}-\ell$, for integral $\ell \ge 0$. If there is a Maass form at this level with $t_j=0$, then there are double poles at $s=\hf-r$ for $A>r \geq 0$.

When $s=\hf -r + it_j$, for $t_j \neq0$, we have
$$
\Res_{s = \hf -r + it_j}Z_Q(s,w)= (\ell_1\ell_2)^{(k-1)/2} c_{r,j} L_Q(w+(k-1)/2-r+it_j,\overline{u_j}),
$$
where the index $j$ differentiates between $+it_j$ and $-it_j$ and $c_{r,j}$ is given by \eqref{crjdef}.

When  $w+2s+\kt= \frac{5}{2}-\ell$
\begin{align}\label{sprime}
&\Res_{ w+2s+\kt= \frac{5}{2}-\ell} Z_Q(s,w) =\notag \\
& = \frac{(4\pi)^{k}\Gamma(1-s)}{2\Gamma(s+k-1)} 
	\sum_\cusp \lt[ \frac{\vol (-1)^\ell \Gamma(2s+\ell-1)Q^{-(\frac 32-s-\ell)}}
	{\ell! \pi^{s+\ell-1}
	\zeta^*(2s+2\ell)\Gamma(s+\ell)\Gamma(1-s-\ell)^2}\rt.  \\
&\times \lt. \lt(K_{\cusp, Q}^+(\tfrac 32 -s -\ell)\overline{\<V,E_\cusp^*(*, s+\ell)\>}  +K_{\cusp,  Q}^-(\tfrac32 - s -\ell)\overline{\<V,E_\cusp^*(*,1-s-\ell)\>}\rt) \rt]. \notag
\end{align}
	
Set $s' = s+w +\kt-1$ and fix $\epsilon >0$.   For $\hf-A <\Re \, s < \hf-\kt$ and $\Re \, w>1$, we have, away from the poles,  the convergent spectral expansion
$$
Z_Q(s,w) =  (\ell_1\ell_2)^{(k-1)/2}\left(S_{\text{\rm cusp}} (s,s')+ S_{\text{\emph{cts,int}}}(s,s')\right)+ \Psi(s,w),
$$

where 
\bea
\label{ZexpansionCusp}
 S_{\text{\rm cusp}}(s,s') := \frac{(4\pi)^k\G(1-s)}{2\G(s+k-1)}  \sum_j\frac{L_Q(s',\overline{u_j})\G(s - \frac12 -it_j)\G(s-\frac12 +it_j)\overline{\<V,u_j\>}}{\G(\frac12 -it_j)\G(\frac12 + it_j)},
 \eea
 $ S_{\text{\emph{cts,int}}}(s,s')$ is given by
\begin{multline}\label{contexp}
	S_{\emph{\text{cts},int}} =
	\frac{(4\pi)^{k}\Gamma(1-s)}{2\Gamma(s+k-1)} \\
    \times 
    \sum_{\mathfrak{a}}\frac{1}{2\pi i } \int_{-\ol{C_{\sigma,\sigma'}}} 
    \frac{\vol Q^{-s'} \zeta_{\mathfrak{a},Q}(s',z)\Gamma(s-\hf-z)\Gamma(s-\hf+z)}
    {\zeta^*(1+2z)\zeta^*(1-2z)\Gamma(\hf-z)\Gamma(\hf+z)}
    \ol{\langle V,E^*_\cusp(*,\thf+z) \rangle}  \ dz, 
   \end{multline}
and the curve of integration is described in \eqref{curvedef}.   We also have the following upper bounds:
For $\Re \, {s'} \ge\hf$, $s$ at least $\epsilon$ away from the poles, and $\Re \, w > 1$, there exists a constant $A''$, depending only on $k$ and $A$, such that
\be\label{Z2}
Z_Q(s,w)- \Psi(s,w) \ll  Q^{\theta - {\Re(s')+\epsilon}}(1+|s|)^{A''}(1 + |{s'}|)^{1+\epsilon}.
\ee 

For $\Re \, {s'} <\hf$ and  $\Re \, w > 1$,
\begin{align} & \label{Z3}
Z_Q(s,w) - \Psi(s,w) \\
& \ \ \ \ \notag \ll_{\ell_1,\ell_2}  Q^{-\Re(s')+\epsilon}\prod_{p^\gamma ||Q}
 \left|\overline{\lambda_j(p^\gamma)} - \overline{\lambda_j(p^{\gamma -1})}p^{-{s'}}\right| 
(1+|s|)^{A'}(1 + |{s'}|)^{1-2\Re({s'})}.
\end{align}
In this last case, the dependence on $\ell_1,\ell_2$ is suppressed, as our applications do not require a precise estimate in this region.

\end{proposition}

\begin{proof}
Applying Proposition~\ref{prop:D0}, in particular the upper bound \eqref{critbound} and the formulation for $\Omega(s;h)$ in \eqref{omega2}, we see that the summation \eqref{Z0series} is locally normally convergent, for $w$ with sufficiently large real part,  when $c< \Re \, s <2$.  For $s$ in the range $\Re \, s < \hf -\kt-\epsilon$,  each $D(s;h_0Q)$ can be represented by a convergent spectral sum, when $s$ is a distance of at least $\epsilon$ from any pole.    Applying Proposition~\ref{prop:D0}, when $s$ is in this range we can interchange this sum over $t_j$ with the sum over $h_0$ and obtain, for $ \Re \, w$ sufficiently large, a spectral expansion for $Z_Q(s,w)$.   Writing $s' = w + s + \kt -1$ to simplify this and further equations, we write:
\be
\label{Zinit}
Z_Q(s,w) = (\ell_1\ell_2)^{(k-1)/2}\left(S_{\text{cusp}} (s,s')+ S_{\text{cts}}(s,s')\right),
\ee
where $S_{\text{cusp}}(s,s')$  is given by \eqref{ZexpansionCusp},
 and 
\begin{multline}\label{ZexpansionCont}
  S_{\text{cts}}(s,s')
	:=  
	\frac{(4 \pi)^{k} \G(1-s)}{2\Gamma(s+k-1)} 
   \sum_\cusp  
   \lt[  \frac{1}{2\pi i} 
   \int_{-\ol{C_\sigma}} 
   \frac{\vol Q^{-s'} \zeta_{\cusp, Q}(s', z)}   
   {\zeta^*(1-2z) \zeta^*(1+2z) } 
   \rt.
   \\
   \times
   \frac{\G(s-\hf -z) \G(s-\hf + z) }
   {\G(\hf - z)\G(\hf + z)}
   \ol{\langle V,E^*_\cusp(*,\thf+z) \rangle}  \ dz 
	\\
	\lt. 
	+
	\sum_{\ell=0}^{\lfloor\hf-\sigma\rfloor} 
	\lt(
	\frac{\vol Q^{-s'}\zeta_{\cusp, Q} \left(s', \hf-s-\ell\right) } { \zeta^* (2-2s-2\ell) \zeta^*(2s+2\ell) }   
	\overline{\langle V, E_\cusp^*\left(*, s+\ell \right)\rangle } 
	\rt.
	\rt.
	\\
	+ 
	\lt. 
	(1-\pmb{\delta}_{\sigma, \ell} )
	\frac{\vol Q^{-s'} \zeta_{\cusp, Q} \left(s', -\hf+s+\ell\right)}
	{\zeta^*\left(2s+2\ell \right)\zeta^*(2-2s-2\ell )} 
	\ol{\langle V,E^*_\cusp(*,1-s-\ell ) \rangle}
	\rt)
	\\
	\lt. 
	\times
	\frac{(-1)^{\ell} \Gamma(2s+\ell-1)}
	{\ell! \Gamma(s+\ell) \Gamma(1-s-\ell)}
	\rt] .
	\end{multline}	
Here $C_\sigma$ and $\pmb{\delta}_{\sigma,\ell}$ are as described in \eqref{csigma} and \eqref{deltasigma}. 
This sum has been taken in the region where the zeta function converges,  where $\Re \, s' >1$.
We will now extend the region in which $Z_Q(s,w)$ is meromorphic to the region  $\Re \, w >1$, $\Re \, s < \hf -\kt$, with $\Re \, s' \le 1$.    We begin by determining the regions of absolute convergence of \eqref{ZexpansionCusp} and  \eqref{ZexpansionCont}.

Our first objective is to find an upper bound for
\begin{multline}\label{target}
\sum_{|t_j| \sim T}L_Q({s'},\overline{u_j})\overline{\<V,u_j\>} = \sum_{|t_j| \sim T}
 \sum_{h_0 \ge 1}\frac{\overline{\rho_j(-h_0Q)}\overline{\<V,u_j\>}}{(h_0Q)^{s'}}.
\\=
\sum_{|t_j| \sim T}L_1({s'},\overline{u_j})
\overline{\<V,u_j\>}Q^{-s'}\prod_{p^\gamma ||Q}
\left(\overline{\lambda_j(p^\gamma)} - \overline{\lambda_j(p^{\gamma -1})}p^{-{s'}}\right).
\end{multline}
Note that as $Q$ has been chosen to be relatively prime to $N$, we may assume that, for both oldforms and newforms, the $u_j$ are eigenfunctions of the Hecke operators at primes $p$ dividing 
$Q$, with eigenvalues $\lambda_j(p)$.   A consequence of this is that for $(n,Q)=1$, and $p$ dividing $Q$, 
$$
\rho_j(np^\gamma) = \rho_j(n)\lambda_j(p^\gamma).
$$
The right hand side of \eqref{target} then follows after a simple computation.

Suppose that $\Re \,{s'} <0$.  If $u_j$ is a newform then by the functional equation we have the bound for a single $L_1({s'},\overline{u_j})$, which we denote simply by $L({s'},\overline{u_j})$. 
$$
L({s'},\overline{u_j}) \ll (1 +|{s'}|+ |t_j|)^{1-2{\Re \, s'}}(\ell_1 \ell_2)^{\hf -{\Re \,s'}}|L({1-s'},\overline{u_j})|.
$$
If $u_j$ is an oldform then the dependence on $\ell_1 \ell_2$ is slightly more complicated.   As we do not need precise dependence on $\ell_1 \ell_2$ for our applications, we will suppress the dependence on $\ell_1 \ell_2$ and,
for $\Re \, {s'} <0$, will use the simpler bound,
\begin{multline}
\label{LQboundnu0}
L_Q({s'},\overline{u_j})\\ \ll_{\ell_1,\ell_2}\left| Q^{-\Re \,s'}\prod_{p^\gamma ||Q}
\left(\overline{\lambda_j(p^\gamma)} - \overline{\lambda_j(p^{\gamma -1})}p^{-{\Re \,s'}}\right)
(1 + |s'|+ |t_j|)^{1-2{\Re \,s'}}L({1-s'},\overline{u_j})\right|.
\end{multline}

For $\Re \, {s'} \ge \hf$, we apply Cauchy-Schwarz to the left hand side of \eqref{target}, obtaining
$$
\left(\sum_{|t_j| \sim T}L_Q({s'},\overline{u_j})\overline{\<V,u_j\>}\right)^2 \le 
\left(\sum_{|t_j| \sim T}|L_Q({s'},\overline{u_j})|^2e^{-\pi |t_j|}\right)\left(\sum_{|t_j| \sim T} |e^{\pi |t_j|} \overline{\<V,u_j\>}\   |^2\right)
$$
From Proposition~\ref{prop:Rezresult} we have
\be\label{partbound2}
\sum_{|t_j| \sim T}e^{\pi |t_j|}|\overline{\<V,u_j\>}|^2  
\ll  L^{-2k}T^{2k}\log T.
\ee

For $\Re \, s' \ge \hf$, the $L$-series average is addressed by the bound,
\be 
\label{LQbound12}
\sum_{|t_j|\sim T}e^{-\pi|t_j|} |L_Q(s', \overline{u_j})|^2 \ll Q^{2\theta-2\Re \, s' +\epsilon}L^{2+\epsilon}
(1+ |s'| +T)^{2 + \epsilon}.
\ee

\begin{remark}
By \eqref{target},
$$
L_Q({s'},\overline{u_j}) \approx \overline{\lambda_j(Q)}Q^{-s'}L({s'},\overline{u_j}).
$$
To obtain \eqref{LQbound12} we have simply estimated  $|\overline{\lambda_j(Q)}|^2$ from above, using the best progress toward the Ramanujan Conjecture, and applied a simple second moment estimate for the $L$-series part that is uniform in the level.  
\end{remark}

Combining \eqref{partbound2} with \eqref{LQbound12} gives us the desired result. For $\Re \, {s'} \ge \hf$,
\be\label{targetend}
\sum_{|t_j| \sim T}L_Q({s'},\overline{u_j})\overline{\<V,u_j\>} 
\ll 
Q^{\theta-s'+\epsilon}L^{1-k+\epsilon}
(1 + |\gamma'|+ |T|)^{1+k+\epsilon}.
\ee

To understand the region of absolute convergence
we now rewrite  \eqref{ZexpansionCusp} as 
$$
S_{\text{cusp}} = \lim_{T \rightarrow \infty} S_{\text{cusp}}(T),
$$
with
\bea\label{cuspZbound1}
 &&S_{\text{cusp}}(T) = \frac{2^{2k-1}\pi^k\G(1-s)}{\G(s+k-1)}\\
 &\times&  \sum_{|t_j| \ll T}\frac{L_Q({s'},\overline{u_j})\G(s - \frac12 -it_j)\G(s-\frac12 +it_j)\overline{\<V,u_j\>}}{\G(\frac12 -it_j)\G(\frac12 + it_j)} \nonumber
 \eea

The sum $S_{\text{cusp}}(T)$ can in turn be rewritten, after breaking up the interval $[0,T]$ dyadically, as
$$
S_{\text{cusp}}(T) \ll \sum_{i=0}^{\infty} S_{\text{cusp}}'(T/2^{i}) 
$$
where $S_{\text{cusp}}'(T)$ is the same as $S_{\text{cusp}}(T)$, but with the condition $|t_j| \ll T$  replaced by 
$|t_j|\sim T$.
By Stirling's formula we have
\be
\label{gammapart}
\frac{\G(1-s)\G(s - \frac12 -it_j)\G(s-\frac12 +it_j)}{\G(s+k-1)\G(\frac12 -it_j)\G(\frac12 + it_j)} \ll (1 + |t_j|)^{2\sigma -2}(1+|s|)^{3-k-3\sigma},
\ee
which is valid regardless of the relative sizes of $t_j$ and $\Im \, s$ for $s$ at least $\epsilon$ away from poles.

We now substitute the upper bounds of \eqref{LQboundnu0},  \eqref{targetend} and 
\eqref{gammapart}, obtaining,  
for $\Re \, {s'} <0$,
\bea\label{S0L}
&&S_{\text{cusp}}'(T) \\
&\ll_{ \ell_1,\ell_2}& Q^{\max(\theta - \Re \,s', -2\Re \,s')+\epsilon}(1 + |{s'}| +T)^{2-2\Re \,s'+ \epsilon }T^{2\sigma+ k-2 }(1+|s|)^{3-k-3\sigma}. 
\nonumber
\eea
In the case $\Re \, {s'} \ge\hf$, we similarly get the more precise bound
\bea\label{S12L}
&&(\ell_1\ell_2)^{(k-1)/2}S_{\text{cusp}}'(T) \\
&\ll& Q^{\theta - \Re \,s'+\epsilon}(1 + |{s'}| +T)^{1+\epsilon}
T^{2\sigma + k-2} (1+|s|)^{3-k-3\sigma}. 
\nonumber
\eea
Note that there may be exceptional poles on the real line, close to $s=1/2-k/2$.   As long as $s$ is at least $\epsilon$ away from any of these, the contribution from this piece of the sum has already been accounted for in the estimate.
The interval $0 \le \Re \, {s'} \le\hf$ is filled in by convexity.
Convergence of the expansion given in  \eqref{ZexpansionCusp} as $T\rightarrow \infty$,  away from the poles, will  occur  when the exponent of $T$ in \eqref{S0L} and  \eqref{S12L} is negative.   This occurs for 
 \be\label{cuspexps}
\Re \, s<  \begin{cases} \Re \, {s'}  - \tkt & \text{ if $\Re \,{s'} < \hf$,}\\
\hf - \tkt & \text{ if $\Re\, {s'} \ge \hf$. }
\end{cases}
\ee
As $s' = s+w +\kt-1$, this implies that
$S_{\text{cusp}}(s,s')$ is convergent and analytic in $s$ and $s'$ away from poles when $\Re \, s < \hf -\tkt$ and $\Re \, w >1$, 
and satisfies
\be\label{Scuspbound}
S_{\text{cusp}}(s,s') 
\ll_{\ell_1,\ell_2} Q^{\theta -\Re \, s'}(1 + |{s'}| )^{1+\epsilon}
(1+|s|)^{3-k-3\sigma}. 
\ee

\subsection{The continuous piece}
We now turn to $S_{\text{cts}}(s,s')$.   This piece converges in the same region as $S_{\text{cusp}}(s,s')$ and satisfies a stronger upper bound.    To demonstrate this we need to estimate
$$
   \sum_\cusp  
   \frac{1}{2\pi i} 
   \int_{(0)} 
   \frac{\vol Q^{-s'} \zeta_{\cusp, Q}(s', z)}   
   {\zeta^*(1-2z) \zeta^*(1+2z) } 
    \frac{\G(s-\hf -z) \G(s-\hf + z) }
   {\G(\hf - z)\G(\hf + z)}
   \ol{\langle V,E^*_\cusp(*,\thf+z) \rangle}  \ dz.
$$

Recall again that $N$ is square-free.  
As $f$ and $g$ are new forms, 
by \cite{Asa76}, applying Lemma 5.7 in \cite{HLee}, 
and summing over the cusps, we see that 
\begin{multline}\label{aa}
 \sqrt \vol \sum_\cusp	\left<V, E^*_\cusp\left(*, \tfrac{1}{2}+z\right)\right> \\ \ll \zeta^*(1+2z)  L^{-k-1}(1 + |t|)^{k-1}e^{-\pi |t|/2}L(\thf+ it, f \otimes \overline{g})\log (1+|t|)N^\epsilon.
\end{multline}

Applying the convexity bound to $L(\thf+ it, f \otimes \overline{g})$, Stirling's formula, \eqref{zetaQexact}, and the fact that
\be\label{bb}
\int_0^T \zeta(s'+it)\zeta(s'-it) dt \ll_{s'} T^{1+\epsilon}
\ee
when $\Re \, s' \geq \frac{1}{2}$
with at most polynomial dependence on $s'$,
we obtain convergence in the desired region.

In this region we observe that $S_{\text{cts}}(s,s')$ has poles due the continuous part of $D(s;h)$ at $s = \hf -r$ for $r \geq 1$, and when $s = \frac \vho 2 -r$ with $r \geq 0$ and  $\vho$ a nontrivial zero of $\zeta$.  We also observe that there are new potential poles in this region due to $\zeta(s'-\hf+s+\ell)$, where $\ell \leq\lfloor\hf -\sigma \rfloor$. There are no poles due to $\zeta(s'+\hf-s-\ell)$ in this region since $-\Re \, s > \ell -\hf$ so $\Re \, s' +\hf -\Re \, s -\ell >1$. 

If we fix $\Re \, s$ at least $\epsilon>0$ away from the points $\frac{1}{2}-r$ for $r \geq 0$, then for $\Re \, s' >1$, the integral component of the continuous contribution in \eqref{ZexpansionCont} is, for $\Re \, w $ sufficiently large and $\Re \, s < \hf -\kt$, 
\begin{align}
 \label{ZexpansionCont2}
 &S_{\text{int}}(s,s') :=
\frac{(4\pi)^{k}\Gamma(1-s)}{2\Gamma(s+k-1)} \\
& \   \times   \sum_{\mathfrak{a}} \lt(  \frac{1}{2\pi i } \int_{(0)} \frac{\vol Q^{-s'}\zeta_{\mathfrak{a}, Q}(s', z)\Gamma(s-\hf-z)\Gamma(s-\hf+z)}{\zeta^*(1+2z)\zeta^*(1-2z)\Gamma(\hf-z)\Gamma(\hf+z)}\frac{}{}\ol{\langle V,E^*_\cusp(*,\thf+z) \rangle} \ dz \rt). \notag 
\end{align} 
This converges when $\Re \, s' >1$, as the zeta factors are $\ll 1$ and the gamma factors and inner products satisfy the same upper bounds as in the cuspidal contribution.
To meromorphically continue this integral to the region $\Re \, s' \le 1 $, we take $s'$ to be within the region to the right of $\Re \, s' =1$ and to the left of $1+C$, where $C(x)$ is as in  description of \eqref{shift1}. Without loss of generality, we can also choose $\epsilon>0$ such that $\epsilon >|\Re \, C(x)|>0$ for all $x$.   Shifting the line of integration to $-C$, a pole of $\zeta_{\mathfrak{a}, Q}(s', z)$ is passed over at $z = 1-s'$, giving
\begin{align}
&S_{\text{int}}(s,s') = \label{ZexpansionContRes}
\frac{(4\pi)^{k}\Gamma(1-s)}{2\Gamma(s+k-1)} \\
&    \times   \lt[ \sum_{\mathfrak{a}}   \frac{1}{2\pi i } \int_{-C} \frac{\vol Q^{-s'}\zeta_{\mathfrak{a}, Q}(s', z)\Gamma(s-\hf-z)\Gamma(s-\hf+z)}{\zeta^*(1+2z)\zeta^*(1-2z)\Gamma(\hf-z)\Gamma(\hf+z)}\frac{}{}\ol{\langle V,E^*_\cusp(*,\thf+z) \rangle} \ dz \rt. \notag \\
\notag
 &\ \ \ \ \ +\lt.  \vol Q^{-{s'}} 
K^+_{\cusp, Q}(s') \frac{\G(s + \frac12 -s')\G(s- \frac32 +s')\overline{\<V,E_\cusp^*(*,\frac 32 -s')\>}}
{\pi^{\hf -s'}\zeta^*(3-2s')\G( s'-\frac12)^2\G( \frac32  -s')} \rt]\,.
\end{align}
Note the cancellation of $\zeta^*(1-2z)$ in the denominator.  We see that the above expression is analytic for $s'$ to the right of $1-C$ and to the left of $1+C$. If we then take $s'$ to be to the left of $\Re s' =1$, we can shift the line of integration back to $\Re \, z=0$, passing over the pole at $z=s'-1$ and so we get
\begin{align}
&S_{\text{int}}(s,s') = \label{ZexpansionCont3}
\frac{(4\pi)^{k}\Gamma(1-s)}{2\Gamma(s+k-1)} \\
&    \times   \lt[ \sum_{\mathfrak{a}}   \frac{1}{2\pi i } \int_{-\overline{C_\sigma}} \frac{\vol Q^{-s'}\zeta_{\mathfrak{a}, Q}(s', z)\Gamma(s-\hf-z)\Gamma(s-\hf+z)}{\zeta^*(1+2z)\zeta^*(1-2z)\Gamma(\hf-z)\Gamma(\hf+z)}\frac{}{}\ol{\langle V,E^*_\cusp(*,\thf+z) \rangle} \ dz \rt. \notag \\
\notag
 & \ \ \ \ \ +\lt.  
 \frac{\G(s + \frac12 -s')\G(s- \frac32 +s')}
{Q^{s'}\pi^{\hf -s'}\zeta^*(3-2s')\G( s'-\frac12)^2\G( \frac32  -s')} \rt. \notag \\
& \ \ \ \ \ \  \times \lt. \lt(K_{\cusp, Q}^+(s')\overline{\<V,E_\cusp^*(*,\tfrac 32 -s')\>}+K_{\cusp, Q}^-(s')\overline{\<V,E_\cusp^*(*,s' -\tfrac 12)\>}\rt)  \rt] \notag
\end{align}
which, still for fixed $s$ at least $\epsilon$ away from $\hf -r$ for $r\geq0$, is a meromorphic function in $s'$  defined for $\Re s' < 1$. We observe that $S_{\text{int}}$ only has poles in $s'$ due to $\Gamma(s+\hf-s')$ and $\Gamma(s-\frac 32+s')$, and we will show next that only the latter function contributes poles in $S_{\text{cts}}$. Indeed, after following nearly identical reasoning as above in resolving the case where $s=\hf-r$ for $r \geq 0$, we derive the following meromorphic expansion for $S_{\text{cts}}$, for all  $\hf-\tkt>\Re \, s =\sigma >\hf-A$ and all $s'$ with $\Re \, s' =\sigma'$ such that the integral component is convergent,
\begin{multline}\label{contexp}
	S_{\text{cts}} =
	\frac{(4\pi)^{k}\Gamma(1-s)}{2\Gamma(s+k-1)} \\
    \times 
    \sum_{\mathfrak{a}}\frac{1}{2\pi i } \int_{-\ol{C_{\sigma,\sigma'}}} 
    \frac{\vol Q^{-s'} \zeta_{\mathfrak{a},Q}(s',z)\Gamma(s-\hf-z)\Gamma(s-\hf+z)}
    {\zeta^*(1+2z)\zeta^*(1-2z)\Gamma(\hf-z)\Gamma(\hf+z)}
    \ol{\langle V,E^*_\cusp(*,\thf+z) \rangle}  \ dz 
    \\ +\Psi(s,w),
\end{multline}
where

\begin{multline}\label{Psidef}
	\Psi(s,w)
	\\
	:=	
	\frac{(4\pi)^{k}\Gamma(1-s)}{2\Gamma(s+k-1)}
	\sum_{\mathfrak{a}} \bigg[ 
	\sum_{\ell=0}^{\lfloor\hf-\sigma\rfloor} 
	\lt(
	\frac{\vol  
	Q^{-s'} \zeta_{\cusp, Q} \left(s', -\frac{1}{2}+s+\ell \right) } 
	{ \zeta^* (2-2s-2\ell) \zeta^*(2s+2\ell) }   
	\overline{\langle V, E_\cusp^*\left(*, s+\ell \right)\rangle } 
	\rt.
	\\
	+ 
	\lt. 
	(1-\pmb{\delta}_{\sigma, \ell} )
	\frac{\vol Q^{-s'} \zeta_{\cusp, Q} \left(s', \frac{1}{2} -s-\ell\right)}
	{\zeta^*\left(2s+2\ell \right)\zeta^*(2-2s-2\ell )} 
	\ol{\langle V,E^*_\cusp(*,1-s-\ell ) \rangle}
	\rt)
	\frac{(-1)^{\ell} \Gamma(2s+\ell-1)}
	{\ell! \Gamma(s+\ell) \Gamma(1-s-\ell)}
	\\
	+
	\pmb{\delta}'_{\sigma'}
 	\frac{\G(s + \frac12 -s')\G(s- \frac32 +s')}
	{Q^{s'}\pi^{\hf -s'}\zeta^*(3-2s')\G( s'-\frac12)^2\G( \frac32  -s')}  
	\\
	\times 
	\lt(K_{\cusp, Q}^+(s')
	\overline{\<V,E_\cusp^*(*,\tfrac 32 -s')\>}
	+(1-\pmb{\delta}_{\sigma',-\hf})K_{\cusp, Q}^-(s')
	\overline{\<V,E_\cusp^*(*,s' -\tfrac 12)\>}\rt) \bigg]
	\,.
	\end{multline}
	Here $K_{\cusp, Q}^\pm(s')$ is the ratio of Dirichlet polynomials defined by \eqref{bigKdef},
	$\pmb{\delta}_{\sigma,\ell}$ is as in \eqref{deltasigma} and similarly
 \beq \label{deltaprime}
 \pmb{\delta}'_{x} = \lt\{
 \begin{array}{ll}
 0 & \mbox{ when } x >1 \\
 1 & \mbox{ when } x \leq 1.
 \end{array}
\rt.
 \eeq
 Morally, we define $C_{\sigma,\sigma'}$ to be $(0)$ or a curve slightly to the right of $(0)$ such that the integrand of the integral portion of $S_{\text{cts}}$ doesn't run afoul of any poles. More specifically we define it to be
 \beq\label{curvedef}
 C_{\sigma,\sigma'} =\lt\{ 
 \begin{array}{ll}
 (0) & \mbox{ when  } \sigma'\neq 1 \mbox{ and }\sigma \neq \hf -r \mbox{ for } r \in \Z_{\geq 0} \\
C & \mbox{ when  } \sigma'=1 \mbox{ and }\sigma = \hf -r \mbox{ for } r \in \Z_{\geq 0} \\
\epsilon_\sigma C &  \mbox{ when  } \sigma = 1 \mbox{ and }\sigma \neq \hf -r \mbox{ for } r \in \Z_{\geq 0} \\
\epsilon'_{\sigma'} C & \mbox{ when  } \sigma' \neq 1 \mbox{ and }\sigma = \hf -r \mbox{ for } r \in \Z_{\geq 0}
 \end{array} 
 \rt. 
 \eeq
 where $C$ is as in the description accompanying \eqref{shift1}. We choose $1>\epsilon_d,\epsilon'_{d'}>0$ to be appropriate scaling factors such that $|\sigma-\hf+r|>|\Re \, \epsilon_\sigma C(x)|$ for all $r \in \Z_{\geq 0}$ and $|\sigma'-1|> |\Re \, \epsilon_{\sigma'} C(x)|$ in the respective cases. 

As remarked before, the right hand side of \eqref{contexp} converges in a superset of the region where $S_{\text{cusp}}(s,s')$ converges.  We also note that $S_{\text{cts}}(s,s')$ indeed has simple poles at $s =\hf-r$ for $r \in \Z_{\geq 1}$ and  poles at $\frac \vho 2 -r$ for $r \in \Z_{\geq 0}$  where $\vho$ is a non-trivial zero of $\zeta$. The only other poles to consider are the potential simple poles at $s' =\frac 3 2 -s-\ell$ and $s' = \hf+s+\ell$ for $\ell \geq 0$. The latter poles should not exist, as we already noted above that they cannot appear when $\Re \, s' >1$. Indeed, suppose that $\Re \, s' \leq 1$ along the line $s'=\hf+s+\ell$, then $\Re \, s \leq  \hf -\ell$, and a simple omitted computation shows that these poles do indeed cancel. We similarly compute the residues when $s'=\frac{3}{2}-s-\ell$, although we no longer necessarily assume that $\Re \, s' \leq 1$,
\begin{align}
&\Res_{s'=\frac 32-s-\ell} S_{\text{cts}}(s,s')
	= \frac{(4\pi)^{k}\Gamma(1-s)}{2\Gamma(s+k-1)}  \\
& \  \lt. \times \sum_\cusp \lt[\pmb{\delta}'_{\hf+\sigma+\ell} 
	\frac{\vol (-1)^\ell \Gamma(2s+\ell-1)Q^{-(\frac 32-s-\ell)}}
	{\ell! \pi^{s+\ell-1}\zeta^*(2s+2\ell)\Gamma(s+\ell)\Gamma(1-s-\ell)^2}  \rt. \rt. \notag \\
& \ \ \ \ \  \times \lt.\lt(K_{\cusp, Q}^-(\tfrac 32 -s - \ell)\ol{\<V,E^*_\cusp(*,1-s-\ell)\>}\rt. \rt. \notag \lt.  +(1-\pmb{\delta}_{\sigma,\ell})K_{\cusp, Q}^+(\tfrac 32 -s -\ell)
	\ol{\< V,E^*_\cusp(*,s+\ell)\>} \rt) \ \notag \\
& \  +\lt. \pmb{\delta}'_{\frac 32 -\sigma-\ell} 
 \frac{\vol (-1)^\ell \G(2s+\ell-1)}
{\ell! Q^{\frac{3}{2}-s-\ell}\pi^{s+\ell-1}\zeta^*(2s+2\ell)\G( 1-s-\ell)^2\G( s+\ell)} \notag \rt. \\
&  \ \ \ \ \  \times \lt. \lt(K_{\cusp, Q}^+(\tfrac 32 -s -\ell)
	\overline{\<V,E_\cusp^*(*, s+\ell)\>} \rt. \rt. \notag \\
&  \ \ \ \ \  \ \ \ \ \ \lt. \lt. +(1-\pmb{\delta}'_{\hf+\sigma+\ell})K_{\cusp, Q}^-(\tfrac32 - s -\ell)\overline{\<V,E_\cusp^*(*,1-s-\ell)\>}\rt) \rt] \notag \\
& = \frac{(4\pi)^{k}\Gamma(1-s)}{2\Gamma(s+k-1)} 
	\sum_\cusp \lt[ \frac{\vol (-1)^\ell \Gamma(2s+\ell-1)Q^{-(\frac 32-s-\ell)}}
	{\ell! \pi^{s+\ell-1}
	\zeta^*(2s+2\ell)\Gamma(s+\ell)\Gamma(1-s-\ell)^2}\rt. \notag  \\
&\times \lt. \lt(K_{\cusp, Q}^+(\tfrac 32 -s -\ell)\overline{\<V,E_\cusp^*(*, s+\ell)\>}  +K_{\cusp Q}^-(\tfrac32 - s -\ell)\overline{\<V,E_\cusp^*(*,1-s-\ell)\>}\rt) \rt]. \notag
\end{align}
The above computation gives us \eqref{sprime} of the proposition.   

\subsection{Removing the poles of $Z_Q(s,w)$}
We have now determined the location of the poles of $S_{\text{cts}}(s,s')$, and these poles are cancelled by
$$\frac{B_A(s)}{\Gamma(s'-\frac32 +s)\Gamma(2s-1)},$$
where
$$
B_A(s):= \prod_{\ell=0}^{[A \rfloor} \zeta(2s+2\ell)(s+\ell-\thf)^2.
$$
It follows that $$\frac{B_A(s)}{\Gamma(s'-\frac32 +s)\Gamma(2s-1)}S_{\text{cts}}(s,s')$$ is analytic in its region of convergence.
This region includes all $(s,w)$ with $c< \Re \, s <2$ and $\Re \, w$ sufficiently large, and all $(s,w)$ with $c< \Re \, s <\hf -\kt$ and $\Re \, w >1$, and all  $(s,w)$ with $\Re \, s >1 $ and $\Re \, w >1 $.  

Returning to the cuspidal piece,
we recall the analytic continuation of $D(s;h)$ given in Proposition~\ref{prop:D0}
and define $\tilde D(s;h)$ by
\be
\label{DtildeDef}
\tilde D(s;h) = D(s;h)e^{s^2}  -  \sum_{t_j\neq0,0\le r <\hf -c -\Im(t_j)}\frac{ c_{r,j} h^{r-it_j}\overline{\rho_j(h)}
e^{(\hf -r+ it_j)^2}}{s-(\hf -r +it_j)}.
\ee
The function $e^{s^2}$ has been included to insure that the sum over $t_j$ above converges.   A polynomially decaying function in the imaginary part of $s$ would suffice, but $e^{s^2}$ is convenient as it introduces no new poles. 

By construction, $\tilde D(s;h)$ is analytic for all $s$ with $\Re \, s >c$.
Now define
\be
\label{ZtildeDef}
\tilde Z_Q(s,w) =\left(\ell_1\ell_2\right)^{(k-1)/2} \sum_{h_0 \ge 1}\frac{\tilde D(s;h_0Q)}{(h_0Q)^{w+(k-1)/2}}.
\ee
By Proposition~\ref{prop:D0}, for   $c< \Re \, s <2$, the sum over $h_0$ in  $Z_Q(s,w)$ converges when $s$ at least $\epsilon$ away from poles and
when $\Re \, w +(k-1)/2 >2-2c$.  The sum 
$$
\sum_{h_0 \ge 1} \sum_{t_j,0\le r <\hf -c - \Im(t_j) }\frac{ c_{r,j}( h_0Q)^{r-it_j}\overline{\rho_j( h_0Q))}
e^{(\hf-r+it_j)^2}}{( h_0Q)^{w + (k-1)/2}(s-\hf +r -it_j)}
$$
also converges in the larger region where  $\Re \, w +(k-1)/2 > 3/2+ \theta  -c$.  Thus $\tilde Z_Q(s,w)$ converges and $\tilde Z_Q(s,w)$ is analytic along the lines  $(\hf-r + it_j, w)$ in the region $c< \Re \, s <2 $, $\Re \, w +(k-1)/2 > 2-2c$.    

By  \eqref{DtildeDef}
\be
\label{ZPDef}
\tilde Z_Q(s,w) = Z_Q(s,w)e^{s^2}- P_Q(s,w),
\ee
where
\begin{align}
\label{PDef}
P_Q(s,w) &=   (\ell_1\ell_2)^{(k-1)/2}\sum_{t_j \neq 0 ,0\le r <\hf -c-\Im(t_j),h_0 \ge 1}\frac{ c_{r,j} (h_0Q)^{r-it_j}\overline{\rho_j(h_0Q)}e^{(\hf -r+ it_j)^2}}
{(h_0Q)^{w+(k-1)/2}(s-(\hf -r +it_j))}
\\ & = (\ell_1\ell_2)^{(k-1)/2} \sum_{t_j\neq 0,0\le r <\hf-c-\Im(t_j)}\frac{ c_{r,j} L_Q(w+(k-1)/2-r+it_j,\overline{u_j})e^{(\hf -r+ it_j)^2}}{s-(\hf -r +it_j)}. \notag
\end{align}

By construction, when $t_j \neq 0$, 
\bea
\label{Zresform}
&&\Res_{s = \hf -r +it_j} Z_Q(s,w)e^{s^2} = \Res_{s = \hf -r +it_j}P_Q(s,w)\\
&=& (\ell_1\ell_2)^{(k-1)/2} c_{r,j} L_Q(w+(k-1)/2-r+it_j,\overline{u_j})e^{(\hf -r+ it_j)^2}.\nonumber
\eea

The interchange in the order of summation above is allowed if the real part of $w$ is sufficiently large.   Once interchanged, it is clear that, given the analytic continuation of 
$L_Q(w+(k-1)/2-r+it_j,\overline{u_j})$, the summation converges for any fixed value of $w$.
Thus the sum defining $P_Q(s,w)$ converges and is, again by construction,  analytic for all points  $s,w \in \C$ with $\Re\, s > c$, except for polar lines at  $(s,w) =(\hf-r + it_j,w )$. If $t_j=0$ for some Maass form of this level then we have double poles at $s=\hf-r$ for $r \geq 0$ and the computation is slightly different. 

The function $\tilde Z_Q(s,w)$ has had its poles of the form $(\hf-r + it_j,w )$ removed, and multiplying by 
$B_A(s)/(\Gamma(s'+s-\frac32)\Gamma(2s-1))$ removes the remaining poles coming from the continuous spectrum.  By the previous analysis, it is also convergent in a region including all $(s,w)$ with $c< \Re \, s <2$ and $\Re \, w$ sufficiently large, and all $(s,w)$ with $c< \Re \, s <\hf -\kt$ and $\Re \, w >1$, and all  $(s,w)$ with $\Re \, s >1 $ and $\Re \, w >1 $.  
The convex  hull of this tube domain is all   $(s,w)$ with 
$\Re \, w >1$ and $\Re \, s >c$.   Consequently, by B\^ochner's theorem, $\tilde Z_Q(s,w)$ is analytic in this convex hull. 
This completes the proof of the proposition.
\end{proof}

\section{Double shifted sums}\label{sect:Zsums}

\subsection{Removing the weighting factor}
The unweighted but smoothed sum we are interested in estimating is
$$
S_Q(x,y) = \sum_{\ell_1m_1 = \ell_2m_2 + hQ \atop h, m_2 \ge 1}
A(m_1)\bar B(m_2) G_1(m_1/y)G_2(m_2/x).
$$
The above sum is related to $Z_Q(s,w)$ via a triple inverse Mellin transform, as set forth in the following proposition.
\begin{proposition}\label{prop:notprime}
Suppose $G_1(x), G_2(x)$  are smooth with compact support in the interval $[1,2]$, with 
$g_1(s),g_2(s)$ denoting  the corresponding Mellin transforms as before.
For $x \gg 1$ and $y\gg1$,
\begin{multline}\label{Sxydef}
S_Q(x,y) = \left(\frac{1}{2 \pi i}\right)^3\int_{(\gamma_1)}\int_{(\gamma_2)}\int_{(\gamma_3)}
\frac{Z_Q(s+w-u, u+(1-k)/2) }{\Gamma((k-1)/2 +w)}\\ \times \Gamma((k-1)/2 +w -u) \Gamma(u)\ell_1^w\ell_2^sx^s y^w g_1(w)g_2(s)du dw ds.
\end{multline}
Here $\gamma_1,\gamma_2,\gamma_3$ are chosen so that $\Re \, s+w-u >1$,
 $\Re \,  u+(1-k)/2 >1$, and $\Re \, (k-1)/2 +w -u >0$. In particular we choose
$ \gamma_1 = (k+1)/2$, $ \gamma_2 = 1+ 2 \epsilon$, 
$ \gamma_3 = (k+1)/2 + \epsilon$, for any $\epsilon >0$.
\end{proposition}

\begin{proof}
Recall the identity \eqref{GRid}.  Changing the variables in \eqref{Z0Def}, and setting $\beta = (k-1)/2 +w$, we have, for any $\gamma$ satisfying $0 < \gamma <\Re \, \beta$,
\begin{align}
& \frac{1}{2 \pi i} \int_{(\gamma)}Z_Q(s+w-u,u +(1-k)/2)\G(u)\G( (k-1)/2 +w -u)
\ell_1^w\ell_2^sdu \\
&\notag= \sum_{{m_2,h_0 \ge 1}\atop{\ell_1m_1= \ell_2m_2 + h_0Q}} A(m_1)\bar B(m_2)
 \frac{1}{2 \pi i} \int_{(\gamma)}\frac{\G(u)\G(\frac{k-1}{2} +w -u)\ell_1^w\ell_2^s \left( 1+ \frac{h_0Q}{\ell_2 m_2 }\right)^{(k-1)/2}du }{(\ell_2 m_2)^{s+w-u}(h_0Q)^{u}}\\
 &\notag  =\G(\tfrac{k-1}{2} +w)  \mkern -20 mu\sum_{{m_2,h_0 \ge 1}\atop{\ell_1m_1= \ell_2m_2 + h_0Q}}
\frac{A(m_1)\bar B(m_2)\ell_1^w\ell_2^s}{(\ell_2 m_2)^{s+w}} \left( 1+ \frac{h_0Q}{\ell_2 m_2 }\right)^{-w-(k-1)/2} \left( 1+ \frac{h_0Q}{\ell_2 m_2 }\right)^{(k-1)/2} \\
& \notag =
\G(\tfrac{k-1}{2} +w) \mkern -20mu \sum_{{m_2,h_0 \ge 1}\atop{\ell_1m_1= \ell_2m_2 + h_0Q}}
\frac{A(m_1)\bar B(m_2)\ell_1^w\ell_2^s}{(\ell_2 m_2)^{s+w} \left( 1+ \frac{h_0Q}{\ell_2 m_2 }\right)^w}=\G(\tfrac{k-1}{2} +w) \mkern -20mu \sum_{{m_2,h_0 \ge 1}\atop{\ell_1m_1= \ell_2m_2 + h_0Q}}
\frac{A(m_1)\bar B(m_2)}{m_2^sm_1^w}. \notag
\end{align}
The proposition then follows, after setting $\gamma = \gamma_3$ and applying the inverse Mellin transform
$$
\left(\frac{1}{2 \pi i}\right)^2 \iint\limits_{(\gamma_1)(\gamma_2)}g_1(w)g_2(s)x^sy^wdw ds
$$
to both sides.
\end{proof}

Having related $S_Q(x,y)$ to $Z_Q(s+w -u, u+(1-k)/2)$ we will now derive an estimate for $S_Q(x,y)$ from the analytic properties of $Z_Q(s+w -u, u+(1-k)/2)$
and prove Theorem~\ref{thm:sum} from the introduction.

\begin{proof}[Proof of Theorem \ref{thm:sum}]
Changing variables in the preceding Proposition,
\begin{multline}\label{Sxydef2}
S_Q(x,y) = \left(\frac{1}{2 \pi i}\right)^3\iiint\limits_{(1+\epsilon)(1+2\epsilon)(1+\epsilon)}
\frac{Z_Q(s, u) \G(w-u)\Gamma(u+(k-1)/2 )}{\Gamma(w+(k-1)/2)}\\ \times (\ell_2 x)^{s+u+(k-1)/2-w}(\ell_1y)^wg_1(w)g_2(s+u+(k-1)/2-w)du dw ds.
\end{multline}
Referring to \eqref{ZexpansionCusp}, rewrite $S_{\text{\rm cusp}}(s,s')$ as
$$
S_{\text{\rm cusp}}(s,s') = S_{\rm cusp,gen}(s,s') + S_{\text{ \rm cusp,except}}(s,s'),
$$
where

$$
 S_{\text{\rm cusp,gen}}(s,s') := \frac{(4\pi)^k\G(1-s)}{2\G(s+k-1)}  \sum_{\Re \, t_j \ne 0}\frac{L_Q(s',\overline{u_j})\G(s - \frac12 -it_j)\G(s-\frac12 +it_j)\overline{\<V,u_j\>}}{\G(\frac12 -it_j)\G(\frac12 + it_j)},
$$
and 
$$
 S_{\text{\rm cusp,except}}(s,s') := \frac{(4\pi)^k\G(1-s)}{2\G(s+k-1)}  \sum_{\Re \, t_j = 0}\frac{L_Q(s',\overline{u_j})\G(s - \frac12 -it_j)\G(s-\frac12 +it_j)\overline{\<V,u_j\>}}{\G(\frac12 -it_j)\G(\frac12 + it_j)}.
$$
Referring to Proposition \ref{prop:ZQright}, we define  $ Z_{Q,\text{\rm gen}}(s,u) $ by
$$
Z_Q(s,u) = Z_{Q,\text{\rm gen}}(s,u) + Z_{Q,\text{\rm except}}(s,u),
$$
where
$$
 Z_{Q,\text{\rm except}}(s,u) :=  \left(\ell_1\ell_2\right)^{(k-1)/2} S_{\text{\rm cusp,except}}(s,s').
$$
With this notation, we rewrite \eqref{Sxydef2} as 
$$
S_Q(x,y) = S_{Q,\text{\rm gen}}(x,y) +S_{Q,\text{\rm except}}(x,y),  
$$
where
\begin{multline}\label{Sxydef3}
S_{Q,\text{\rm except}}(x,y) = \left(\frac{1}{2 \pi i}\right)^3\iiint\limits_{(1+\epsilon)(1+2\epsilon)(1+\epsilon)}
\frac{ Z_{Q,\text{\rm except}}(s,u)\G(w-u)\Gamma(u+(k-1)/2 )}{\Gamma(w+(k-1)/2)}\\ \times (\ell_2 x)^{s+u+(k-1)/2-w}(\ell_1y)^wg_1(w)g_2(s+u+(k-1)/2-w)du dw ds
\end{multline}
and 
\begin{multline}\label{Sxydef4}
S_{Q,\text{\rm gen}}(x,y) = \left(\frac{1}{2 \pi i}\right)^3\iiint\limits_{(1+\epsilon)(1+2\epsilon)(1+\epsilon)}
\frac{Z_{Q,\text{\rm gen}}(s,u) \G(w-u)\Gamma(u+(k-1)/2 )}{\Gamma(w+(k-1)/2)}\\ \times (\ell_2 x)^{s+u+(k-1)/2-w}(\ell_1y)^wg_1(w)g_2(s+u+(k-1)/2-w)du dw ds.
\end{multline}

To estimate $S_{Q,\text{\rm except}}(x,y)$ we first move the $u$ line of integration  to $\Re \, u = \epsilon -(k-1)/2$, passing over no poles in the process.   We now have $\Re \, s'  = 1/2 +2\epsilon$.  Moving the $w$ line to $\Re \, w = 0$, and referring to \eqref{targetend}, the definition of $Z_{Q,\text{\rm except}}(s,u)$, and \eqref{Sxydef3} we find that
\be\label{except}
S_{Q,\text{\rm except}}(x,y)
\ll 
Q^{\theta-1/2-\epsilon}L^{1+3\epsilon}x^{1+2\epsilon}.
\ee

To estimate $S_{Q,\text{\rm gen}}(x,y)$, we refer to \eqref{Sxydef4} and move the $s$ line of integration to $\Re \, s = \hf-\kt-\epsilon$.
Applying  Proposition \ref{prop:ZQright} and using the above notation we write
\begin{equation}\label{temp13}
S_{Q,\text{\rm gen}}(x,y)=S_Q^{(1)}(x,y)+S_Q^{(2)}(x,y)+S_Q^{(3)}(x,y),
\end{equation}
where
\begin{multline}\label{essq1}
S_Q^{(1)}(x,y):=\left(\frac{1}{2 \pi i}\right)^3\iiint\limits_{(\hf-\kt -\epsilon)(1+2\epsilon)(1+\epsilon)}
\mkern -20mu\frac{( Z_{Q,\text{\rm gen}}(s,u)-\Psi(s,u))\G(w-u)\Gamma(u+(k-1)/2 )}{\Gamma(w+(k-1)/2)}\\ \times (\ell_2 x)^{s+u+(k-1)/2-w}(\ell_1y)^wg_1(w)g_2(s+u+(k-1)/2-w)du dw ds  
\end{multline}
\begin{multline}\label{essq2}
S_Q^{(2)}(x,y):= \left(\frac{1}{2 \pi i}\right)^2 \mkern -10 mu \sum_{\substack{\Re \,t_j \ne0 \\ 0\leq r \leq \kt}} \ \iint\limits_{(1+2\epsilon)(1+\epsilon)}
\frac{c_{r,j}L_Q\left(-r+it_j +u +\frac{k-1}{2}, \overline{u_j }\right)\G(w-u)\Gamma(u+\frac{k-1}{2}) }{\Gamma(w+\frac{k-1}{2})}\\ \times (\ell_2 x)^{-r+it_j+u+\kt-w}(\ell_1y)^wg_1(w)g_2(-r+it_j+u+\tkt-w)du dw
\end{multline}
and, after a slight shift of the $w$ integral,
\begin{multline}\label{essq3}
S_Q^{(3)}(x,y)=\left(\frac{1}{2 \pi i}\right)^3\iiint\limits_{(1+\epsilon)(1+3\epsilon)(1+\epsilon)}
\mkern -20mu\frac{\Psi(s,u)\G(w-u)\Gamma(u+(k-1)/2 )}{\Gamma(w+(k-1)/2)}\\ \times (\ell_2 x)^{s+u+(k-1)/2-w}(\ell_1y)^wg_1(w)g_2(s+u+(k-1)/2-w) \ du dw ds.
\end{multline}

The last integral, $S_Q^{(3)}(x,y)$ is really just a shorthand for the residual contribution due to the continuous part of the spectrum, since $\Psi(s,u)$ is defined piecewise. 
To analyze $S_Q^{(1)}(x,y)$, let $s' = s+u +\kt -1$, so $\Re \, s' = \hf$.
By \eqref{Z2} in Proposition \ref{prop:ZQright} we have
$$
 Z_{Q,\text{\rm gen}}(s,u)- \Psi(s,u) \ll  Q^{\theta - {\Re \, s' +\epsilon}}(1+|s|)^{A''}(1 + |{s'}|)^{1+\epsilon}.
$$
As this has polynomial growth in $\Im \, s$ and $\Im\, u$, and there is decay faster than any polynomial due to the factors $g_1(w)g_2(s+u+(k-1)/2-w)$,
it follows that the double integral over $s$ and $u$ in $S_Q^{(1)}(x,y)$ converges and is bounded above by $Q^{\theta - \hf +\epsilon}$.  The integral over $w$ also converges, and its size is governed by the factor 
\be\label{Zref}
 (\ell_2 x)^{s+u+(k-1)/2-w}(\ell_1y)^w =  (\ell_2 x)^{s+u+(k-1)/2}((\ell_1y)/(\ell_2 x))^w. 
\ee
As $y \gg x$ and $\ell_1\sim \ell_2\sim L$, it is advantageous to move $w$ down past $\Re \, w = \Re \, u = 1+\epsilon$, picking up a pole at $w=u$.
Thus $S_Q^{(1)}(x,y)$ is finally bounded above by
\be\label{temp13.5}
Q^{\theta - \hf +\epsilon}L^{1+\epsilon}y(y/x)^{\epsilon}.
\ee
This is weaker than \eqref{except} above, unless $x=y$. 
\begin{remark}\label{remark:bottle}
The upper bound \eqref{temp13.5}, coming from the pole at $w = u$ is the bottleneck in this proof.    If would be desirable to move the $u$ line of integration as far past $1+ \epsilon$ as possible.   The difficulty is that this can only be done if one possesses an upper bound for $Z_{Q,\text{\rm gen}}(s,u)- \Psi(s,u)$ in the region $\Re \, u < 1$.   This is possible, using the meromorphic continuation of $Z_{Q,\text{\rm gen}}(s,u)$ given in \cite{HLee}.
\end{remark}

For $S_Q^{(2)}(x,y)$, move the $u$ line of integration to $\Re \, u = 1+r -\kt - \Re \, it_j$,   so $\Re \, (-r+it_j +u +\kt) = 1$.   Note that for $r\ge0$, the argument of that gamma function is
$$
\Re \, u +(k-1)/2= \Re \, 1+r -\tkt - \Re \, it_j +(k-1)/2 \ge \thf,
$$
so the first pole of $\Gamma(u+(k-1)/2 )$ is not passed.
The double integral over $u,w$ in  $S_Q^{(2)}(x,y)$ converges, as does the sum over the $t_j$.
Referring to \eqref{targetend}, we see that an upper bound for  $S_Q^{(2)}(x,y)$   is given by 
$$
Q^{\theta - \hf+\epsilon}L^{1+\epsilon}y(y/x)^{r-\kt- \Re \, it_j},
$$
which is an improvement over \eqref{temp13.5} as $r-\kt- \Re \, it_j \le 0$.

We now consider $S_Q^{(3)}$. Changing variables from $u$ to $s'$, with $s' = s+u +\kt-1$, gives us
\begin{multline*}
S_Q^{(3)}(x,y)=\left(\frac{1}{2 \pi i}\right)^3 \mkern -20 mu \iiint\limits_{(1+\epsilon)(1+3\epsilon)(1+\kt+2\epsilon)}
\mkern -20mu\frac{\Psi(s,s'-s-\kt+1)\G(w-s'+s+\kt-1)\Gamma(s'-s+\hf )}{\Gamma(w+(k-1)/2)}\\ \times (\ell_2 x)^{s'-w+\hf}(\ell_1y)^wg_1(w)g_2(s'-w+\thf) \ ds' dw ds.
\end{multline*}
We now move the $s'$ line of integration to $\Re \, s' = \hf+2\epsilon$.	Here $\Re \, s' -s+\hf =\hf+2\epsilon -1-\epsilon +\hf =  \epsilon$, so the pole of $\G(s'-s+\hf)$ where $s'-s+\hf=0$  is not passed over.  However, from \eqref{Psidef} we do get that
\begin{multline}\label{temp15}
S_Q^{(3)}(x,y)=\left(\frac{1}{2 \pi i}\right)^3 \mkern -20 mu \iiint\limits_{(1+\epsilon)(1+3\epsilon)(\frac{1}{2}+2\epsilon)}
\mkern -20mu\frac{\Psi(s,s'-s-\kt+1)\G(w-s'+s+\kt-1)\Gamma(s'-s+\hf )}{\Gamma(w+(k-1)/2)}\\ \times (\ell_2 x)^{s'-w+\hf}(\ell_1y)^wg_1(w)g_2(s'-w+\thf) \ ds' dw ds 
+R(x,y),
\end{multline}
where
\begin{multline*}
R(x,y):= \sum_{\mathfrak{a}}\left(\frac{1}{2 \pi i}\right)^3 \mkern -20 mu \iiint\limits_{(1+\epsilon)(1+3\epsilon)(1)}
\mkern -10mu 
\frac{(4\pi)^k\G(1-s)\G(s + \frac12 -s')\G(s- \frac32 +s')}
	{2Q^{s'}\G(s+k-1)\pi^{\hf -s'}\zeta^*(3-2s')\G( s'-\frac12)^2\G( \frac32  -s')}  
	\\
	\times 
	\lt(K_{\cusp, Q}^+(s')
	\overline{\<V,E_\cusp^*(*,\tfrac 32 -s')\>}
	+K_{\cusp, Q}^-(s')
	\overline{\<V,E_\cusp^*(*,s' -\tfrac 12)\>}\rt) 
	\,
\\ \times \frac{\G(w-s'+s+\kt-1)\Gamma(s'-s+\hf )}{\Gamma(w+(k-1)/2)}(\ell_2 x)^{s'-w+\hf}(\ell_1y)^wg_1(w)g_2(s'-w+\thf) \ ds' dw ds
\end{multline*}
and $K_{\mathfrak{a},Q}^\pm(s')$ is given by \eqref{zetaQexact} and \eqref{bigKdef}. 
By changing variables $s' \to 1+z$ and the order of integration we get
\begin{multline*}
R(x,y)= \sum_{\mathfrak{a}}\left(\frac{1}{2 \pi i}\right)^3 \mkern -20 mu \iiint\limits_{(0)(1+\epsilon)(1+3\epsilon)}
\mkern -10mu 
\frac{(4\pi)^k\pi^{\hf +z}\G(1-s)\G(s - \hf-z)\G(s-\hf +z)}
	{2Q^{1+z}\G(s+k-1)\zeta^*(1-2z)\G(\hf+z)^2\G( \hf-z)}  
	\\
	\times 
	\lt(K_{\cusp, Q}^+(1+z)
	\overline{\<V,E_\cusp^*(*,\tfrac12 -z)\>}
	+K_{\cusp, Q}^-(1+z)
	\overline{\<V,E_\cusp^*(*,\thf+z)\>}\rt) 
	\,
\\ \times\frac{\G(w-2+\kt- z+s)\Gamma(3/2 + z-s )}{\Gamma(w+(k-1)/2)}  (\ell_2 x)^{3/2 + z-w}(\ell_1y)^wg_1(w)g_2(3/2 + z-w) \  dw ds dz. \end{multline*}

Now taking the first integral component of $S_Q^{(3)}(x,y)$ as given in \eqref{temp15}, we collect the different components of $\Psi(s,s'-s-\kt+1)$ to get that
\begin{multline*}
S_Q^{(3)}(x,y) = \sum_{\substack{\mathfrak{a}, \\ 0 \leq r \leq \kt }} \mkern -10mu \left(\frac{1}{2 \pi i}\right)^3 \mkern -20 mu \iiint\limits_{(\hf-r)(1+3\epsilon)(\frac{1}{2}+2\epsilon)} \frac{(-1)^r(4\pi)^k\vol Q^{-s'}\Gamma(1-s)\Gamma(2s+r-1)}{2r!\Gamma(s+k-1)\G(s+r)\G(1-s-r)}\\ 
\times\left(
	\frac{ \zeta_{\cusp, Q} \left(s', -\frac{1}{2}+s+r \right) \overline{\langle V, E_\cusp^*\left(*, s+r \right)\rangle } } 
	{ \zeta^* (2-2s-2r) \zeta^*(2s+2r) }   
+
	\frac{ \zeta_{\cusp, Q} \left(s', \frac{1}{2} -s-r\right)\ol{\langle V,E^*_\cusp(*,1-s-r ) \rangle}}
	{\zeta^*\left(2s+2r \right)\zeta^*(2-2s-2r )} 
  \right) \\
\times  \frac{\G(w-s'+s+\kt-1)\Gamma(s'-s+\hf )}{\Gamma(w+(k-1)/2)}  (\ell_2 x)^{s'-w+\hf}(\ell_1y)^wg_1(w)g_2(3/2+z-w) \ds'\dw \ds \\
+R(x,y).
\end{multline*}
By substituting $s \to \hf-r+z$ and shifting the line of integration for $s'$ to $\Re s' = \hf$ we get
\begin{multline*}
S_Q^{(3)}(x,y) = \sum_{\substack{\mathfrak{a}, \\ 0 \leq r \leq \kt }} \mkern -10mu \left(\frac{1}{2 \pi i}\right)^3   \iiint\limits_{(0)(1+3\epsilon)(\frac{1}{2})} \frac{(-1)^r(4\pi)^k\vol Q^{-s'}\Gamma(\hf-z)\Gamma(2z-r)}{2r!\Gamma(z+k-r-\hf)\G(\hf+z)\G(\hf-z)}\\ 
\times\left(
	\frac{ \zeta_{\cusp, Q} \left(s', z \right) \overline{\langle V, E_\cusp^*\left(*, \hf +z \right)\rangle } } 
	{ \zeta^* (1-2z) \zeta^*(1+2z) }   
+
	\frac{ \zeta_{\cusp, Q} \left(s', -z\right)\ol{\langle V,E^*_\cusp(*,\hf-z ) \rangle}}
	{\zeta^*\left(1+2z \right)\zeta^*(1-2z )} 
  \right) \\
\times  \frac{\G(w+z-s'-r+\tfrac{k-1}{2})\Gamma(s'-r-z)}{\Gamma(w+(k-1)/2)}  (\ell_2 x)^{s'-w+\hf}(\ell_1y)^wg_1(w)g_2(s'-w+\thf) \ ds' dw dz \\
+R(x,y).
\end{multline*}

Our method for bounding each of the integrals above, except for $R(x,y)$, is completely analogous to how we bounded $S_Q^{(2)}(x,y)$.  By shifting the line of integration $w$ left past the pole at $w+z-s'-r+\frac{k-1}{2}=0$, we see that 
our bound for these integrals is also an improvement on  \eqref{temp13.5}. 

The  $z, s$ and $w$ integrals in $R(x,y)$, as in \eqref{temp15}, are all convergent.  By shifting the line of integration for $z$ to $\Re \, z = -\hf-\frac{\epsilon}{2}$ we get a bound that is an improvement on \eqref{temp13.5} and our proof is complete.  
\end{proof}

\section{Application to subconvexity}\label{sect:subcon}
Our amplification argument is modeled on that of Blomer in \cite{B}, which is in turn based on that of \cite{DFI3}.  Recall we have fixed  a positive integer $Q$, with $(Q,N)=1$.  Recall also that we have fixed $N_0,\ell_1,\ell_2 \in\N$,  with $N_0,\ell_1,\ell_2$ square free, $(N_0,\ell_1\ell_2)=1$ and have set $N=N_0\ell_1\ell_2/(\ell_1,\ell_2)$.   In this section we will add the requirement that $\ell_1,\ell_2$  be prime.  Let $\chi$ be a primitive character mod $Q$ then for $\Re s >1$ we have
$$
L(s,f\otimes\chi)=\sum_{m\ge1}{A(m)\chi(m)\over m^s}.
$$
Using the approximate functional equation, we have that
$$
L(\thf,f\otimes\chi)= 
\sum_m {A(m)\chi(m)\over \sqrt m} V\left({m\over Q\sqrt{N_0}}\right)
+
\gw
\sum_m {\bar A(m)\bar\chi(m)\over \sqrt m} V\left({m\over Q \sqrt{N_0 }}\right),
$$
where $\omega$ is the root number.
Let
$$
B_\chi(x)=\sum_m A(m)\chi(m)\ G\left({m\over x}\right),
$$
where $G$ is smooth of compact support in $[1,2]$ and $x$ is a parameter.
By a smooth dyadic partition of unity, and summation by parts, we have
\be\label{Lbnd}
L(\thf,f\otimes\chi)\ll \max_{x \ll Q^{1+\epsilon}} Q^{-\hf}B_\chi(x).
\ee

Set $L= Q^{1/4-\theta/2+\epsilon}$, and for any $x \ll Q$,
consider the amplified sum 
\be\label{Sdef}
S= \sum_{\psi (Q)}\left|\sum_m A(m)\psi(m)G\left({m\over x}\right)\right|^2
\left|\sum_{\ell \sim L, \ell \,\text{prime}, \, (\ell, N_0)=1} \bar\chi(\ell)\psi(\ell)\right|^2.
\ee
Notice that the term with $\psi=\chi$ contributes 
$$
|B_\chi(x)|^2
\left|\sum_{\ell \sim L, \ell \,\text{prime}, \, (\ell, N_0)=1}1\right|^2.
$$

The number of $\ell\sim L$ with $\ell$ prime is asymptotic to
$$
 L/\log(L) \sim \frac{ Q^{1/4-\theta/2+\epsilon}}{\log(Q^{1/4-\theta/2+\epsilon})},
$$
so taking only the contribution to $S$ from $\psi=\chi$ gives
\be\label{Bbnd}
B_\chi(x)\ll S^{\hf}Q^{-1/4+\theta/2}.
\ee
Combining this  with \eqref{Lbnd} we now have
\be\label{Lbnd2}
L(\thf,f\otimes\chi)\ll  \max_{x \ll Q}Q^{-3/4+\theta/2} S^{\hf}.
\ee

If we open the squares in $S$, and bring the sum over $\psi$ inside, we obtain,
\begin{multline}
S= \sum_{m_1} \sum_{m_2} \sum_{\ell_1}\sum_{\ell_2} A(m_1)\overline{A(m_2)} G\lt(\frac {m_1} x\rt) G\lt(\frac {m_2}{x}\rt) \overline{\chi(\ell_1)}\chi(\ell_2)  \sum_\psi \psi(\ell_1m_1) \overline{\psi(\ell_2m_2)}\\
=\phi(Q)  \sum_{\substack{m_1,m_2,\ell_1,\ell_2, (\ell_1m_1,Q)=1 \\\ \ell_1m_1 \equiv \ell_2m_2 \pmod Q} }  A(m_1)\overline{A(m_2)} G\lt(\frac {m_1} x\rt) G\lt(\frac {m_2}{x}\rt) \overline{\chi(\ell_1)}\chi(\ell_2). 
\end{multline}

We can remove the restriction $(\ell_1m_1,Q)=1$ by letting
\begin{align}
S& = \phi(Q)\sum_{\substack{a \text{ mod }Q\\ (a,Q)=1}}  \lt|\sum_m \sum_{\substack{\ell \sim L, \ell \text{ prime},(\ell,N_0)=1\\ m\ell \equiv a \text{ mod } Q}} A(m)G\lt(\frac mx\rt) \overline{\chi(\ell)} \rt|^2 \\
 &\leq \phi(Q)\sum_{a \text{ mod }Q} \lt|\sum_m \sum_{\substack{\ell \sim L, \ell \text{ prime},(\ell,N_0)=1\\ m\ell \equiv a \text{ mod } Q}} A(m)G\lt(\frac mx\rt) \overline{\chi(\ell)} \rt|^2 \notag \\
& \ \ = \phi(Q) \sum_a \sum_{m_1} \sum_{m_2} \sum_{\ell_1}\sum_{\substack{\ell_2\\ m_1\ell_1 \equiv m_2\ell_2 \equiv a  \text{ mod } Q}} A(m_1)\overline{A(m_2)} G\lt(\frac{m_1}{x}\rt)G\lt(\frac{m_2}{x} \rt)  \notag\chi(\ell_1)\overline{\chi(\ell_2)},
\end{align}
which can be expressed as,
\[
S \leq \phi(Q)\sum_{\substack{\ell_1,\ell_2\sim L \text{ primes} \\ (\ell_1\ell_2,N_0)=1}}\chi(\ell_1)\overline{\chi(\ell_2)}\sum_{\substack{m_1,m_2\\ m_1\ell_1 \equiv m_2\ell_2 \text{ mod }Q}}  A(m_1)\overline{A(m_2)}G\lt(\frac{m_1}{x}\rt)G\lt(\frac{m_2}{x}\rt).
\]

This sum, over all $m_1,m_2$ such that $\ell_1m_1\equiv \ell_2 m_2\pmod Q$,  can be broken into three pieces, two of which are treatable by our Proposition~\ref{prop:notprime}.
\be\label{Sfinal}
S \ll \phi(Q)
\sum_{\ell_1} \sum_{\ell_2} \chi(\ell_1)\overline{\chi(\ell_2)} (S_1+S_2+S_3),
\ee
where
$$
S_1 = \sum_{\substack{m_2\ge 1 \\ \ell_1m_1 = \ell_2m_2}}  A(m_1)\bar A(m_2) 
G\left({m_1\over x}\right) G\left({m_2\over x}\right), \ \ 
$$
$$
S_2 =\sum_{\substack{m_2,h_0\ge 1 \\ \ell_1m_1 = \ell_2m_2 +h_0Q}}  A(m_1)\bar A(m_2) 
G\left({m_1\over x}\right) G\left({m_2\over x}\right),
$$
and
$$
S_3 =\sum_{\substack{m_2,h_0\ge 1 \\ \ell_2m_2 = \ell_1m_1 +h_0Q}}  A(m_1)\bar A(m_2) 
G\left({m_1\over x}\right) G\left({m_2\over x}\right).
$$
It is easily verified that
$$
\phi(Q)\sum_{\ell_1} \sum_{\ell_2} \bar\chi(\ell_1)\chi(\ell_2) S_1\ll LQ^{2 +\epsilon}.
$$   
The sums $S_2,S_3$ are bounded using Theorem ~\ref{thm:sum}, with $x=y \ll Q^{1+\epsilon}$.  This gives us
$$
S_2,S_3 \ll Q^{\theta-\hf}x^{1+\epsilon}L^{1+\epsilon} \ll Q^{\hf+\theta+\epsilon}L^{1+\epsilon} .
$$
Substituting, we have for their contribution:
$$
\phi(Q)
\sum_{\ell_1} \sum_{\ell_2} \bar\chi(\ell_1)\chi(\ell_2) (S_2 + S_3) \ll QL^2Q^{\hf+\theta+\epsilon}L^{1+\epsilon}
\ll Q^{3/2+\theta+\epsilon}L^{3+\epsilon} .
$$
Referring to \eqref{Sfinal} we have, 
$$
S \ll \phi(Q)
\sum_{\ell_1} \sum_{\ell_2} \bar\chi(\ell_1)\chi(\ell_2) (S_1+S_2+S_3)
\ll Q^{2+\epsilon} L+ Q^{3/2+\theta+\epsilon}L^{3+\epsilon}
\ll Q^{9/4-\theta/2+\epsilon}.
$$
Finally, referring to \eqref{Lbnd2},
$$
L(\thf,f\otimes\chi)\ll  \max_{x \ll Q^{1+\epsilon}}Q^{-3/4+\theta/2} S^{\hf} \ll Q^{3/8+\theta/4+\epsilon}.
$$
This completes the proof of Theorem~\ref{thm:sub}. 

\ \ 

{\bf{\hfill Appendix: Estimates of triple products of automorphic functions II \hfill}}

{\hfill{Andre Reznikov}\hfill}

\begin{abstract}
We prove a sharp bound for the average value of  the triple product of modular functions for the Hecke subgroup $\G_0(N)$.  Our result is an extension of the main result in \cite{BR4} to a {\it fixed} cuspidal representation of the adele group $\PGL(\ba)$.
\end{abstract}
\let\thefootnote\relax\footnote{Partially supported by  the ERC grant 291612, by the ISF   grant 533/14, and by the Minerva  Center at ENI}
\maketitle

\section{Appendix}
\label{intro}
\subsection{Maass forms} We recall the setup of \cite{BR4} which should be read in conjunction with this appendix.
Let $Y$ be a compact Riemann surface  with a Riemannian metric of
constant curvature $-1$ and the associated volume element $dv$.
The corresponding  Laplace-Beltrami operator  is non-negative and
has purely discrete spectrum on the space $L^2(Y,dv)$ of functions
on $Y$. We will denote by $0=\mu_0< \mu_1 \leq \mu_2 \leq ...$ its
eigenvalues  and by $\phi_i$ the corresponding
eigenfunctions (normalized to have $L^2$ norm one). In the theory
of automorphic forms the functions $\phi_i$ are called
automorphic functions or Maass forms (after H. Maass,
\cite{M}).  We write $\mu_i=(1-\lm_i^2)/2$  and $\phi_i=\phi_{\lm_i}$ as  is customary in representation theory of the group
$\PGLR$.

For any three Maass forms
$\phi_i,\ \phi_j, \ \phi_k$ we define the following
triple product or triple period:
\begin{eqnarray}\label{cijk}
c_{ijk}=\int_Y\phi_i\phi_j\phi_kdv \ .
\end{eqnarray}

One would like to bound the coefficient $c_{ijk}$ as a function of
eigenvalues  $\mu_i,\ \mu_j,\ \mu_k$. In particular, we would like
to find bounds for these coefficients  when one or more of these
indices tend to infinity. The study of these triple coefficients goes back  to
pioneering works of Rankin and Selberg (see \cite{Ra}, \cite{Se}), and reappeared in celebrated works of Waldspurger \cite{W} and Jacquet \cite{J} (also see  \cite{HK}, \cite{W}, \cite{Ich}). Recently, an interest in analytic questions related to triple products was initiated in the groundbreaking paper of Sarnak  \cite{Sa1} (see also  \cite{Gd2} for the first general result on exponential decay). This was motivated by the widespread use of triple products in applications  (e.g., \cite{Sa2}).

In  \cite{BR4} we considered the following particular case of the triple period problem. Namely, we fix two
Maass forms $\phi = \phi_\tau,\ \phi'= \phi_{\tau'}$ as above and
consider coefficients defined by the triple period:
\begin{eqnarray}\label{ci}
c_i=\int_Y\phi\phi'\phi_idv
\end{eqnarray}
as  $\{\phi_i = \phi_{\lm_i}\}$ run over the orthonormal basis of Maass forms.
We note first that one  has exponential decay for the coefficients $c_i$
in the parameter $|\lm_i|$ as $i$ goes to $\8$. For that reason, one renormalizes coefficients $|c_i|^2$ by an appropriate ratio of Gamma functions dictated by the Watson formula \cite{W} (see also Appendix in \cite{BR4} where these factors were computed using another point of view). Taking into account the asymptotic behavior of these factors, we introduced normalized coefficients
\begin{eqnarray}\label{ai-ci}
|a_i|^2=|\lm_i|^2\exp\left(\frac{\pi}{2}|\lm_i|\right)\cdot |c_i|^2\ .
\end{eqnarray} Under such a normalization, we showed that
\begin{eqnarray}\label{mean-value-ai}
\sum_{ |\lm_i|\leq T}|a_i|^2\leq A\cdot T^2\
\end{eqnarray} for some explicit constant $A=A(\G,\phi,\phi')$.
According to the Weyl law, there are approximately $cT^2$ terms in the above sum, and hence the bound \eqref{mean-value-ai} is consistent with the Lindel\"{o}f  bound on average (in fact it is not difficult to show that the bound \eqref{mean-value-ai} essentially is  sharp, see \cite{Re1}).

There are various natural questions concerning the bound \eqref{mean-value-ai} which were not discussed in \cite{BR4}. These are mostly related to the dependence of the constant $A$ on various parameters (i.e., dependence on the subgroup $\G$ and forms $\phi$, $\phi'$), and also to the fact that we restricted the discussion to Maass forms, leaving aside the case of holomorphic forms. Another restriction of  the treatment we presented was  the fact that we used in an essential way the compactness of $Y$ (although the non-compact case was treated by us differently in \cite{BR1}).  All these issues turn out to be important in applications. In this appendix, we answer some of these questions for Hecke congruence subgroups. The methods we employ are elaborations on the method of \cite{BR1}.
\subsubsection{Hecke subgroups} For an integer $N\geq 1$, we consider the Hecke subgroup $\G_0(N)\subset\PSLZ$ of the modular group. We normalize the scalar product on the quotient Riemann surface $Y_{N}=\G_0(N)\setminus \uH$ by $\langle f,g\rangle_{Y_N}=
{\rm vol}_\uH(Y_N)\inv\int_{Y_{N}}f(z)\bar g(z)d\mu_\uH$  where $d\mu_\uH$ is the standard volume element on the upper half plane $\uH$ (i.e., we normalize the volume element $dv_{Y_{N}}$ on $Y_N$ to have the total volume $1$). Let $\phi$ be a (primitive) Hecke-Maass form for the group $\G_0(N_0)$ for some fixed level $N_0\geq 1$. We assume that $\phi$ is normalized by the $L^2$-norm $||\phi||_{L^2(Y_{N_0})}=1$.  For an integer $\ell> 1$, denote by $^\ell\phi(z)=\phi(\ell z)$  the corresponding {\it old form} for the Hecke subgroup $\G_0(\ell N_0)$. The corresponding function $^\ell\phi$  also turns out  to be  $L^2$-normalized on $Y_{\ell N_0}$ with respect to the normalization of measures we choose. This follows easily from the Rankin-Selberg method. For two such Maass forms $\phi$ and $\phi'$, we define  triple products by
\begin{eqnarray}\label{ci-ell}
c_i(\ell)=\int_{Y_{\ell N_0}}{^\ell\phi}\ {^\ell\phi'}\ \phi_i\ dv_{Y_{\ell N_0}}
\end{eqnarray}
as  $\{\phi_i = \phi_{\lm_i}\}$ run over the orthonormal basis of Maass forms on $Y_{\ell N_0}$. We have the corresponding  normalized coefficients
\begin{eqnarray}\label{ai-ell}
|a_i(\ell)|^2=|\lm_i|^2\exp\left(\frac{\pi}{2}|\lm_i|\right)|c_i(\ell)|^2\ .
\end{eqnarray}

\begin{thma}{THM1}
There exists an effectively computable
constant $A$  such that the following bound holds for all $T \geq 1$,
\begin{eqnarray}\label{mean-value}
\sum_{ |\lm_i|\leq T}|a_i(\ell)|^2\leq A\cdot T^2\ ,
\end{eqnarray} where the summation is over an orthonormal basis of Maass forms for the group $\G_0(\ell N_0)$.
The constant $A$ depends on $N_0$, $\phi$ and $\phi'$, but not on $\ell$.
\end{thma}

This could be viewed as a Lindel\"{o}f on the average type bound in two parameters $\ell$ and $T$. Namely, there are about $[\G_0(\ell N_0):\G(1)]\cdot  T^2$ terms in the sum as predicted by the Weyl-Selberg law for the surface $Y_{\ell N_0}$ (assuming the uniformity of the corresponding remainder term). Hence  the resulting bound \eqref{mean-value} is consistent with the expected Lindel\"{o}f type bound $|a_i(\ell)|\ll (\ell |\lm_i|)^\eps$ under our normalization of measures   ${\rm vol}(Y_{\ell N_0})=1$. 

There is also an analogous contribution from the Eisenstein series. One defines normalized triple coefficients $a_{s,\kappa}(\ell)$ arising from triple products $\langle {^\ell\phi}\ {^\ell\phi'}, E_{\kappa}(s)\rangle_{Y_{\ell N_0}}$ where $E_{\kappa}(s)$ is the Eisenstein series associated to a cusp $\kappa$ of $\G_0(\ell N_0)$. We prove the following bound for the full spectral expansion
\begin{eqnarray}\label{mean-value-Eisen}
\sum_{ |\lm_i|\leq T}|a_i(\ell)|^2+\haf\sum_{\{\kappa\}}\int_{-T}^{T}|a_{it,\kappa}(\ell)|^2dt \leq A\cdot T^2\ .
\end{eqnarray}

\subsection {The method}
The proof we presented in \cite{BR4} was based on the uniqueness of the triple
product in representation theory of the group $\PGLR$.
We review quickly the general ideas behind our  proof. It is
based on ideas from representation theory (see the seminal book \cite{GGPS}, and also \cite{Bu}, \cite{L}). Namely, we use the fact
that every automorphic form $\phi$ generates an automorphic
representation of the group $G = \PGLR$; this means that starting
from $\phi$ we produce a smooth irreducible representation of the
group $G$ in a space $V$ and its realization $\nu : V \to
C^{\8}(X)$ in the space of smooth functions on the automorphic
space $X = \G \backslash G$ endowed with the invariant measure of the total mass one. We denote by $\phi_v(x)=\nu(v)(x)$ the corresponding to $v\in V$ automorphic function. The Maass form corresponds to a unit $K$-invariant vector $e_0\in V$, where $K=PSO (2)$ is a maximal compact connected subgroup of $G$. In our case, we have the family of spaces $X_\ell=\G_0(\ell N_0)\setminus G$ (we denote by $X_0=\G_0( N_0)\setminus G$), and the corresponding family of  isometries $\nu_\ell: V\to C^{\8}(X_\ell)$ of the {\it same} abstract representation of $G$. We note that forms $\phi(\ell z)$  belong to the same {\it adelic} representation of $\PGL(\ba)$ generated  the new  Hecke-Maass form $\phi$.

 The triple product  $c_i=\int_Y\phi\phi'\phi_idv$ extends to a
$G$-equivariant trilinear form on the corresponding automorphic
representations  $l^{aut}:V\otimes V'\otimes V_i\to\bc$, where $V
= V_\tau, V' = V_{\tau'}, V_i = V_{\lambda_i}$ .

Then we use a general result from representation theory that such a
$G$-equivariant trilinear form is unique up to a scalar, i.e., that  $\dim\Hom_G(V\otimes V'\otimes V_i,\bc)\leq 1$ (see \cite{O},\cite{Mo}, \cite{Lo} and \cite{Pr} for $p$-adic $GL(2)$).
This
implies that the automorphic form $l^{aut}$ is proportional to an
explicit ``model" form $l^{mod}$ which we describe using explicit
realizations of representations of the group $G$; it is important
that this last form carries no arithmetic information.

Thus we can write $l^{aut} = a_i \cdot l_i^{mod}$ for some constant
$a_i$ and hence $c_i= l_i^{aut}(e_\tau\otimes e_{\tau'} \otimes
e_{\lambda_i}) = a_i \cdot l_i^{mod}(e_\tau \otimes e_{\tau'}
\otimes e_{\lambda_i})$, where $e_\tau,\ e_{\tau'},\
e_{\lambda_i}$ are $K$-invariant unit vectors in the automorphic
representations $V, V',\ V_i$ corresponding to the automorphic
forms $\phi$, $\phi'$ and $\phi_i$.

 It turns out that the proportionality coefficient
 $ a_i$ in the last formula carries
 important ``automorphic" information while the second factor
 carries no arithmetic information and can be computed
 in terms of Euler $\G$-functions using explicit realizations of
 representations $V_\tau$, $V_{\tau'}$ and $V_{\lambda_i}$ (see Appendix to \cite{BR4}). This second factor is
 responsible for
 the exponential decay, while the first factor $a_i$ has a
 polynomial behavior in the parameter $\lm_i$.

   In order to bound the quantities $a_i$, we use the fact that they
   appear as coefficients in the spectral
   decomposition of the diagonal Hermitian form $H_{\Dl}$ given by $$H_\Dl(v\otimes w)=\int_X|\phi_v(x)\phi_w(x)|^2dx$$
   on the space $E = V_\tau \otimes V_{\tau'}$.
   This gives an inequality $\sum |a_i|^2 H_i \leq H_\Dl$
   where $H_i$ is a Hermitian
form on $E$ induced by the model trilinear form $l_i^{mod}: V
\otimes V' \otimes V_i \to \bc$ as above.

Using the geometric properties of the diagonal form and simple explicit estimates of forms $H_i$,
we establish the mean-value bound for the coefficients $|a_i|^2$.
Here is where one obtains the dependence of the constant $A$ in \eqref{mean-value-ai} on parameters involved. In the method of \cite{BR4}, we used $L^2$ theory by averaging the form $H_\Dl$ and comparing the resulting form with the $L^2$-form. The coefficient $A$ that one obtains in such an argument depends in particular on the injectivity radius of $X$. While in certain cases it gives an optimal result, it obviously has two drawbacks. One is related to the possible non-compactness of $X$ since in the cusp the injectivity radius tends to zero.  Another problem arises when one considers a sequence of subgroups with co-volume going to infinity along the sequence. In that case, the bound which the method of \cite{BR4} provides for the constant (e.g., $A(\G_0(p))\leq {\rm vol}(Y_p)\approx p$, see \cite{K})  is too weak for many applications. Both of these problems arise in the classical setup of Hecke subgroups $\G_0(N)$. Here we obtain an optimal bound for {\it old forms}. We do not know how to obtain similar results for new forms. We will discuss improvements over a trivial bound for new forms elsewhere. Theorem A (see Section \ref{THM1}) could be viewed as the exact analog of the result in \cite{BR4} for a {\it fixed} adelic representation.

\subsection{Proof of Theorem A} We have a family of objects  $(X_\ell, \nu_\ell, H_\Dl^{\ell}, a_i(\ell))$ parameterized by the level $\ell$. However the {\it model } Hermitian form $H_i$ is the same since the abstract representation $V$ of $\PGLR$ does not change.
The proof of the bound \eqref{mean-value-ai} given in \cite{BR4} was based on the spectral decomposition $\sum |a_i(\ell)|^2 H_i \leq H^\ell_\Dl$ of the diagonal form and on the construction of the test vector $u=u_T\in E$ such that $H^\ell_\Dl(u)\leq aT^2$ and $H_i(u)\geq 1$ for $|\lm_i|\leq T$. We now construct the test vector {\it independently} of $\ell$. The dependence on $\ell$ is hidden in the automorphic realization $\nu_\ell(u)$ as a function on $X_\ell$.
\subsubsection{Construction of vector $u$}\label{consttest} We slightly change the construction of the test vector $u_T$ given in Section 5.3.2 of \cite{BR4}.
Let us identify the space $E = V \otimes V'$ with a  subspace of smooth functions $C^\8(\br\times\br)$. Choose a smooth non-negative function $\al\in C^\8(\br)$ with the support ${\rm supp}(\al)\subset [-\iota,\iota]$ for $\iota>0$ to be specified later,  and $\int_\br\al(x)dx=1$. Let $||\al||^2_{L^2(\br)}=c^2$ for some $c>0$. Consider the diagonal element $a_T={\rm diag}(T^{-\haf},T^\haf)\in G$. We define  vectors
\begin{eqnarray}\label{test-vect}
v_T=T^{\haf+\tau}\cdot\pi_\tau(a_T)\al\ \ {\rm and}\ \  v'_T=T^{\haf+\tau'}\cdot\pi_{\tau'}(\left(
                                          \begin{smallmatrix}
1 & 1 \\
                                   & 1 \\
\end{smallmatrix}
                                        \right)a_T) \al\ .
                                        \end{eqnarray}
We set our test vector to $u_T(x,y)= v_T(x)\otimes v'_T(y)$.  Recall that the action is given by $\pi_\tau({\rm diag}(a\inv,a) v(x)=|a|^{1-\tau}v(a^2x)$, and $\pi_{\tau}(\left(
                                          \begin{smallmatrix}
1 & 1 \\
                                   & 1 \\
\end{smallmatrix}
                                        \right)) v(x)=v(x-1)$.  By an easy calculation we have $||u_T||^2_E=c^2T^2$. We note that geometrically the vector $u_T$ is a small non-negative bump function around the point $(0,1)\in \br^2$, with the support in the box of the size $\iota T\inv$, and satisfies $\int_{\br^2}u_T(x,y)dxdy=1$.
A computation identical to the one performed in Section 5.3.4 of \cite{BR4} gives then for a small enough (but fixed) $\iota>0$, $H_i(u_T)\geq \beta$   with some explicit $\beta>0$  independent of $|\lm_i|\leq T$ (in fact for $|\tau|,\ |\tau'|\leq T$, it does not depend on these either). We remark that the only difference with the construction of the test vector given in \cite{BR4} is that here we constructed $u_T$ with the help of the action of $G$ on $V$ (while in \cite{BR4} we constructed essentially the same vector explicitly in the model). This will play a crucial role in our estimate of the corresponding automorphic function.

We now need to estimate $H^\ell_\Dl(u_T)$.  We claim that $H^\ell_{\Dl}(u_T)\leq B T^2$ for some explicit constant $B$ independent of $\ell$. Since we have
$$H^\ell_{\Dl}(u_T)=\int_{X_\ell}|{^\ell\phi_{v_T}}{^\ell\phi_{v'_T}}|^2dx\leq \haf||{^\ell\phi_{v_T}^2}||^2_{L^2(X_\ell)} +\haf||{^\ell\phi_{v'_T}^2}||^2 _{L^2(X_\ell)}\ ,$$ it is enough to show that $||{^\ell\phi_{v_T}^2}||^2_{L^2(X_\ell)}\leq \beta' T^2$.  This would finish the proof of Theorem A (Section \ref{THM1}) following the argument  in Section 4.7 \cite{BR4}.  Since $\langle {^\ell\phi_{v_T}^2}, 1\rangle_{X_\ell}=||{^\ell\phi_{v_T}}||^2=c T$, it is easy to see that such a bound is sharp. We claim that
\begin{eqnarray}\label{supnorm}
\sup_{x\in X_\ell}|{^\ell\phi_{v_T}}(x)|\leq \beta'' T^\haf\ ,
\end{eqnarray} for some $\beta''$ independent of $\ell$. Here  ${^\ell\phi_{v_T}}(x)=\nu_\ell(v_T)(x)$. Note that the bound provided by the Sobolev theorem \cite{BR2} would give a much weaker in $\ell$  bound of the order of $(\ell T)^\haf$ (see \cite{K} for the corresponding discussion).

\subsubsection{Supremum norm}\label{supnorm-sect} Recall that we started with an $L^2$-normalized Hecke-Maass form $\phi$ on $\G_0(N_0)$, and the corresponding isometry of $\nu:V\to C^\8(X_0)$ of the principal series representation $V\simeq V_\tau$. We then constructed another isometry $\nu^\ell:V\to C^\8(X_{\ell})$ by using the map
\begin{eqnarray}\label{nu-ell-nu}
^\ell\phi_v(x)=\nu^\ell(v)(x)=\nu(v)\left(\left(
                                          \begin{smallmatrix}
\ell^\haf &  \\
                                   &\ \ \ell^{-\haf} \\
\end{smallmatrix}
                                        \right)x\right)=\phi_v\left(\left(
                                          \begin{smallmatrix}
\ell^\haf &  \\
                                   &\ \ \ell^{-\haf} \\
\end{smallmatrix}
                                        \right)x\right)\ ,
\end{eqnarray}        for any $v\in V$.    This relation might be viewed as the relation between  functions on $G$ invariant on the left for an appropriate $\G$ (e.g., for  $\G_0(N_0)$ and for $\G_0(\ell N_0)$). In particular, we see that the supremum of the function $^\ell\phi_v$ on $X_\ell$ and that of the function $\phi_v$ on $X_0$  are equal for the {\it same} vector $v\in V$. Hence it is enough to show that $\sup_{x\in X_0}|{\phi_{v_T}}(x)|\leq \beta'' T^\haf$. In fact, this is obvious since  $v_T=T^{\haf+\tau}\pi(a_T)\al$ is given by the (scaled) action of $G$ on a {\it fixed} vector.  We have
\begin{eqnarray*}\label{sup-action}
\sup_{X_0}|{\phi_{v_T}(x)}|=T^{\haf}\sup_{X_0}|{\phi_{\pi(a_T)\al}(x)}|=
T^{\haf}\sup_{X_0}|{\phi_{\al}(x\cdot a_T)}|=
T^{\haf}\sup_{X_0}|{\phi_{\al}(x)}|= \beta'' T^\haf\ ,
\end{eqnarray*}
since $\al$ is a {\it fixed} vector in a  {\it fixed} automorphic cuspidal  representation $(\nu, V)$, and the action does not change the supremum norm. \qed

\subsubsection*{Remark}   It is easy to see that the condition that forms $^\ell\phi$ and $^\ell\phi'$ are  of the same level is not essential for the proof, as well as that these are the Hecke forms. In particular, under our  normalization of the measure on $X_\ell$, we see that for a vector $v\in V$, $L^2$-norms $||{^\ell\phi_v}||_{X_\ell}=||{^\ell\phi_v}||_{X_{\ell'}}$ are equal  if $\ell|\ell'$ (here we view the function ${^\ell\phi_v}$ as both $\G_0(\ell N_0)$-invariant function and as
$\G_0(\ell' N_0)$-invariant function). Hence for two Maass forms  $\phi$ and $\phi'$ on $\G_0(N_0)$, we obtain the bound:
\begin{eqnarray}\label{mean-value1}
\sum_{ |\lm_i|\leq T}|\langle {^{\ell_1}}\phi\cdot{^{\ell_2}}\phi',\phi_i\rangle_{Y_{\ell' N_0}}|^2\cdot|\lm_i|^2e^{\frac{\pi}{2}|\lm_i|}\leq A\cdot T^2\ ,
\end{eqnarray} where the summation is over an orthonormal basis of Maass forms for the subgroup $\G_0(\ell' N_0)$ with $\ell'={\rm lcm}(\ell_1,\ell_2)$. We also have the analogous bound for the full spectral expansion \eqref{mean-value-Eisen} including the Eisenstein series contribution. The constant $A$ in \eqref{mean-value1} depends on $N_0$ and on eigenvalues of forms $\phi$ and $\phi'$.

\subsection{Holomorphic forms} The approach given above is applicable to holomorphic forms as well. In principle, there are no serious changes needed as compared to the Maass forms case. The main difficulty is that   we have to fill in the gap left in  \cite{BR4} concerning the model trilinear functional for the discrete series representations of $\PGLR$.

Let $\phi^k$, $\phi'^k$ be (primitive) holomorphic forms of weight $k$ for the subgroup $\G_0(N_0)$. We assume these are $L^2$-normalized. For $\ell>1$, we consider (old) forms ${^\ell\phi^k}(z)=\phi^k(\ell z)$ and ${^\ell\phi'^k=\phi'^k(\ell z)}$ on $\G_0(\ell N_0)$. Under our normalization of measures for $Y_{\ell N_0}$, we have  $||{^\ell\phi^k}||_{Y_{\ell N_0}}= ||{^\ell\phi'^k}||_{Y_{\ell N_0}}=\ell^{-\frac{k}{2}}$. This follows from the Rankin-Selberg method. Hence it would have been  more natural to consider normalized forms $\phi^k|_{[a_\ell]_k}=\ell^{\frac{k}{2}}\cdot{^\ell\phi^k}$.

For a (norm one) Maass form $\phi_i$ on $\G_0(\ell N_0)$, we define the corresponding triple coefficient by
\begin{eqnarray}\label{ci-ell-hol}
c^k_i(\ell)=\int_{Y_{\ell N_0}}{^\ell\phi^k}\ \overline{^\ell\phi'^k}\ \phi_i\ y^k\ dv_{Y_{\ell N_0}}\ .
\end{eqnarray} As with Maass forms, we renormalize these coefficients in accordance with the Watson formula by introducing normalized triple product coefficients
\begin{eqnarray}\label{ai-ell-hol}
|a^k_i(\ell)|^2=|\lm_i|^{2-2k}\exp\left(\frac{\pi}{2}|\lm_i|\right)|c^k_i(\ell)|^2\ .
\end{eqnarray}

\begin{thmb}{THM2}
There exists an effectively computable
constant $A$  such that the following bound holds for all $T \geq 1$,
\begin{eqnarray}\label{mean-value-hol}
\sum_{\haf T\leq |\lm_i|\leq T}|a^k_i(\ell)|^2\leq A\cdot {\ell^{-2k}}T^{2}\ ,
\end{eqnarray} where the summation is over an orthonormal basis of Maass forms for the group $\G_0(\ell N_0)$.
The constant $A$ depends on $N_0$, $\phi$ and $\phi'$, but not on $\ell$.
\end{thmb}

\subsubsection*{Remark} The proof we give applies to a slightly more general setup of forms  of different  level $\ell_1$ and $\ell_2$.   Namely, we have
\begin{eqnarray}\label{mean-value-hol-dif-level}
\sum_{ \haf T\leq |\lm_i|\leq T}|\langle {^{\ell_1}\phi^k}\ \overline{{^{\ell_2}}\phi'^k},\phi_i\rangle_{Y_{\ell' N_0}}|^2\cdot(\ell_1\ell_2)^{k}|\lm_i|^{2-2k}e^{\haf\pi|\lm_i|}\leq A\cdot T^2\ ,
\end{eqnarray} for two fixed forms $\phi^k$ and $\phi'^k$ on $\G_0(N_0)$ and for $\ell'={\rm lcm}(\ell_1,\ell_2)$. Breaking the interval $[1,T]$ into dyadic parts, we obtain for the full range (with the Eisenstein series contribution included),
\begin{eqnarray}\label{mean-value-hol-dif-level2}
\sum_{ |\lm_i|\leq T}|\langle {^{\ell_1}\phi^k}\ \overline{{^{\ell_2}}\phi'^k},\phi_i\rangle_{Y_{\ell' N_0}}|^2\cdot(\ell_1\ell_2)^{k}|\lm_i|^{2-2k}e^{\haf\pi|\lm_i|}+\\
\sum_{\{\kappa\}}\int_{-T}^{T}\langle {^{\ell_1}\phi^k}\ \overline{{^{\ell_2}}\phi'^k},E_\kappa(it)\rangle_{Y_{\ell' N_0}}|^2\cdot(\ell_1\ell_2)^{k}|t|^{2-2k}e^{\haf\pi|t|}dt\leq A\cdot T^2\ln(T)\ .\nonumber
\end{eqnarray}  This is slightly weaker than \eqref{mean-value} for Maass forms.

\subsection{Proof of Theorem B} As we seen in the case of Maass forms, the proof is based on the explicit form of the trilinear functional, its value on special vectors leading to the normalization \eqref{ai-ell-hol}, and the construction of test vectors for which we can estimate supremum norm effectively. We explain below changes and additions needed in order to carry out this scheme for discrete series.

\subsubsection{Discrete series} Let $k\geq 2$ be an even integer, and $(D_k,\pi_{D_k})$ be the corresponding discrete series representation of $\PGLR$. In particular, for $m\in 2\bz$, the space of  $K$-types of weight $m$ is non-zero (and in this case is one-dimensional) if and  only if $|m|\geq k$. This defines $\pi_k$ uniquely. Under the restriction to $\PSLR$, the representation $\pi_k$ splits into two representations $(D^\pm_k,\pi^\pm_{D_k})$ of ``holomorphic" and ``anti-holomorphic" discrete series  of $\PSLR$, and  the element $\dl={\rm diag}(1,-1)$ interchanges them.

We consider two realizations of discrete series as subrepresentations and as quotients of induced representations. Consider the space  $\CH_{k-2}$  of smooth even homogeneous functions on $\br^2\setminus 0$ of  homogeneous degree $k-2$ (i.e., $f(tx)=t^{k-2}f(x)$ for any $t\in \br^\times$ and $0\not=x\in\br^2$). We have the natural action of $\GLR$ given by ${\pi}_{k-2}(g)f(x)=f(g\inv x)\cdot \det(g)^{(k-2)/2}$, which is trivial on the center and hence defines a representation $(\CH_{k-2}, \pi_{k-2})$ of $\PGLR$. There exists a unique non-trivial invariant subspace $W_{k-2}\subset \CH_{k-2}$. The space $W_{k-2}$ is finite-dimensional, $\dim W_{k-2}=k-1$, and is generated by monomials $x_1^{m}x_2^{n}$, $m+n=k-2$. The quotient space $\CH_{k-2}/ W_{k-2}$ is isomorphic to the space of smooth vectors of the discrete series representation $\pi_k$.

We also consider the dual situation. Let  $\CH_{-k}$ be the space of smooth even homogeneous functions on $\br^2\setminus 0$ of  homogeneous degree $-k$. There is a natural $\PGLR$-invariant pairing $\langle\ ,\ \rangle: \CH_{k-2}\otimes \CH_{-k}\to\bc$ given by the integration over $S^1\subset \br^2\setminus 0$. Hence $\CH_{-k}$ is the smooth dual of $\CH_{k-2}$, and vice versa. There exists a unique  non-trivial invariant subspace $D^*_k\subset \CH_{-k}$. The quotient $\CH_{-k}/D^*_k$ is isomorphic to the finite-dimensional representation $W_{k-2}$.

Of course  $D^*_k$ is isomorphic to $D_k$, but we will distinguish between two realizations of the same abstract representation  as a subrepresentation $D^*_k\subset \CH_{-k}$ and as a quotient $\CH_{k-2}\to D_k$. We denote corresponding maps by $i_k: D^*_k\subset \CH_{-k}$ and $q_k: \CH_{k-2}\to D_k$.

\subsubsection{Trilinear invariant functionals}
Let $(V_{\lm,\eps},\pi_{\lm,\eps})$ be a unitary representation of the principal series of $\PGLR$. These are parameterized by $\lm\in i\br$ and by $\eps=0, 1$ describing the action of the element $\dl$ (see \cite{Bu}). The space $\Hom_G(D_k\otimes D^*_k, V_{\lm,\eps})$ is one-dimensional. We will work with the space of invariant trilinear functionals $\Hom_G(D_k\otimes D^*_k\otimes V_{-\lm,\eps},\bc)$ instead. We construct  below a non-zero functional  $l^{ind}_{k,\lm,\eps}\in \Hom_G(\CH_{k-2}\otimes\CH_{-k}\otimes V_{-\lm,\eps},\bc)$ for induced representations (in fact, this space is also one-dimensional) by means of (analytic continuation of) an explicit kernel.  We use it to define a non-zero functional   $l^{mod}_{k,\lm,\eps}\in\Hom_G(D_k\otimes D^*_k\otimes V_{-\lm,\eps},\bc)$. What is more important, we will use $ l^{ind}_{k,\tau,\eps}$ in order to carry out our computations in a  way similar to the principal series.

Let $ l^{ind}_{k,\lm,\eps}\in \Hom_G(\CH_{k-2}\otimes\CH_{-k}\otimes V_{-\lm,\eps},\bc)$ be a non-zero  invariant functional.  Such a functional induces the corresponding functional on $\CH_{k-2}\otimes D^*_k\otimes V_{-\lm,\eps}$ since $D^*_k\subset \CH_{-k}$. Moreover, any such functional vanishes on the subspace $W_{k-2}\otimes D^*_k\otimes V_{-\lm,\eps}$ since there are no non-zero maps between $W_{k-2}\otimes D^*_k$ and $V_{\lm,\eps}$. Hence we obtain a functional  $ l^{mod}_{k,\lm,\eps}\in \Hom_G(D_k\otimes D^*_k\otimes V_{-\lm,\eps},\bc)$ on the corresponding quotient space. We denote by $T^{mod}_{k,\lm,\eps}: D_k\otimes D^*_k\to V_{\lm,\eps}$ the associated  map, and  by $H^{mod}_{k,\lm,\eps}(u)=||T^{mod}_{k,\lm,\eps}(u)||_{V_{\lm,\eps}}^2$, $u\in D_k\otimes D^*_k$ the corresponding Hermitian form.

\subsubsection{Model  functionals} We follow the construction from \cite{BR4}. Denote \begin{eqnarray}\label{K-kernel2}
K_{k,\lm}(x,y,z)=|x-y|^{\frac{-1-\lm}{2}}|x-z|^{\frac{-1+\lm}{2}-k+1}|y-z|^{\frac{-1+\lm}{2}+k-1}\ .
\end{eqnarray}
In order to construct $ l^{mod}_{k,\lm,\eps}\in \Hom_G(\CH_{k-2}\otimes\CH_{-k}\otimes V_{-\lm,\eps},\bc)$, we consider the following function  in three variables $x,\ y,\ z\in \br$
\begin{eqnarray}\label{K-circle-discr}
K_{k-2,-k,\lm,\eps}(x,y,z)=(sgn(x,y,z))^{\eps}\cdot K_{k,\lm}(x,y,z)\ ,
\end{eqnarray}  where $sgn(x_1,x_2,z_3)=\prod_{i\not= j}sgn(x_i-x_j)$ (this is an $\SLR$-invariant function on $\br^3$ distinguishing two open orbits). An analogous  expression could be written in the circle  model on the space $C^\8(S^1)$.  Viewed as a kernel, $K_{k-2,-k,\lm,\eps}$ defines an invariant non-zero functional $l^{ind}_{k,\lm,\eps}$ on the (smooth part of) the representation $\CH_{k-2}\otimes\CH_{-k}\otimes V_{-\tau,\eps}\subset C^\8(\br^3)$.
 Such a kernel should be understood in the regularized sense (e.g.,  analytically  continued following  \cite{G1}). We are interested in $\lm\in i\br$, $|\lm|\to\8$, and hence all  exponents in \eqref{K-circle-discr} are non integral. This implies that the regularized kernel does not have a pole at relevant points.

 We denoted by $l^{mod}_{k,\lm,\eps}\in \Hom_G(D_k\otimes D^*_k\otimes V_{-\lm,\eps},\bc)$ the corresponding model functional. The difference with principal series clearly lies in the fact that we only can compute the auxiliary functional $l^{ind}_{k,\lm,\eps}$. However, for $k$ fixed, it turns out that necessary computations are essentially identical to the ones we performed for the principal series in \cite{BR4}.

\subsubsection{Value on $K$-types}  In order to obtain the normalization \eqref{ai-ell-hol}  and to compare our model functional $l^{mod}_{k,\lm,\eps}$ to the automorphic triple product \eqref{ci-ell-hol}, we have to compute, or at least to bound,  the value $l^{mod}_{k,\lm,\eps}(e_k\otimes e_{-k}\otimes e_0)$ where $e_{\pm k}\in D_k$ are  highest$/$lowest $K$-types of  norm one, and $e_0\in V_{\lm,\eps}$ is  a  $K$-fixed vector of norm one. For Maass forms, this is done in the Appendix of \cite{BR4} by explicitly calculating this value in terms of $\G$-functions.  In fact, the relevant calculation is valid for  $K$-fixed vectors for any three induced representations with generic values of parameters (i.e., those for which the final expression is well-defined).  Using the action of the Lie algebra of $G$ (see \cite{Lo} for the corresponding calculation where it is used to prove uniqueness), one obtains recurrence relations between values of  the model functional on various weight vectors.  For a generic value of $\tau$, this allows one to reduce the computation of $l^{mod}_{\tau,\lm,\eps}(e_k\otimes e_{-k}\otimes e_0)$ to the value  of $l^{mod}_{\tau,\lm,\eps}(e_0\otimes e_{0}\otimes e_0)$. By analytic continuation,  this relation holds for our set of parameters corresponding to the discrete series. From this, one deduces  the bound
\begin{eqnarray}\label{K-value}
|l^{mod}_{k,\lm,\eps}(e_k\otimes e_{-k}\otimes e_0)|^2\leq a |\lm_i|^{2k-2}\exp\left(-\frac{\pi}{2}|\lm_i|\right)\ ,
\end{eqnarray} for some explicit constant $a>0$ depending on $k$. In fact, this is the actual order of the magnitude for the above value, and one can obtain the exact value in terms of the Euler gamma function and that way derive the asymptotic for the value $l^{mod}_{k,\lm,\eps}(e_k\otimes e_{-k}\otimes e_0)$ as $|\lm|\to\8$.

\subsubsection*{Remark} There is a natural trilinear functional on Whittaker models of representations of $G$. This is the model which appears in the Rankin-Selberg method as a result of unfolding. The above computation (and the similar one for Maass forms performed in \cite{BR4}) shows that our normalization of the trilinear functional and the one coming from the Whittaker model coincide up to a constant of the absolute value one.

\subsubsection{Test vectors} Our construction is  very close to the construction we made in Section~\ref{consttest} for principal series representations, with appropriate modifications. We construct a test vector  $u_T(x,y)= v_T(x)\otimes v'_T(y)\in D_k\otimes D^*_k\subset \CH_{k-2}\otimes\CH_{-k}$ satisfying $H^{mod}_{k,\lm,\eps}(u_T)\geq \beta >0$ for $ T/2\leq |\lm|\leq T$ and  some constant $\beta>0$ independent of $\lm$. Vectors $v_T\in \CH_{k-2}$ and  $v'_T\in \CH_{-k}$ are first constructed  in the line model of induced representations, and then we relate these to vectors in the discrete series representation $D_k$.

 Choose a smooth non-negative function $\al\in C^\8(\br)$ with the support ${\rm supp}(\al)\subset [-\iota,\iota]$ for $\iota>0$ to be specified later, and normalized by the condition $\int_\br\al(x)dx=1$.  Consider the diagonal element $a_T={\rm diag}(T^{-\haf},T^\haf)\in G$. We define $w_T\in \CH_{k-2}$ by
\begin{eqnarray}\label{1-test-vect}
w_T=T^{\haf}\cdot\pi_{k-2}(a_T)\al\  .
                                        \end{eqnarray}
Recall that the action is given by $\pi_{k-2}({\rm diag}(a\inv,a)) v(x)=|a|^{2-k}v(a^2x)$.
We note that geometrically the vector $w_T$ is a small non-negative bump function around the point $0\in \br$, with the support in the interval $T\inv\cdot[-\iota,\iota]$, and satisfying $\int_{\br}w_T(x)dx=T^{\haf(1-k)}$. We now set $ v_T=q_k(w_T)\in D_k$, where  $q_k: \CH_{k-2}\to D_k$ is the projection. Note that $v_T=T^\haf\cdot \pi_{D_k}(a_T)\tilde v$ for some $\tilde v\in D_k$, i.e., the vector $v_T$ is obtained by the action of $G$ on some {\it fixed} vector in $D_k$. This will be crucial in what follows since we will need to estimate the supremum norm for the automorphic realization of the vector $v_T$.

The construction of the test vector  $v'_T\in D^*_k$ is slightly more complicated since we cannot simply project a vector to $D^*_k\subset \CH_{-k}$ because the value of the functional $l^{ind}_{k,\lm,\eps}$ might change significantly. Let $\al$ be as above. We now view it as a vector in $\CH_{-k}$.  We choose  a smooth real valued function $\al'\in C^\8(\br)$ satisfying the following properties:
\begin{enumerate}
  \item ${\rm supp}(\al')\subset [M-\iota,M+\iota]$, where the parameter $M$ is to be chosen later,
  \item $\int_\br x^m\al'(x)dx=-\int_\br x^m\al(x)dx\ {\rm for}\ 0\leq m\leq k-2\ .$
\end{enumerate} The last condition implies that the vector $w=\al+\al'\in D_k^*$ since $\int_\br x^m w(x)dx=0$ for $0\leq m\leq k-2$.
We define now the second test vector by
\begin{eqnarray}\label{2-test-vect}
  v'_T=T^{\haf}\cdot\pi_{-k}(\left(
                                          \begin{smallmatrix}
1 & 1 \\
                                   & 1 \\
\end{smallmatrix}
                                        \right)a_T) w\ .
                                        \end{eqnarray}
Clearly we have $v'_T\in D^*_k$.

Recall that $\pi_{-k}({\rm diag}(a\inv,a)) v(x)=|a|^{k}v(a^2x)$, and $\pi_{-k}(\left(
                                          \begin{smallmatrix}
1 & 1 \\
                                   & 1 \\
\end{smallmatrix}
                                        \right)) v(x)=v(x-1)$.
Hence geometrically the vector $v'_T=\al_T+\al_T'$ is the sum of two bump functions $\al_T$ and $\al_T'$ with their supports satisfying
${\rm supp}(\al_T)\subset 1+T\inv[-\iota, \iota]$ and ${\rm supp}(\al_T')\subset 1+T\inv[M-\iota, M+\iota]$, both near the point $1\in \br$. We also have $\al_T\geq 0$ and $\int_{\br}\al_T(x)dx=T^{\haf(k-1)}$.

We now set our test vector to $u_T=v_T\otimes v_T'\in D_k\otimes D^*_k$.  We want to show that for $ T/2\leq|\lm|\leq T$, \begin{eqnarray}\label{mod-lower-bnd-hol} |l^{mod}_{k,\lm,\eps}(u_T\otimes u)|\geq c'>0\end{eqnarray}  for some vector $u\in V_{\lm,\eps}$ with $||u||=1$, and with a constant $c'>0$ independent of $\lm$.

As we explained before, $l^{mod}_{k,\lm,\eps}(q_k(v)\otimes w\otimes u)= l^{ind}_{k,\lm,\eps}(v\otimes w\otimes u)$ for any triple $v\otimes w\otimes u\in \CH_{k-2}\otimes\CH_{-k}\otimes V_{\lm,\eps}$. Hence we work with $l^{ind}_{k,\lm,\eps}(w_T\otimes v'_T\otimes u)$ instead of $l^{mod}_{k,\lm,\eps}$ because  $l^{ind}_{k,\lm,\eps}$ is given by an explicit integral. Let $K_{k,\lm}(x,y,z)$ be as in \eqref{K-kernel2}.
We have then
\begin{eqnarray}\label{triple-value-hol-test}
l^{ind}_{k,\lm,\eps}(w_T\otimes v'_T\otimes u)=\int K_{k,\lm}(x,y,z)(sgn(x,y,z))^{\eps}\
 w_T(x)v'_T(y)u(z)\ dxdydz\ .
\end{eqnarray} Hence it is enough to show that the absolute value of the integral
\begin{eqnarray}\label{triple-value-hol-z}
I_\lm(z)=\int
 K_{k,\lm}(x,y,z) w_T(x)v'_T(y)\ dxdy=\langle K_{k,\lm}(x,y,z) , w_T(x)v'_T(y)\rangle \
\end{eqnarray} is not small for $z$ in some fixed interval of $\br$. 

We have
$I_\lm(z)=K(z)+K'(z)$,  where $K(z)=\langle K_{k,\lm}(x,y,z) , w_T(x)\al_T(y)\rangle$ and
$K'(z)=\langle K_{k,\lm}(x,y,z) , w_T(x)\al'_T(y)\rangle$. We will show now that for a certain range of $z$,  integrands in these integrals are not small, do not oscillate, and that integrals do  not cancel each other. This is implies the existence of a vector $u$ above (e.g., $L^2$-normalized  smooth characteristic function supported in the same range of $z$). 

The function $w_T(x)\al_T(y)$ is a non-negative function with the support in a small box of  size $\iota T\inv$ around the point $(0,1)\in\br^2$, and  the gradient of the function $K_{k,\lm}(x,y,z)$ is bounded by $|\lm|\leq T$ in this box. This implies that for a small enough (but fixed) $\iota>0$, the {\it argument} of $K_{k,\lm}(x,y,z)$ belongs to a small interval of $S^1$ for $x,y,z$ satisfying the above restrictions. It is also easy to see that there exists a  constant $c''>0$ such that $|K(z)|\geq c''>0$  for $z\in [10,20]$ and  $|\lm|\leq T$.    
Hence there are no significant cancellations in the integral   \eqref{triple-value-hol-z} for $z$ in the above range. We normalized our vectors so that $\int_{\br^2}v_T(x)\al_T(y)dxdy=1$, and hence the integral $K(z)$ is not small for $z$ not near singularities of the kernel $K_{k,\lm}(x,y,z)$ which is the case because of the restriction on $x$ and $y$.   This part is identical to our argument in Section 5.3.4 of \cite{BR4}.

We are left with the second term $K'(z)$. We want to show that there are no cancellations between two terms $K(z)$ and $K'(z)$ for  $z\in [10,20]$.  The reason for this is that while these terms are of about the same size, their arguments are different, and  not opposite if $\lm$ is not too small (e.g.,  $T/2\leq |\lm|\leq T$). Namely,  the argument of the kernel function $K_{k,\lm}$ in \eqref{triple-value-hol-z} on the support of $u_T$  is given by
\begin{eqnarray}\label{argument-K-lm}
\lm/2T\cdot \left[t_1(1-z\inv)-t_2(1+(z-1)\inv)\right]+\lm/2\left[\ln(z)+\ln(z-1)\right] +O(T\inv)\ ,
\end{eqnarray} where $t_1\in [-\iota,\iota]$, $t_2\in [1-\iota, 1+\iota]$ for $\al_T$, and $t_2\in [M-\iota, M+\iota]$ for $\al_T'$.  By choosing appropriate value of $M$, we can see that the {\it difference} of these arguments  is not close to $0$ and $\pi$ for any {\it fixed}  $z\in [10,20]$ and $T/2\leq|\lm|\leq T$. Hence integrals $K(z)$ and $K'(z)$ do not cancel each other since $u_T$ is real valued.

We have shown that $H_\lm(u_T)\geq c'>0$ for $T/2\leq |\lm|\leq T$, and some explicit $c'>0$ which is independent of $\lm$.

\subsubsection{Raising the level} We now discuss what happens when we change the level. Since our test vectors $v_T$ and $v_T'$ are {\it not} $K$-finite, we have to pass to automorphic functions on the space $X_\ell$. We use the standard notation  $j(g,z)=\det(g)^{-\haf} (cz+d)$ for
$g= \left(
                                          \begin{smallmatrix}
a & b \\
                               c    & d \\
\end{smallmatrix}
                                        \right)\in G^+$ and $z\in\uH$.
Let $\phi^k$ be a primitive holomorphic form of weight $k$ on $\uH$ for the subgroup $\G_0(N_0)$. We normalize $\phi^k$ by its norm on $Y_{N_0}$.  According to the well-known dictionary, we associate to $\phi^k$   the function $\phi_{e_k}\in C^\8(X_0)$ given by
\begin{eqnarray} \phi_{e_k}(g)=\phi^k(g(i))\cdot j(g,z)^{-k} \ , \end{eqnarray} where $z=g(i)$. In the opposite direction, we have
$\phi^k(g(i))=\phi_{e_k}(g)\cdot j(g,z)^{k}$.  We have the associated isometry $\nu_k=\nu_{\phi^k}:D_k\to C^\8(X_{N_0})$ which gives $\nu_k(e_k)=\phi_{e_k}$.

Let $\ell>1$ be an integer. We denote by $a_\ell={\rm diag}(
\ell^\haf ,\ell^{-\haf} )$.
For a given $\nu_k=\nu_{\phi^k}:D_k\to C^\8(X_{N_0})$, we construct the corresponding isometry ${\nu_k^\ell}:D_k\to C^\8(X_{\ell})$ as follows. For a vector $v\in D_k$, we consider the corresponding automorphic function
\begin{eqnarray}\label{nu-ell-nu-hol}
{^\ell\phi}_v(x)=\nu_k^\ell(v)(x)=\nu_k(v)\left(a_\ell x\right)=\phi_v\left(a_\ell x\right)\  .
\end{eqnarray} Obviously, we have $\sup_{X_\ell}|{^\ell\phi}_v|=\sup_{X_0}|{\phi}_v|$ for the same vector $v\in D_k$. We want to compare this to the classical normalization of old forms.
For the  lowest weight vector $e_k\in D_k$, we have with $z=g(i)$
 \begin{eqnarray}\label{nu-ell-nu-hol2}
{^\ell\phi}_{e_k}(g)=\phi_{e_k}
\left(a_\ell g\right)=\phi^k(\ell z)\cdot j(a_\ell g,i)^{-k}=\ell^{\frac{k}{2}} \phi^k(\ell z)\cdot j(g,i)^{-k}\ .
                                        \end{eqnarray}
On the other hand, classically, old forms are given by $ {^\ell\phi^k}( z)=\phi^k(\ell z)$. Hence we acquire the extra factor $ \ell^{\frac{k}{2}}$.

Since, as we noted, test  vectors $v_T\in D_k$ and $v'_T\in D^*_k$ are obtained by the (scaled) group action applied to  {\it fixed} vectors, and since the operation of raising the  level by $\ell$ does not change the supremum norm, we arrive at the following bound
\begin{eqnarray*}
\sup_{X_\ell}|{{^\ell\phi_{v_T}}(x)}|=\sup_{X_0}|{\phi_{v_T}(x)}|=
T^{\haf}\sup_{X_0}|{\phi_{\pi(a_T)w}(x)}|=
T^{\haf}\sup_{X_0}|{\phi_{w}(x)}|= \beta' T^\haf\ ,
\end{eqnarray*} for some constant $\beta'$. The same holds for the automorphic function ${^\ell\phi_{v'_T}}$. This implies that $||{^\ell\phi_{v_T}}{^\ell\phi_{v'_T}}||^2_{X_\ell}\leq \beta T^2$ for some $\beta>0$ independent of $\ell$ and $T$.

To summarize, we have proved the bound
\begin{eqnarray}\label{mean-value-hol}
\sum_{ \haf T\leq |\lm_i|\leq T}|\langle {^{\ell}}\phi^k\ \overline{^\ell\phi'^k},\phi_i\rangle_{Y_{\ell N_0}}|^2\cdot\ell^{2k}|\lm_i|^{2-2k}e^{\haf\pi|\lm_i|}\leq A\cdot T^2\ ,
\end{eqnarray} where the summation is over an orthonormal basis of Maass forms for the subgroup $\G_0(\ell N_0)$. The constant $A$ in \eqref{mean-value-hol} depends on $N_0$ and on the weight $k$ of forms $\phi$ and $\phi'$. 

The above argument also proves the case of forms with different level, i.e., the bound \eqref{mean-value-hol-dif-level}.  Under  the normalization of measures on $X_\ell$, we see that for a vector $v\in V$, $L^2$-norms $||{^\ell\phi_v}||_{X_\ell}=||{^\ell\phi_v}||_{X_{\ell'}}$ are equal  if $\ell|\ell'$ (here we view the function ${^\ell\phi_v}$ as both $\G_0(\ell N_0)$-invariant function and as $\G_0(\ell' N_0)$-invariant function). Obviously, the supremum norms of $^\ell\phi_v$ on $X_\ell$ and on $X_{\ell'}$ are also coincide. Hence for integers  $\ell_1$ and $\ell_2$, and  $\ell'={\rm lcm}(\ell_1,\ell_2)$, we have $$||{^{\ell_1}\phi_{v_T}}{^{\ell_2}\phi_{v'_T}}||^2_{X_{\ell'}}\leq  \haf||{^{\ell_1}\phi_{v_T}}^2||^2_{X_{\ell'}}+
\haf||{^{\ell_2}\phi_{v'_T}}||^2_{X_{\ell'}}\leq  \beta T^2\ .$$
This implies \eqref{mean-value-hol-dif-level}.
\qed


{\bf Acknowledgments.}  It is a pleasure to thank Joseph Bernstein for endless discussions concerning automorphic functions. In particular, results presented here are part of a joint project.  I would like to thank Jeff Hoffstein for asking the question during the  Oberwolfach workshop ``The Analytic Theory of Automorphic Forms" in August 2012 which led me to complete the long overdue results from \cite{BR4}. It is also a pleasure to thank  organizers of that workshop for their kind invitation.

\bibliography{Mergedbibfile}	
\bibliographystyle{amsplain}	
\end{document}